\documentclass[reqno,11pt]{amsart}
\usepackage{graphicx, amsmath, amssymb, amscd, mathtools, hyperref, amsthm, euscript, 
amsfonts,bm,color}  
\DeclareMathAlphabet{\mathpzc}{OT1}{pzc}{m}{it}

\newtheorem{thm}[equation]{Theorem}

\newtheorem{Ex}[equation]{Example}

\newtheorem{rmk}[equation]{Remark}\newtheorem{Rmk}[equation]{Remark}\newtheorem{remark}[equation]{Remark}

\newtheorem{prop}[equation]{Proposition}

\newtheorem{cor}[equation]{Corollary}
\newtheorem{lem}[equation]{Lemma}\newtheorem{Lem}[equation]{Lemma}
\newtheorem{lemma}[equation]{Lemma}
\newtheorem{Def}[equation]{Definition}

\numberwithin{equation}{section}

\numberwithin{equation}{section}

\newcommand{\be}{begin{equation}}

\newcommand{\bH}{\mathbb H}

\newcommand{\e}{\epsilon}
\newcommand{\z}{\mathbb{Z}}

\newcommand{\br}{\mathbb{R}}

\newcommand{{\grinv}}{{\Cal G}^{-r}}

\newcommand{\ba}{\backslash}

\newcommand{\G}{\Gamma}

\newcommand{\Cal}{\mathcal}

\newcommand{\la}{\langle}
\newcommand{\ra}{\rangle}

\newcommand{\SL}{\operatorname{SL}}

\newcommand{\bp}{\begin{pmatrix}}
\newcommand{\ep}{\end{pmatrix}}
\renewcommand{\be}{\begin{equation}}
\newcommand{\ee}{\end{equation}}

\renewcommand{\bp}{{\rm bp}}

\newcommand{\PGL}{\operatorname{PGL}}

\renewcommand{\L}{\Cal L}

\newcommand{\PSL}{\op{PSL}}

\newcommand{\PS}{\op{PS}}\newcommand{\Haar}{\op{Haar}}

\newcommand{\norm}[1]{\lVert #1 \rVert}

\newcommand{\op}{\operatorname}\newcommand{\supp}{\operatorname{supp}}

\newcommand{\BR}{\operatorname{BR}}
\newcommand{\BMS}{\operatorname{BMS}}

\renewcommand{\setminus}{-}

\newcommand{\Z}{\z}\newcommand{\Om}{\Omega}

\newcommand{\ga}{\gamma}
\newcommand{\F}{\mathcal F}
\newcommand{\La}{\Lambda}
\renewcommand{\i}{\op{i}}
\newcommand{\FF}{\F^{(2)}}

\def\scrB{{\mathcal B}}
\def\scrC{{\mathcal C}}

\def\scrF{{\mathcal F}}

\def\scrL{{\mathcal L}}

\def\scrO{{\mathcal O}}

\def\e{\mathrm{e}}
\def\i{\mathrm{i}}

\def\dim{\operatorname{dim}}
\def\dist{\operatorname{dist}}

\def\PGL{\operatorname{PGL}}

\def\SL{\operatorname{SL}}

\def\PSL{\operatorname{PSL}}

\def\supp{\operatorname{supp}}

\newcommand{\GaG}{\Gamma\backslash G}

\newcommand{\sfrac}[2]{{\textstyle \frac {#1}{#2}}}

\newcommand{\fg}{\mathfrak{g}}
\newcommand{\fp}{\mathfrak{p}}
\newcommand{\fa}{\mathfrak{a}}
\newcommand{\fb}{\mathfrak{b}}

\newcommand{\fk}{\mathfrak{k}}

\newcommand{\fn}{\mathfrak{n}}

\newcommand{\ad}{\mathrm{ad}}

\newcommand{\RR}{\br}\newcommand{\Ga}{\Gamma}\newcommand{\bb}{\mathbb}
\newcommand{\cal}{\mathcal}\newcommand{\NN}{\mathbb N}
\renewcommand{\e}{\varepsilon}
\renewcommand{\epsilon}{\e}

\newcommand{\dg}{D_\Gamma^{\star}}
\newcommand{\inte}{\op{int}}
\newcommand{\fc}{\mathcal C}
\newcommand{\vo}{\mathsf v_0}

\begin{document}

\title[Anosov groups]{Anosov groups: local mixing, counting, and equidistribution.}

\author{Sam Edwards, Minju Lee, and Hee Oh}
\address{Mathematics department, Yale university, New Haven, CT 06520 }
\begin{abstract} Let $G$ be a connected semisimple real algebraic group, and $\Gamma<G$ be
 a  Zariski dense Anosov  subgroup with respect to a minimal parabolic subgroup.
We describe the asymptotic behavior of matrix coefficients $\la (\exp tv).  f_1, f_2\ra$ in $L^2(\Gamma\ba G)$ as $t\to \infty$
for any $f_1, f_2\in C_c(\Gamma\ba G)$ and any vector $v$ in the interior of the limit cone of $\Gamma$. 
These asymptotics involve higher rank analogues of Burger-Roblin measures which are introduced in this paper.
As an application, for any affine symmetric subgroup $H$ of $G$, we obtain a bisector counting result for $\Gamma$-orbits with respect to the corresponding generalized Cartan decomposition of $G$. Moreover, we obtain analogues of the results of Duke-Rudnick-Sarnak and Eskin-McMullen for counting discrete $\Gamma$-orbits in affine symmetric spaces $H\ba G$. 
\end{abstract}

\email{samuel.edwards@yale.edu}
\email{minju.lee@yale.edu}
\email{hee.oh@yale.edu}
\thanks{Edwards was supported by postdoctoral scholarship 2017.0391 from the Knut and Alice Wallenberg Foundation and Oh was supported in part by NSF grants}

\maketitle

\tableofcontents
\section{Introduction}
 Let $G$ be a connected semisimple real algebraic group. We fix a Cartan decomposition $G=K (\exp \fa^+) K$, where $K$ is a maximal compact subgroup and $\exp \fa^+$ is a positive Weyl chamber of a maximal real split torus of $G$. Let $\Gamma<G$ be a Zariski dense discrete subgroup. Consider a matrix coefficient of $L^2(\GaG)$ given by
 \be\label{hm} \la \exp(tu) f_1, f_2\ra= \int_{\Gamma\ba G} f_1(x\exp (tu) ) f_2(x)\, dx ,\ee
 where $u\in \fa^+-\{0\}$ and $dx$ denotes the $G$-invariant measure on $\Gamma\ba G$. Understanding its asymptotic behavior as $t\to \infty$ is of basic importance in the study of dynamics of flows in $\Gamma\ba G$, and has many applications, including to equidistribution and counting problems. 
A classical result  due to Howe-Moore \cite{HM} implies that
 \be\label{hm2} \lim_{t\to \infty} \la \exp(tu) f_1, f_2\ra=\frac{1}{\op{Vol}(\Gamma\ba G) }\int f_1 \,dx \int f_2\, dx .\ee
In particular, if $\Gamma$ has infinite co-volume in $G$, then 
\begin{equation}\label{e11}\lim_{t\to \infty}  \la \exp(tu) f_1, f_2\ra = 0.
\end{equation} 

This leads us to the following local mixing type question: for a given unit vector $u\in  \fa^+$, do there exist a normalizing function $\Psi_{\Gamma,u} :(0,\infty) \to (0,\infty)$ and locally finite Borel measures $\mu_{u}, \mu^*_u$ on $\Gamma\ba G$ such that for any $f_1, f_2\in C_c(\Gamma \ba G)$\footnote{for a topological space $X$, the notation $C_c(X)$ means the space of all continuous functions on $X$ with compact support, and
if $M$ is a compact group acting continuously on $X$, the notation $C_c(X)^M$  means  the subspace of $C_c(X)$ consisting of $M$-invariant functions.}
 \be\label{e2} \lim_{t\to \infty } \Psi_{\Gamma,u} (t)  \la \exp(tu) f_1, f_2\ra = \mu_u(f_1)\mu_u^*(f_2)\, ?\ee

 When $G$ has rank one, this was completely answered by Roblin and Winter (\cite{Ro}, \cite{Wi})  for geometrically finite subgroups and by Oh and Pan \cite{OP} for co-abelian subgroups of convex cocompact subgroups.
 
When $G$ has rank at least two, the location of the vector $u$ relative to the \emph{limit cone} of $\Gamma$ turns out to play an important role.
The limit cone of $\Gamma$, which we denote by $\L_\Ga $, is defined as the smallest closed cone in $\fa^+$ containing the Jordan projection
of $\Gamma$. Benoist showed that $\L_\Ga$ is  convex and has non-empty interior \cite{Ben}.
 Indeed, it is not hard to show that if $u\notin \L_\Gamma$, then for any $f_1, f_2\in C_c(\Gamma \ba G)$,
$$\ \la \exp(tu) f_1, f_2\ra =0\quad\text{ for all $t$ large enough; }$$ see Proposition \ref{vanish}.

The main goal of this paper is to prove the local mixing result,
giving  a positive answer to Question  \eqref{e2} for all directions $u$ in the interior of $\L_\Ga$ and describe 
applications to counting and equidistribution results associated to a symmetric subgroup $H$ of $G$ for a large class of discrete subgroups, called Anosov subgroups. Anosov subgroups of $G$ are defined with respect to any parabolic subgroup. In this paper, we focus on Zariski dense Anosov groups  with respect to a minimal parabolic subgroup $P$ of $G$.
Let $\F:=G/P$ the Furstenberg boundary.
We denote by $\F^{(2)}$ the unique open $G$-orbit in $\F\times \F$ via the diagonal action.
A Zariski dense discrete subgroup $\Gamma<G$ is called {\it Anosov} with respect to $P$ 
if there exists a finitely generated word hyperbolic group $\Sigma$ such that $\Gamma=\Phi(\Sigma)$ where
 $\Phi: \Sigma \to G$ is a  $P$-Anosov representation, i.e., $\Phi$ induces a continuous equivariant map $\zeta$ from the Gromov boundary $\partial \Sigma$ to $\F$ such that 
for all $x\ne  y\in \partial \Sigma$, $(\zeta(x),\zeta(y))$ belongs to $\F^{(2)}$.  The definition of a $P$-Anosov representation for a general discrete subgroup requires a certain contraction property, which is automatic for Zariski dense subgroups (see \cite[Theorem 1.5]{GW}).  
The notion of Anosov representations was first introduced by Labourie for surface groups \cite{La}, and then extended by Guichard and Wienhard \cite{GW} to general word hyperbolic groups.

If $G$ has rank one, the class of Zariski dense Anosov groups coincides with the class of Zariski dense convex cocompact subgroups of $G$ \cite[Theorem 5.15]{GW}.
Guichard and Wienhard \cite[Theorem 1.2]{GW} showed that $P$-Anosov representations form an open subset of the space $\op{Hom}(\Sigma, G)$; this abundance of Anosov subgroups contrasts
with the fact that there are only countably many lattices in any simple algebbraic group not locally isomorphic to $\op{PSL}_2(\br)$.
The class of Anosov subgroups includes  subgroups of a real-split simple algebraic group which arise as the Zariski dense image of a Hitchin representation \cite{Hi} of a surface subgroup studied by Labourie and Fock-Goncharov (\cite{La}, \cite{FG}), as well as
 Schottky groups (cf. Lemma \ref{Scc}). We refer to (\cite{KLP}, \cite{GG}, \cite{BPS}, etc.)  for other equivalent definitions of Anosov subgroups, and
to (\cite{Ka}, \cite{Win}) for excellent survey articles.
 
 \medskip
 
In the whole paper, by an Anosov subgroup, we mean a Zariski dense Anosov subgroup with respect to a minimal parabolic subgroup. In the rest of the introduction, we let $\Gamma<G$ be an Anosov  subgroup.

We recall the definition of the growth indicator function $\psi_{\Gamma}\,:\,\fa^+ \rightarrow \br \cup\lbrace- \infty\rbrace$ given by Quint  \cite{Quint1}: 
for any vector $u\in \fa^+$,
\begin{equation}\label{grow}
\psi_{\Gamma}(u):=\|u\| \inf_{\underset{u \in\scrC}{\mathrm{open\;cones\;}\scrC\subset \fa^+}}
\tau_{\cal C}
\end{equation}
where $\tau_{\cal C}$ is the abscissa of convergence of the series $\sum_{\ga\in\Ga,\,\mu(\ga)\in\cal C}e^{-t\norm{\mu(\ga)}}$.
Here $\mu:G\to \mathfrak{a}^+$ is the Cartan projection and $\|\cdot \|$ is the norm on $\mathfrak{a}$ induced from a left invariant Riemannian metric on $G/K$.
 Observe that in the rank one case, $\psi_\Gamma$ is simply the critical exponent of $\Gamma$. Quint \cite{Quint1} showed that $\psi_\Gamma $ is a concave and upper semi-continuous function which is positive on $\op{int}\L_\Ga$; here $\inte \L_\Ga$ denotes the interior of $\L_\Ga$. 
If $2\rho\in\fa^*$ denotes the sum of all positive roots with respect to the choice of $\fa^+$, then $\psi_\Gamma \le 2\rho$.
When $\Gamma$ is a lattice,  it follows from \cite{GO} that $\psi_\Gamma=2\rho$. On the other hand,
when $\Gamma$ is of infinite co-volume in a simple Lie group of rank at least $2$, 
Quint deduced from \cite{Oh} that $\psi_\Gamma \le 2(\rho-\eta_G) $,
where $2\eta_G$ is the sum of the maximal strongly orthogonal subset of the root system of $G$  \cite{Q3}.

\medskip

\noindent{\bf Local mixing.} 
Let $N^+$ and $N^-$ denote the maximal expanding and contracting horospherical subgroups, respectively, associated with $\fa^+$ (see \eqref{defN1}, \eqref{defN2}), and $M$  the centralizer of $\exp \fa $ in $K$. 
For each $u\in \op{int} \L_\Gamma$, Quint \cite{Quint2} constructed a higher-rank analogue of the Patterson-Sullivan density supported on the limit set $\Lambda_\Ga$, which is the minimal $\Gamma$-invariant subset of $\F$.  Using this, we define the $N^{\pm} M$-invariant Burger-Roblin measures $ m^{\BR}_{u} $ and $ m^{\BR_*}_{u}$, respectively, on $\Gamma\ba G$ (see  \eqref{def.BR} and \eqref{dualb}), which can be considered as the higher rank generalizations of the Burger-Roblin measures in the rank one case (\cite{Bu}, \cite{Ro}, \cite{OhShah}). We denote by $\i$ the opposition involution of $\fa$ (Definition \ref{op}), and set  $r:=\op{rank }( G)=\dim\fa \ge 1$.

  \begin{thm} \label{m1}
 For any $u\in \op{int}\L_\Gamma$, there exists $\kappa_u>0$ such that for all $f_1, f_2\in C_c(\Gamma\ba G)^M$,
$$\lim_{t\to +\infty} t^{(r-1)/2} e^{t(2\rho-\psi_\Ga)(u)}  \int_{\Gamma\ba G} f_1(x \exp (tu)) f_2(x) \,dx 
= \kappa_u \cdot m^{\BR}_{\i(u)} (f_1) \,m^{\BR_*}_{u} (f_2).$$
\end{thm}
We mention that this theorem is not expected to hold for $u\in \partial \L_\Gamma$ in view of
\cite[Theorem 1.1]{DG}.
See Theorem \ref{m11} for a more refined version of this theorem; in fact it is this refined version which is needed in the application to counting problems as  stated in Theorems \ref{m3} and \ref{thm3}.

\medskip

\noindent{\bf Equidistribution of maximal horospheres.}  We also obtain the following equidistribution result for translates of maximal horospheres:

\begin{thm}\label{m2}  For any $u\in \op{int} \L_\Gamma$,  $f\in C_c(\Gamma \ba G)^M$, $\phi\in C_c(N^+)$, and $x=[g]\in \Gamma\ba G$, we have
\begin{equation*}
\lim_{t\rightarrow\infty} t^{(r-1)/2}e^{t (2\rho-\psi_\Ga)(u)}\int_{N^+}f(x n \exp (tu))\phi(n)\,dn= \kappa_u \cdot  m^{\mathrm{BR}}_{\i( u)}(f)\, 
\mu_{ gN^+, u}^{\mathrm{PS}}(\phi),
\end{equation*}
where $dn$ and $\mu_{ gN^+, u}^{\PS}$ are respectively the Lebesgue and Patterson-Sullivan measures on $gN^+$ as defined in
  \eqref{eq.Leb2} and \eqref{eq.PSN}.
\end{thm}

\noindent{\bf Bisector counting for a generalized Cartan decomposition.}
Let $H$ be a symmetric subgroup of $G$, i.e.\ $H$ is the identity component of the set of fixed points for an involution $\sigma$ of $G$. 
Up to a conjugation, we may assume that $\sigma$ commutes with the Cartan involution $\theta$ which fixes $K$. We then have a generalized Cartan decomposition $G=H \cal W  (\exp \fb^+) K$, where $\fb^+\subset \fa^+$ and $\cal W$ is a subgroup of the Weyl group (see Section \ref{sec.aff} for details).  Set $r_0:=\op{rank} H\ba G = \dim\fb$. Note that $1\le r_0\le r$.

\begin{thm} \label{m3}  For any $v\in \fb^+\cap \inte \L_\G$,
there exist $ c>0$ and  a norm $|\cdot |$ on $\fb$  such that for any right $H\cap M$-invariant bounded subset $\Omega_H\subset H$ with $\mu_{H,v}^{\PS}(\partial \Om_H) =0$ and any left $M$-invariant bounded subset $\Omega_K\subset K$ with  
$\mu^{\PS, *}_{K,\i(v)}(\partial \Om_K^{-1})=0$,
we have
   $$\lim_{T\to \infty}  \frac{ \# (\Gamma\cap \Omega_H \cal (\exp \fb^+_T )\, \Omega_K ) }{e^{\psi_\Ga(v) T}\cdot  T^{(r_0-r)/2}} =c\;  \mu_{H,v}^{\PS}(\Om_H) \,\mu^{\PS,*}_{K,\i(v)}(\Om_K^{-1}), $$
 where  $\fb^+_T= \{ w \in \fb^+ :|w| \le T\} $ and $\mu_{H,v}^{\PS}$ and $\mu^{\PS,*}_{K, \i(v)}$ are measures on $H$ and $K$ defined in 
Definition \ref{Hskinning} and Lemma \ref{km2} respectively.   \end{thm} 

When $\Gamma$ is a cocompact lattice in a rank one Lie group and  $H$ is compact, this goes back to Margulis' thesis from 1970
(see \cite{Ma} for an English translation published in 2004). In the case when $\Gamma$ is a geometrically finite subgroup of a rank one Lie group, this was shown in \cite{Ro} for $H$ compact, and in  \cite{OhShah} and \cite{MO} for general symmetric subgroups.
For lattices of higher rank Lie groups, the bisector counting problem was studied in \cite{GO}, \cite{GOS}, and \cite{GOS2}. 
For non-lattices of higher rank Lie groups, it was studied for $H$ compact, by Quint \cite{Q5} and Thirion
 \cite{Thirion} for Schottky groups\footnote{Thirion's work applies to the so-called Ping-Pong groups  which are slightly more general than Schottky groups.} and by Sambarino for Anosov subgroups \cite{Samb1} (see also \cite{Ca2}). Hence the main novelty of this paper lies in our treatment of non-compact symmetric subgroups $H$ in a general higher rank case. 
It is interesting to note the presence of the decaying polynomial term $T^{(r_0-r)/2}$ when $\fa\ne \fb$, as the results
in loc. cit. have all purely exponential terms.
We mention that a related counting result was obtained for $\op{SO}(p,q-1)\ba \op{SO}(p,q)$ in a recent paper of Carvajales \cite{Ca};
in this case, $\fb$ lies in the wall of $\fa$ and hence Theorem \ref{m3} does not apply, and
the asymptotic is again purely exponential.

By the concavity and upper semi-continuity of $\psi_\Gamma$, there exists  a unique unit vector $u_\Gamma\in \fa^+$ (called the maximal growth direction) such that 
$$\psi_\Gamma (u_\Gamma)=\max_{v\in \fa^+, \,\|v\|=1}\psi_\Gamma (v).$$   
It is known that $u_\Gamma\in \inte \L_\Gamma$ (\cite{Q5}, \cite{Samb1}). When $u_\Gamma\in \fb^+$, the norm $|\cdot|$ in Theorem \ref{m3} associated to $u_\Gamma$ may simply be taken as the Euclidean norm $\|\cdot \|$ as above, i.e. the one obtained from the inner product $\langle\cdot,\cdot\rangle$ on $\fa$ induced by the Killing form. For a general vector $v\in\fb^+\cap\inte\scrL_\Ga$, one may take any norm that arises from an inner product for which $v$ and $(\nabla\psi_\Ga(v))^{\perp}=\lbrace w\in \fb\,:\langle \nabla\psi_\Ga(v),w\rangle=0\rbrace $ are orthogonal.

\medskip
\noindent{\bf Example.} When $\fa=\fb$, we automatically have $u_{\Ga}\in\fb^+\cap \inte \L_\Gamma$; so Theorem \ref{m3} applies. For groups $G$ of rank one, this is always the case for any symmetric subgroup $H$. In general, this case arises as follows: let $\iota$ be any involution of $G$ that commutes with the Cartan involution $\theta$ and fixes $\fa$ pointwise. Then defining $\sigma :=\iota\circ\theta$, we have $\sigma|_{\fa}=-1$, and hence $\fa=\fb$. For example, for any element $m\in K$ of order two which commutes with $\exp \fa$, $\iota(g):= mgm$ satisfies the above conditions. More specifically, the pair $G=\mathrm{PGL}_n(\RR)$ and $H=\op{PO}(p,n-p)$ may be realized this way by taking $m=\op{diag}( \op{Id}_p,-\op{Id}_{n-p})$.

\medskip 
\noindent{\bf Counting in affine symmetric spaces.} Around 1993, Duke-Rudnick-Sarnak \cite{DRS} and Eskin-McMullen \cite{EM} showed the following (see also \cite{BOh}, \cite{GO}, \cite{GOS}, \cite{GOS2} etc.):
 \begin{thm} [Duke-Rudnick-Sarnak, Eskin-McMullen] \label{m0} Let $\G<G$ be a lattice such that $\mathsf v_0\Gamma\subset H\ba G$ is closed  for $\vo=[H]$. Suppose  that
 $(H\cap \Gamma)\ba H$ has finite volume. We have, as $T\to \infty$,
   $$\lim_{T\to \infty}\frac{\# (\mathsf v_0\Gamma\cap \vo (\exp \fb_T^+) K ) }{  \mathsf m \big(\vo (\exp \fb_T^+) K\big) }= \frac{\op{Vol}\big((\Gamma\cap H)\ba H\big)}{\op{Vol}(\Gamma\ba G)} 
 ,$$
 where   $\fb^+_T= \{ w \in \fb^+ : \| w\| \le T\} $ and $\mathsf m$ is a suitably normalized $G$-invariant measure on $H\ba G$.
 \end{thm} 
 
In order to state an analogue of Theorem \ref{m0} for $\Gamma$ Anosov, we introduce the following condition on an $H$-orbit: a closed orbit $[e]H\subset \Gamma\ba G$ is said to be
{\it uniformly proper} if there exists a right $K$-invariant neighborhood $\cal O$ of $[e]$ in $\G\ba G$ such that
$$\{[h]\in (\Ga \cap H)\ba H: [h] \exp (\fb^+\cap \L_\Ga) \cap \cal O\ne \emptyset\}$$ is bounded.\footnote{As $[e]H$ is closed, the inclusion $(H\cap \Gamma)\ba H\to \Gamma\ba G$ is a proper map and hence
the set $S_b:=\{[h]\in (\Gamma\cap H)\ba  H: [h] \exp b \cap \cal O\ne \emptyset\}$ is bounded for each $b\in \fb^+\cap \L_\Ga$.
Now the uniform properness of $[e]H$ means that the union $\bigcup_{b\in \fb^+\cap \L_\Ga} S_b$ is also bounded.}
 We show in Lemma \ref{bd} that $[e]H$ is uniformly proper whenever $r=r_0$ and the limit set $\Lambda_\Ga$ is contained in the open
set $ HP/P\subset \cal F$ (cf.\ \cite[Proposition 7.1.8]{Sch}).

 When $\vo \Gamma$ is closed, the measure $\mu_{H,v}^{\PS}$ in Theorem \ref{m3} induces a locally finite Borel measure on $(\Gamma\cap H)\ba  H$, whose total measure will be called the skinning constant 
$\op{sk}_{\Gamma, v} (H)$ of $H$ with respect to $\Gamma$ and $v$;
this constant is positive if and only if $\Lambda_\Ga \cap HP/P   \ne \emptyset$ (see \eqref{hpp}).
\begin{thm}\label{thm3} Let $\Gamma<G$ be an Anosov subgroup such that $\mathsf v_0\Gamma\subset H\ba G$ is closed for $\vo=[H]$.
Suppose that $u_\Ga\in\mathfrak b^+$ and $[e]H$ is uniformly proper. 
Then $\op{sk}_{\Gamma,u_\Ga} (H)<\infty$ and there exists a constant $c>0$ such that
\begin{equation*}
\lim_{T\to \infty}\frac{ \# (\vo \Gamma\cap \vo  (\exp \fb^+_T)  K ) }{ e^{\delta_\Ga T} \cdot T^{(r_0-r)/2}   }=c\; \op{sk}_{\Gamma,u_\Ga} (H)     
\end{equation*}
 where $\fb^+_T= \{ w \in \fb^+ :\|w\| \le T\} $ and $\delta_\Ga=\psi_\Gamma(u_\Ga)$.
\end{thm}
See Theorem \ref{thm33} for a more refined version where $u_\Gamma$ is replaced by a more general $v\in \fb^+\cap \inte \L_\Gamma$.

\begin{rmk}
\rm 
\begin{enumerate}

\item When $G$ has rank one, this is proved in \cite{OhShah} and \cite{MO} for any geometrically finite group $\Gamma$
under the finite skinning constant hypothesis, based on the fact that the Bowen-Margulis-Sullivan measure is finite.
In the higher rank case, the corresponding Bowen-Margulis-Sullivan measure is infinite (\cite[Proposition 3.5]{Samb1}, see also \cite[Corollary 4.9]{LO}) and for this reason, the finite skinning constant hypothesis
seems insufficient for our approach to work.

\item We mention a recent work of Carvajales \cite[Theorem B]{Ca2} for a counting result on the orbit $\mathsf{v}_0 \Gamma$ where
$G=\PSL_n(\br)$ and $H=\op{SO}(p,n-p)$ under the assumption that
 $\Gamma\cap H$ is finite.  In this case, 
  $\fa=\fb$ and the uniform properness of $[e]H$ easily follows, and
 hence Theorem \ref{thm3} applies. 
\end{enumerate}
\end{rmk}


\noindent{\bf On the proofs.}
The  following mixing result for the Bowen-Margulis-Sullivan measures was proved by Thirion \cite{Thirion} for Schottky groups  and by Sambarino \cite{Samb1} for  Anosov groups which arise from representations of the fundamental group of a closed negatively curved Riemannian manifold,  using thermodynamic formalism.
The general case was recently proved by Chow and Sarkar \cite{CS}, based on the fundamental work of Bridgeman, Canary, Labourie and Sambarino \cite{BCLS}:
 \begin{thm}  \label{mixing_bms}
For any $u\in \op{int} \L_\Gamma$ and  any $f_1,\,f_2\in C_c(\GaG)^M$,
\begin{equation*}
\lim_{t\rightarrow\infty} t^{(r-1)/2}\int_{\GaG} f_1(x) f_2(x \exp(tu) )\,dm_{{u}}^{\mathrm{BMS}}(x)=\kappa_{u}  \cdot m_{{u}}^{\mathrm{BMS}}(f_1)\,m_{ {u}}^{\mathrm{BMS}}(f_2),
\end{equation*}
where  $m^{\BMS}_{u}$ is the BMS-measure associated to $u$ (see \eqref{eq.BMS0} and \eqref{def.BMS2}).
\end{thm}
This theorem is known only for Zariski dense Anosov subgroups with respect to a minimal parabolic subgroup
and this is the main reason why we restrict our attention to this class of groups in this paper.
\begin{rmk} \rm In fact, Chow and Sarkar proved in \cite{CS} that a version of this theorem holds for general $f_1, f_2\in C_c(\GaG)$ which are not necessarily $M$-invariant, provided the right hand side is replaced by the sum $\kappa_u\sum_{Y} m_{{u}}^{\mathrm{BMS}}|_Y(f_1)\,m_{ {u}}^{\mathrm{BMS}}|_Y(f_2),$
where the sum is taken over all $P^\circ$-minimal subsets $Y$ of $\Ga\ba G$. Given this result, the $M$-invariance condition in Theorems \ref{m1}, \ref{m2} and \ref{m3}  should not be necessary with an appropriate modification of the main terms.
\end{rmk}

Using the product structures of the Haar measure $dx$ and $dm^{\BMS}_{ u}(x)$, one can deduce mixing for one measure from that of the other via the study of transversal intersections. This observation is originally due to Roblin \cite{Ro} in the case of the unit tangent bundle of a rank one locally symmetric manifold, and has been extended and utilized in (\cite{MO}, \cite{OhShah}) to the frame bundle.
 This study leads us to generalize the definition of the family of Burger-Roblin measures $m^{\BR}_{u}$ and $m^{\BR_*}_{u}$ for $u\in \op{int} \L_\Gamma$, which turn out to control the asymptotic behavior of matrix coefficients as in Theorem \ref{m1} in a quite similar way as in the rank one case.

As in \cite{DRS} and \cite{EM} (also as in \cite{OhShah} and \cite{MO}), passing from Theorem \ref{m1} to Theorems \ref{m3} and \ref{thm3} requires the following equidistribution statement for translates of $H$-orbits (more precisely,
we need Proposition \ref{eq}). The idea of using mixing in the equidistribution and counting problem goes back to  \cite{Ma}:
 \begin{prop}\label{prop.mixH2.1} 
For any $v\in \fb^+\cap \inte \L_\Gamma$, there exists $\kappa_v>0$ such that for any $f\in C_c(\Gamma\ba G)^{M}$ and $\phi\in C_c(H)^{H\cap M}$, \begin{multline*}
\lim_{t\rightarrow\infty} t^{(r-1)/2}e^{(2\rho-\Theta) (tv )} \int_{H} f ([h] \exp (tv)) \phi (h)\,dh
 = \kappa_{v} \cdot  \mu_{H,v}^{\PS}(\phi)  \,m_{\i(v)}^{\mathrm{BR}}(f)
\end{multline*} where $dh$ denotes the Haar measure on $H$ and $\Theta\in \fa^*$ is given by $\Theta(w)=\langle \nabla\psi_\Ga(v),w\rangle$
as in \eqref{DefT}.
\end{prop}

We note that when $\Ga$ is a lattice or when $G$ is of rank one, the equidistribution for 
the translates 
\begin{equation}\label{eq.HI}
\int_{H} f([h]\exp b_i)\phi(h)\,dh
\end{equation}
 is known for a general sequence $b_i\to\infty$ in $\fb^+$ (\cite[Thm. 1.2]{EM}, \cite{OhShah}).

Note that in our setting of a general Anosov subgroup, Proposition \ref{prop.mixH2.1} 
(or Proposition \ref{eq}) applies only for those sequences $b_i$ from the same fixed direction.
Such lack of uniformity in \eqref{eq.HI} makes it difficult to obtain the main terms of the counting functions in Theorems \ref{m3} and \ref{thm3} directly from Proposition \ref{prop.mixH2.1}. A similar issue as this was treated in \cite{Samb1}.
This difficulty in our setting is overcome by introducing a norm $|\cdot|$ on $\mathfrak a$ induced by an inner product with respect to which $v$ and $\op{ker} \Theta$ are orthogonal to each other, where
 $\Theta$ is a linear form on $\mathfrak a$ tangent to the growth indicator function at $v$. With respect to such a norm, the set of elements in the cone
$\cal L_\Gamma\cap \mathfrak b^+$ of norm at most $T$ can be expressed
as the union $\bigcup_{w\in \op{ker} \Theta} R_T(w)$ where $R_T(w)$ is an {\it interval} of the form $[t_w, \tfrac 12 (-|w|^2+\sqrt{|w|^4+4T^2})]$ for some $t_w\ge 0$. By analyzing an appropriate integral over the interval $R_T(w)$ (Lemma \ref{dom}), we then get the desired main term as in Proposition \ref{lem.w1}.

  \medskip
  
\noindent{\bf Organization}: We start by reviewing some basic notions, including higher rank analogues of Patterson-Sullivan measures as defined by Quint \cite{Quint2} in Section \ref{ps}. Section \ref{sec.meas} introduces generalized BMS-measures, in particular higher-rank versions of Burger-Roblin measures are defined. The product structure of these measures is discussed in Section \ref{sec.Norb}. We then deduce equidistribution of translates of PS-measures on horospheres from local mixing in Section \ref{step1}. This is then used in Section \ref{s2} to show mixing for the Haar measure and equidistribution of translates of Lebesgue measures on maximal horospheres. Properties of the main types of discrete subgroups we study are discussed in Section \ref{sec.Anosov}. The remainder of the paper is mainly devoted to proving the claimed counting statements. As a first step towards this, we prove equidistribution of translates of orbits of symmetric subgroups in $\GaG$ in Section \ref{sec.aff}. These equidistribution statements are combined with the strong wavefront property in Section \ref{gamma} to give the various counting results.

\medskip 
  
\noindent{\bf Acknowledgements.} We would like to thank Andr\'es Sambarino for helpful discussions.   We would also like to thank the anonymous
referee and Anna Wienhard for many helpful comments.
  
 \section{\texorpdfstring{$(\Gamma, \psi)$}{}-Patterson-Sullivan measures}\label{ps}
Let $G$ be a connected, semisimple real algebraic group, and $\Gamma <G$ be a Zariski dense discrete subgroup. In this section, we review the notion of $(\Gamma, \psi)$-Patterson-Sullivan measures associated to a certain class of linear forms $\psi$ on $\fa$, as constructed by Quint in \cite{Quint2}. We present these measures as analogously as possible to the Patterson-Sullivan measures on the limit set of $\Gamma$ in the rank one case.

We fix, once and for all, a Cartan involution $\theta$ of the Lie algebra $\mathfrak{g}$ of $G$, and decompose $\fg$ as $\mathfrak g=\mathfrak k\oplus\mathfrak{p}$, where $\fk$ and $\fp$ are the $+ 1$ and $-1$ eigenspaces of $\theta$, respectively. We denote by $K$ the maximal compact subgroup of $G$ with Lie algebra $\fk$. We also choose a maximal abelian subalgebra $\fa$ of $\mathfrak p$.
Fixing a left $G$-invariant and right $K$-invariant Riemannian metric on $G$ induces a Weyl-group invariant inner product and corresponding norm on $\mathfrak a$, which we denote by $\la\cdot,\cdot\ra$ and $\| \cdot \|$ respectively. Note also that the choice of this Riemannian metric induces a $G$-invariant metric $d(\cdot,\cdot)$ on $G/K$. The identity coset $[e]$ in $G/K$ is denoted by $o$.

Let $A:=\exp \mathfrak a$.
Choosing a closed positive Weyl chamber $\fa^+$ of $\fa$,  let $A^+=\exp \mathfrak a^+$. The centralizer of $A$ in $K$ is denoted by $M$, and we set 
$$N=N^-$$ to be the maximal contracting horospherical subgroup for $A$: for an element
  $a$ in the interior of $A^+$,
  \begin{equation}\label{defN1}N^-=\{g\in G: a^{-n} g a^n\to e\text{ as $n\to +\infty$}\}.\ee
  Note that $\log (N)$ is the sum of all positive root subspaces for our choice of $\fa^+$. Similarly, we will also need to consider the maximal expanding horospherical subgroup
  \begin{equation}\label{defN2}  N^+:=\{g\in G: a^n g a^{-n}\to e\text{ as $n\to +\infty$}\}.\ee
We set $$P^+=MAN^+,\quad\text{and} \quad P=MAN;$$ they are minimal parabolic subgroups of $G$. The quotient $$\F=G/P$$ is known as the Furstenberg boundary of $G$, and is isomorphic to $K/M$.
\medskip 
\begin{Def}[Busemann function] \rm The Iwasawa cocycle $\sigma: G\times \F \to \mathfrak a$ is defined as follows: for $(g, \xi)\in G\times \F$,
$\exp \sigma(g,\xi)$ is the $A$-component of $g k$ in the $KAN$ decomposition, where $\xi=[k]\in  K/M$:
$$gk\in K \exp (\sigma(g, \xi)) N.$$
The Busemann function $\beta: \F\times G/K\times G/K\to\mathfrak a $ is now defined as follows: for $\xi\in \F$ and $[g], [h]\in G/K$,
 $$\beta_\xi ( [g], [h]):=\sigma (g^{-1}, \xi)-\sigma(h^{-1}, \xi).$$
\end{Def}
\medskip
Observe that the Busemann function is continuous in all three variables. To ease notation, we will write  $\beta_\xi ( g, h)=\beta_\xi ( [g], [h])$. The following identities will be used throughout the article:
\begin{gather}\label{eq.basic0}
\begin{aligned}
\beta_\xi(g, h)+\beta_\xi&(h, q)=\beta_\xi (g, q), \\ 
\beta_{g\xi}(gh,gq)&=\beta_\xi(h,q),\text{ and}\\
\beta_\xi(e,g)=&-\sigma(g^{-1}, \xi).
\end{aligned}
\end{gather}
Geometrically, if $\xi=[k]\in K/M$, then for any unit vector $u\in\mathfrak a^+$,
$$\la \beta_\xi(g, h), u\ra =\lim_{t\to +\infty} d ([g], \xi_t)- d([h], \xi_t),$$
where $\xi_t= k \exp (tu)o\in G/K$.
\medskip

\begin{Def} [Conformal measures and densities] \rm
Given $\psi\in \mathfrak a^*$ and a closed subgroup $\Delta <G$,
a Borel probability measure $\nu_{\psi}$ on $\F$ is called a $(\Delta, \psi)$-conformal measure if, for any $\ga\in \Delta$ and $\xi\in \F$,
 \be\label{gc0} \frac{d\ga_* \nu_\psi}{d\nu_\psi}(\xi) =e^{\psi (\beta_\xi (e, \gamma))},\ee
 where $\ga_* \nu_\psi(Q)=\nu_{\psi}(\gamma^{-1} Q)$ for any Borel subset $Q\subset \F$. 
\end{Def}
  
\begin{Def} [Lebesgue measure]\rm
Let $m_o$ denote the $K$-invariant probability measure on $\F$, and
$\rho$ denote the half sum of all positive roots with respect to $\mathfrak a^+$. Then, using the decomposition of the Haar measure in the $KAN$ coordinates, one can check the following (cf. \cite[Proposition 3.3]{Q7}): if $m_o$ is a $(G, 2\rho)$-conformal measure, i.e.\ for any $g\in G$ and $\xi\in \F$,
\begin{equation}\label{Lebconform}
\frac{dg_* m_o}{dm_o}(\xi) =e^{2\rho (\beta_\xi(e,g)) }.
\end{equation}

\end{Def}
\medskip

\noindent{\bf Limit set and Limit cone.}
\begin{Def}[Limit set] \rm The limit set  $\La_\G$ of $\Gamma$ is defined to be the set of all points $x\in \F$ such that the Dirac measure $\delta_x$ is a limit point (in the space of Borel probability measures on $\scrF$) of $\lbrace \gamma_{\ast} m_o\,:\,\gamma\in\Gamma\rbrace$. \end{Def}
We refer to \cite[Lem 2.13]{LO} for an alternative definition. Benoist showed that $\La_\G$ is the minimal $\Gamma$-invariant closed subset of $\F$. Moreover, $\La_\G$ is Zariski dense in $\F$ \cite[Section 3.6]{Ben}. 

An element of $G$ is called elliptic if it is contained in a compact subgroup, and hyperbolic if it is conjugate to an element of $A^+$.
Any $g\in G$ can be written as the commuting product 
\be\label{JJJ} g=g_hg_e g_u \ee
 where $g_h$ is hyperbolic, $g_e$ is elliptic and $g_u$ is unipotent.
  An element $g\in G$ is called {\it loxodromic} if $g_h$ is conjugate to an element of $ \op{int}A^+$.

\begin{lemma} \label{Can} For any open subset $U\subset\F$ with $U\cap \La_\Ga\ne \emptyset$, 
$U\cap \La_\G$ is not contained in any smooth submanifold of $ \F$ of smaller dimension.
\end{lemma}
\begin{proof}  This is proved in \cite{Wi} when $G$ has rank one, and our proof is similar.  
Suppose that there exists an open subset $U\subset \F$ such that
$U\cap \La_\G\ne \emptyset$ is contained in a smooth submanifold $S$ of $\F$ of smaller dimension. Since $\Ga$ is Zariski dense,
the set of attracting fixed points of loxodromic elements of $\Ga$ is dense in $\La_\Ga$. Hence there exists a loxodromic element $\ga\in\Ga$ whose attracting fixed point
is contained in $U\cap\La_\Ga$.
We may assume without loss of generality that $\gamma=am$ where $m\in M$ and $a\in \op{int} A^+ $, and hence $[e]=P\in \F$ is the attracting fixed point of $\ga$.
Since $\mathfrak n^+:=\op{Lie} N^+$ is nilpotent, the map
$ \mathfrak n^+\to G/P$ given by $x\mapsto \exp(x)[e] $ is algebraic and its image $N^+[e]$ is Zariski open and dense in $G/P$ for $[e]=P$.
Therefore $N^+[e]\cap \Lambda_\Ga$ is Zariski dense in $N^+[e]$.
On the other hand, since $U$ is a neighborhood of $[e]$, we may assume without loss of generality that
 $U\subset N^+[e]$ by replacing $U$ by a smaller open neighborhood of $[e]$ if necessary.
 We now claim that the hypothesis that $U\cap \La_\G\subset S$ implies that
  $N^+[e]\cap \La_\Ga$ cannot be Zariski dense in $N^+[e]$, which then yields a desired contradiction.

Choose a basis of $\mathfrak n^+$ consisting of eigenvectors of $\op{Ad}_a$, and for $x\in\mathfrak n^+$, we write
 $x=(x_1,\cdots,x_d)$ for the coordinates with respect to this basis.
It follows that there exist $0<c_i<1$, $i=1, \cdots, d$ such that
\begin{equation}\label{eq.eig2}
\op{Ad}_a x=(c_1 x_1,\cdots,c_d x_d) .
\end{equation}
Choose $\ell\in\bb N$ such that $c_1>\max_{1\leq i\leq d}c_i^{\ell+1}$.
By the implicit function theorem, after shrinking $U$ and rearranging the indices if necessary, we may assume that
$
U\cap S=\{[\exp(x)]\in U : x_1=f(x_2,\cdots,x_d)\}
$
for some smooth function $f$. Let $p$ be the Taylor polynomial of $f$ of degree $\ell$.
Then, by shrinking $U$ further, there exists $C>0$ such that for all $x\in U$,
\begin{equation}\label{eq.app1}
|f(x_2,\cdots,x_d)-p(x_2,\cdots,x_d)|\leq C\norm{(x_2,\cdots,x_d)}^{\ell+1}
\end{equation}
(here $\|\cdot\|$ denotes the Euclidean norm on $\fn^+$). Since the action of $\op{Ad}_a$ on the polynomial ring $\bb R[\mathfrak n^+]$  is diagonalizable, we can write
\begin{equation}\label{eq.eig}
x_1-p(x_2,\cdots,x_d)=\sum_{i=1}^k p_i(x),
\end{equation}
where $p_i\in \bb R[\mathfrak n^+]$ are non-zero polynomials such that $p_i(\op{Ad}_a x)=\beta_i\cdot p_i(x)$
where $
1>\beta_1>\cdots>\beta_k>0.
$
Note that $\beta_1\geq c_1$, due to the presence of $x_1$ in \eqref{eq.eig}.
Since $U\cap\La_\Ga\subset S$, combining \eqref{eq.app1} and \eqref{eq.eig}, we conclude
\begin{equation}\label{eq.com}
\left|\sum_{i=1}^k p_i(x)\right |\leq C \norm{(x_2,\cdots,x_d)}^{\ell+1},\text{whenever } [\exp x] \in U\cap\La_\Ga.
\end{equation}
Now let  $x\in \mathfrak n^+$ be such that $[\exp x]\in \La_\Ga$.
Since $[e]\in U\cap\La_\Ga$ and $\op{Ad}_{\ga^n}x\to 0$ as $n\to\infty$, we have $[\exp \op{Ad}_{\ga^n}x]\in U\cap\La_\Ga$ for all sufficiently large $n$. Applying \eqref{eq.eig2} and \eqref{eq.com}, we obtain
$$
\left|\sum_{i=1}^k p_i(\op{Ad}_{\ga^n}x)\right |= \left|\sum_{i=1}^k \beta_i^{n}\cdot p_i(\op{Ad}_{m^n}x) \right|\leq C\left(\max_{1\leq i\leq d}{c_i^{\ell+1}}\right)^{n}\norm{\op{Ad}_{m^n} x}^{\ell+1}.
$$
Therefore, by dividing by $\beta^n$,
$$
\left|\,p_1( \op{Ad}_{m^n} x)+\sum_{i=2}^k \left(\tfrac{\beta_i}{\beta_1}\right)^n p_i(\op{Ad}_{m^n}x)\, \right|\leq C\left(\tfrac{\max_{1\leq i\leq d}{c_i^{\ell+1}}}{\beta_1}\right)^{n}\norm{\op{Ad}_{m^n} x}^{\ell+1}.
$$
Since $M$ is  a compact subgroup, it follows that for some sequence $n_i\to \infty$, $m^{n_i}\to e$.
Taking the limit along this subsequence yields
$p_1(x)=0$.
This shows that $\{x\in \mathfrak n^+:[\exp x] \in \La_\Ga\} \subset\{x\in\mathfrak n^+ : p_1(x)=0\},$
implying that $N^+ [e] \cap \La_\Ga$
is not Zariski dense in $N^+[e]$.
This finishes the proof.
\end{proof}

 We remark that a weaker version of this lemma was proved for $G=\SL_n(\mathbb C)$ by Cantat  (see \cite[Sec.6]{GR}).

\begin{Def} [Cartan projection] \rm The Cartan projection $\mu: G\to \fa^+$ is defined as follows:
for each $g\in G$, there exists a unique element $\mu(g)\in \mathfrak a^+$ such that
\begin{equation*}
g\in K\exp(\mu(g))K.
\end{equation*}
\end{Def}

The Jordan projection of $g$ is defined
 as $ \lambda(g)\in \fa^+$, where $\exp \lambda(g)$ is the element of $A^+$ conjugate to $g_h$ where $g_h$ is as in \eqref{JJJ}.
 
 \begin{Def}[Limit cone] \rm The \emph{limit cone} $\L_\Gamma\subset\fa^+$ of $\Gamma$ is defined as the smallest closed cone containing the \emph{Jordan projection} of $\Gamma$.
\end{Def}

Quint showed the following:
\begin{thm} \cite[Theorem IV.2.2]{Quint1} \label{growth} The growth indicator function $\psi_\Gamma$, defined in \eqref{grow}, is concave, upper-semicontinuous, and satisfies
$$\L_\Gamma= \{u\in \fa^+: \psi_\Gamma(u)>-\infty\}.$$
Moreover, $ {\psi_\Gamma}(u)$ is non-negative on  $\L_\Gamma$ and positive on $\op{int}\L_\Gamma$.
\end{thm}


\begin{prop}\label{vanish} Let $\mathsf m$ be a locally finite Borel measure on $\Gamma\ba G$ and $\scrC\subset\fa^+$ a closed cone with $\scrL_\Ga\subset \inte\scrC$.
For any $f_1, f_2\in C_c(\Ga\ba G)$, there exists $t_0>0$ such that for all $v \in \fa^+ -\scrC $ with $\|v\|\ge t_0$,
$$
\int_{\Ga\ba G} f_1(x\exp(v))f_2(x)\,d\mathsf m(x)=0.
$$
In particular, if $u\in \fa^+-\L_\Ga$, then
$\int_{\Ga\ba G} f_1(x\exp(tu))f_2(x)\,d\mathsf m(x)=0$ for all $t\gg 1$ large enough.
\end{prop}
\begin{proof}
It suffices to check that for any given compact subset $L\subset G$ such that $L=L^{-1}$, we have
$$L\exp(-v)\cap \Ga L=\emptyset$$
 for all sufficiently large $v\in \mathfrak a^+-\scrC$.

Suppose that there exist a compact subset $L\subset G$, sequences  $\ell_n,\ell_n'\in L$, $\ga_n\in\Ga$, and $v_n\in \fa -\cal C$ with $\|v_n\|=t_n\to\infty$ such that 
$$\ell_n\exp(-v_n)=\ga_n\ell_n'.$$
We may assume that $v_n/t_n$ converges to some unit vector $v\in (\fa^+-\inte\scrC)$; hence $v\not\in\scrL_\Ga$.
By \cite[Lemma 4.6]{Ben}, there exists a compact subset $M=M(L)$ of
$\fa$ such that for all $g\in G$, 
$$
\mu(LgL)\subset\mu(g)+M.
$$
Note that $\cal L_\Ga$ is equal to the the asymptotic cone of $\mu(\Ga)$ \cite[Thm. 1.2]{Ben}.
Since $v\not\in\cal L_\Ga$, 
we can find an open cone $\cal D$ containing $v$ such that $\cal D\cap \mu(\Ga)$ is finite.
Then 
$$
\mu(\ga_n^{-1})=\mu(\ell_n'\exp(v_n)\ell_n^{-1})\in\mu(\exp(v_n))+M.
$$
As $\cal D$ is open, there exists $n_0$ such that
$\frac{1}{t_n} (v_n+ M) \subset\cal D$ for all  $n\ge n_0$.
Since $\mu(\exp(v_n))+M=t_n(\frac{v_n}{t_n} +\frac{1}{t_n}M)$, we conclude $\mu(\ga_n^{-1})\in \cal D$ for all $n\ge n_0$.
This yields a contradiction.

The second claim follows from the first one
 as we can find a closed cone $\cal C$ such that $u\notin \cal C$ and $\L_\Gamma\subset \inte \cal C$.
\end{proof}

Set $$D_\Gamma:=\{\psi\in \fa^*: \psi \ge \psi_\Gamma \text{ on $\mathfrak a^+$}\},$$
which  is a non-empty set \cite[Section 4.1]{Q4}.
An element $\psi\in D_\Gamma$ is said to be tangent to $\psi_\Gamma$ at $u\in \fa$
if $\psi(u)=\psi_\Gamma(u)$.
 The following collection of linear forms is of particular importance:
\be\label{ds} D_\Gamma^\star:=\{\psi\in D_\Gamma: \psi \text{ is tangent to $\psi_\Ga$ at some $u\in \mathcal L_\Ga\cap  \op{int} \fa^+ $}\}.\ee

\begin{Def}[Patterson-Sullivan measures] \rm For $\psi\in\fa^*$, a $(\Gamma,\psi)$-conformal measure supported on $\Lambda_\Gamma$ will be called a $(\Gamma, \psi)$-PS measure.
\end{Def}

Generalizing the work of Patterson-Sullivan (\cite{Pa}, \cite{Su}),
Quint \cite{Quint2} constructed  a $(\Gamma, \psi)$-PS measure for every $\psi\in \dg$.
  
\noindent{\bf Maximal growth direction.} 
  Since $\psi_\Gamma$ is concave, upper-semicontinuous, and the unit norm ball in $\fa$ is strictly convex, there exists a unique unit vector
  ${{u_\Ga}}\in 
  \cal L_\Gamma$ (called the maximal growth direction) such that
   \begin{equation}\label{tug}\delta_\Ga:=\max_{u\in\mathfrak a^+,\norm{u}=1}\psi_\Ga(u)=\psi_\Ga({u_\Ga}).\end{equation}
Note that $u_{\Gamma}$ must be stabilized by the \emph{opposition involution} (see Def. \ref{op})
\begin{Ex} \rm If $G=\PSL_3(\br)$, then $u_\Gamma=\op{diag} (\frac{1}{\sqrt 2}, 0, -\frac{1}{\sqrt 2})$ for any Zariski dense subgroup
$\Gamma$.
\end{Ex}

\subsection*{Uniqueness of tangent forms.}
The claims (2) and (3) of the following lemma follows from  \cite[Sec 4.1]{Q4} (see also \cite[Lemma 4.8]{Samb2}).
\begin{lemma}  \label{di}
\begin{enumerate}
\item
For any $u\in \op{int}\L_\Gamma$, there exists a linear form $\psi_u\in D_\Gamma^\star$ tangent to $\psi_\Gamma$ at $u$.
\item
For any unit vector $u\in \op{int} \L_\Gamma$ at which $\psi_\Ga$ is differentiable, there exists a unique linear form $\psi_u\in D_\Gamma^\star$ tangent to $\psi_\Gamma$ at $u$, and it is given by
\begin{equation*}
\psi_u(\cdot)=\la(\nabla\psi_\Ga)(u), \cdot\ra =D_u \psi_\Gamma(\cdot);
\end{equation*}
\item
If ${u_\Ga}\in\op{int}\L_\Gamma$ and $\psi_\Ga$ is differentiable at ${u_\Ga}$, then
$\psi_{{u_\Ga}}$ is given by
$$\psi_{u_\Ga} (\cdot )=\delta_\Ga\la {u_\Ga}, \cdot \ra.$$
\end{enumerate}
\end{lemma}
\begin{proof} 
Let $P\subset\mathfrak a$ be an affine hyperplane such that $P\cap \mathfrak a^+$ is an $(r-1)$-simplex and 
$P\cap\cal L_\Ga$ is a bounded convex subset of $P\simeq \bb{R}^{r-1}$. 
Since $P\cap\scrL_{\Gamma}$ is convex and $\psi_\Ga : \mathfrak a\to\bb{R}$ is concave, 
the following set $S$ is convex:
$$ S:=\lbrace (x,y)\in (P\cap\cal L_\Ga)\times\bb{R}\,:\,0\leq y\leq \psi_{\Gamma}(x)\rbrace.$$
Since $\br (P\cap \op{int} \L_\Gamma)\supset \op{int} \L_\G$, it suffices to prove (1) for $ u\in P\cap\op{int} \cal L_\Ga$.
Since  $(u,\psi_{\Gamma}(u))\in \partial S$,  the supporting hyperplane theorem implies that there exists a hyperplane $C\subset P\times \RR$ passing through $(u,\psi_{\Gamma}(u))$ such that the interior of $S$ is contained in a connected component
of $P\times\bb R \setminus C$. As $u\in\op{int}\cal L_\Ga$, such a hyperplane $C$ must be the graph of a function. We may therefore write
$C=\{(x,\varphi(x))\in P\times\bb{R} \}$
for some affine map $\varphi :P\to\bb{R}$ satisfying $\varphi(x)\geq \psi_\Gamma(x)$ for all $x\in P\cap\cal L_\Ga$.
Consider the unique linear form in $\fa^*$ which extends $\varphi$, which we also denote by $\varphi$ by abuse of notation.
Since  $\varphi(x)\geq\psi_\Ga(x)$  for all $x\in P\cap\cal L_\Ga$ and $\L_\Ga$ has non-empty interior, it follows that  $\varphi\geq\psi_\Ga$.
Since $\varphi(u)=\psi_\Ga(u)$, this proves (1).

To prove (2), define $\psi_u(\cdot):=\la(\nabla\psi_\Ga)(u), \cdot\ra.$
By differentiating $\psi_\Ga(tu)=t\psi_\Ga(u)$ with respect to $t$, we get by the chain rule
that
\begin{equation}\label{eq.scale}
\la \nabla\psi_\Ga(tu),u\ra=\psi_\Ga(u).
\end{equation}

Hence $\psi_u(u)=\psi_\Ga(u)$ by plugging $t=1$.
Next, let $\mathfrak b$ be a vector space such that $\mathfrak a=\bb Ru\oplus\mathfrak b$, and let $v\in\mathfrak b$ be arbitrary.
Consider the closed interval $I=\{s\in\bb R : u+sv\in \cal L_\Ga\} $ and let $f(s):=\psi_{\Ga}(u+sv)$.
Note that $f:I\to \br$ is concave, differentiable at $s=0$, and $f'(0)=\la\nabla\psi_\Ga(u),v\ra$.
Hence, using \eqref{eq.scale},
$$
f(s) \leq \psi_\Ga(u)+s\la\nabla\psi_\Ga(u),v\ra
=\la \nabla\psi_\Ga(u), u+sv\ra=\psi_u(u+sv).$$
As $v\in \mathfrak b$ is arbitrary, this implies $\psi_u\ge \psi_\Ga$. Hence $\psi_u\in D_\Gamma^\star$.

To show the uniqueness, suppose that $\psi\in D_\Gamma^\star$  is tangent to $\psi_\Ga$ at $u$. Let $v\in\mathfrak b$ be arbitrary.
Define $g:I\to\br$ by  $g(s):= \psi(u+sv) $.  Then $g\ge f$ and $g(0)=f(0)$. 
Since  $f$ is a concave function on an interval $I$ and
 differentiable at $0\in S$, it follows that $g(s)=f(0)+ sf'(0) $. Since $f(0)=\psi_\Gamma(u)$ and $f'(0)=\la \nabla\psi_\Ga(u), v \ra$, it follows that
 $$\psi(u+sv)=\psi_\Gamma (u) +s \la \nabla\psi_\Ga(u), v \ra = \la \nabla\psi_\Ga(u),u+s v \ra ;
$$ 
 this
 proves the uniqueness.

Next we claim that $\nabla\psi_\Ga({u_\Ga})=c {u_\Ga}$ for some $c\ne 0$. 
Consider a curve $\alpha:(-\epsilon,\epsilon)\to\bb S^{r-1}\cap\mathfrak a^+$ such that $\alpha(0)={u_\Ga}$.
By definition of ${u_\Ga}$, $s\mapsto\psi_\Ga(\alpha(s))$ achieves its maximum at $s=0$.
Hence, its derivative at $s=0$ vanishes, and
$\la\nabla\psi_\Ga({u_\Ga}),\alpha'(0)\ra=0.$
Since $\alpha'(0)\in\op{T}\bb{S}^{r-1}$ can be arbitrary, $\nabla\psi_\Ga({u_\Ga})$ is parallel to ${u_\Ga}$.
Combining this with \eqref{eq.scale},  the claim follows.
Since  $\psi_\Ga({u_\Ga})=\delta_\Gamma$, we have $c=\delta_\Gamma$.
This completes the proof of the lemma.
\end{proof}

\section{Generalized BMS measures}\label{sec.meas}
 Using the notation introduced in Section \ref{ps}, given a pair of $\Gamma$-conformal measures on $\F$, we now define an $MA$-invariant locally finite Borel measure on $\Gamma\ba G$, which we call a generalized BMS-measure.  
Haar measures, BR-measures, and BMS measures are all constructed in this way.

\begin{Def}[Opposition involution]\label{op} \rm Denote by $w_0\in K$ a representative of the unique element of the Weyl group $N_K(A)/M$ such that $\op{Ad}_{w_0}\mathfrak a^+= -\mathfrak a^+$.
  The opposition involution  $\i:\mathfrak a \to \mathfrak a$ is defined by $$\i (u)= -\op{Ad}_{w_0} (u).$$
 Note that for all $g\in G$, we have 
 $$\lambda(g^{-1})=\i(\lambda(g)), \quad \mu(g^{-1})=\i(\mu(g)),\quad\text{and}$$
 $$\i(\mathfrak a^+)=\mathfrak a^+\quad\text{and}\quad \psi_\G \circ \i=\psi_\G.$$
In particular, $\i$ preserves $\inte \L_\Ga$. \end{Def}
Note that  for all rank one groups, $\op i$ is the identity map.
\medskip

\noindent{\bf Example.} When $G=\op{PSL}_d(\bb R)$, with the Riemannian metric given by the inner product $\langle X,Y\rangle=\op{tr}( XY^t)$, we have
\begin{gather}\label{eq.PSL}
\begin{aligned}
&\mathfrak a=\{\op{diag}(t_1,\cdots,t_{d}) : t_1+\cdots+ t_{d}=0\},\; \\
&\mathfrak a^+=\{\op{diag}(t_1,\cdots,t_{d})\in\mathfrak a : t_1\geq\cdots\geq t_d\}
\end{aligned}
\end{gather}
and $\langle \cdot,\cdot \rangle :\mathfrak a\to\bb R$ is given by $\langle X,Y\rangle=\op{tr}(XY)$.
The opposition involution is given by
\begin{equation*}
\op i(\op{diag}(t_1,\cdots,t_d))=\op{diag}(-t_d,\cdots,-t_1).
\end{equation*}

   For each $g\in G$, we define 
   $$g^+:=gP\in G/P\quad\text{and}\quad g^-:=gw_0P\in G/P.$$
Observe that $(gm)^{\pm}=g^{\pm}$ for all $g\in G$, $m\in M$; we may thus also view the above as maps from $G/M$ to $\scrF$.
 Hence, for the identity element $e\in G$, $e^+=[P]$, $e^-=[w_0P]$ and $g^{\pm}=g(e^{\pm})$ for any $g\in G$.
 Let $\F^{(2)}$ denote the unique open $G$-orbit in $\F\times \F$:
$$\F^{(2)}=G. (e^+, e^-)=\{(gP, gw_0P)\in G/P\times G/P: g\in G\}.$$ 
\medskip
\noindent \textbf{Example.} If $G=\PGL_d(\RR)$, $\scrF$ may be identified with the space of complete flags
$ \{\lbrace V_1\subset \cdots \subset V_{d-1}\rbrace: \text{dim} V_i=i\}$ in $\RR^d$; $\scrF^{(2)}$ is then identified with the set of pairs of flags 
$(\lbrace V_1\subset \cdots \subset V_{d-1}\rbrace , \lbrace W_1\subset \cdots\subset
 W_{d-1} \rbrace)$  in general position, i.e., $V_i\oplus W_{d-i}=\RR^d$ for all $1\le i \le d-1$.

\begin{Def}[Hopf parameterization]\rm The homeomorphism 
  $G/M\to \FF \times \mathfrak a$
  given by $gM \mapsto (g^+, g^-, b=\beta_{g^-}(e,g)) $
  is called the Hopf parameterization of $G/M$.
\end{Def}

 \noindent{\bf Example.} Using the linear fractional transformation action of $G=\PSL_2(\br)$ on $\bH^2\cup\hat{\br}$,
 we have $P^-=\op{Stab}(\infty)$ and $P^+=\op{Stab}(0)$, where $P=P^-$ and $P^+$ are the upper and lower triangular subgroups of $G$ respectively. Hence
 $$g^+=gP=g(\infty), \;\;  g^-=gw_0 P=g(0) \in \partial \bH^2=G/P$$
where $w_0=\begin{psmallmatrix} 0 & 1\\-1&0\end{psmallmatrix}$.

\medskip

We will make use of the following identities, which are all straightforward:
\begin{gather}\label{eq.basic1}
\begin{aligned}
\beta_{g^-}(e,g)=-\sigma(g^{-1},kw_0P)&=-\i(\log a)\text{ if }g=kan\in KAN^+;\\
\beta_{g^+}(e,g)=-\sigma(g^{-1},kP)&=\log a\text{ if }g=kah\in KAN^-.
\end{aligned}
\end{gather}

In particular, for $a\in A$,
$$\beta_{e^+}(e,a)+\op i  (\beta_{e^-}(e,a))=0.$$

\noindent \textbf{The generalized BMS-measure: $m_{\nu_1,\nu_2}$.}
 Fix   a pair of linear forms $\psi_1, \psi_2\in \mathfrak a^*$. 
 Let $\nu_1=\nu_{\psi_1}$ and $\nu_2=\nu_{\psi_2}$ be respectively
 $(\Gamma,\psi_1)$ and $(\Gamma,\psi_2)$ conformal measures on $\F$.
Using the Hopf parametrization, define the following locally finite Borel measure $\tilde m_{\nu_1, \nu_2}$ on $G/M$ 
  as follows: for $g=(g^+, g^-, b)\in \F^{(2)}\times \mathfrak a$,
\begin{equation}\label{eq.BMS0}
d\tilde m_{\nu_1, \nu_2} (g)=e^{\psi_1 (\beta_{g^+}(e, g))+\psi_2( \beta_{g^-} (e, g )) } \;  d\nu_{1} (g^+) d\nu_{2}(g^-) db,
\end{equation}
  where $db=d\ell (b) $ is the Lebesgue measure on $\mathfrak a$.
\begin{lemma}\label{qa} 
The measure $\tilde m_{\nu_1, \nu_2}$ is left $\Gamma$-invariant and right $A$-quasi-invariant:
for all $a\in A$,
$$
a_*\tilde m_{\nu_1, \nu_2}=e^{(-\psi_1+\psi_2\circ \i)(\log a)}\,\tilde m_{\nu_1, \nu_2} .$$
\end{lemma}
\begin{proof}
Let  $\ga\in\Ga$ and $g\in G$ be arbitrary.
Note
\begin{equation*}
\beta_{\ga g^\pm}(e,\ga g)=\beta_{\ga g^\pm}(e,\ga)+\beta_{\ga g^\pm}(\ga,\ga g)=\beta_{\ga g^\pm}(e,\ga)+\beta_{ g^\pm}(e,g).
\end{equation*}
Recall the conformality of the measures $\nu_{1}$ and $\nu_2$:
\begin{align*}
d\nu_{1}(\ga g^+)=e^{\psi_1(\beta_{g^+}(e,\ga^{-1}))}d\nu_{1}(g^+)\text{ and }
d\nu_{2}(\ga g^-)=e^{\psi_2(\beta_{g^-}(e,\ga^{-1}))}d\nu_{2}(g^-).
\end{align*}
Combining these, we have
\begin{align*}&
d\tilde m_{\nu_1, \nu_2}(\ga g)=e^{\psi_1(\beta_{\ga g^+}(e,\ga g))+\psi_2(\beta_{\ga g^-}(e,\ga g))}d\nu_{1}(\ga g^+)d\nu_{2}(\ga g^-)d\ell(b+\beta_{\ga g^-}(e,\ga))\\
&=e^{\psi_1(\beta_{g^+}(e,g))+\psi_2(\beta_{g^-}(e,g))}d\nu_{1}(g^+)d\nu_{2}(g^-)d\ell(b)=d\tilde m_{\nu_1, \nu_2}(g).
\end{align*}

Let  $a\in A$.
By the identities \eqref{eq.basic0} and \eqref{eq.basic1},
\begin{gather*}
\beta_{g^\pm}(e,ga)=\beta_{g^\pm}(e,g)+\beta_{g^\pm}(g,ga)=\beta_{g^\pm}(e,g)+\beta_{e^\pm}(e,a);\\
ga^\pm=g^\pm,\text{ and }\beta_{e^+}(e,a)=-\op i \beta_{e^-}(e,a)=\log a.
\end{gather*}
Combining with definition \eqref{eq.BMS0}, we have
\begin{align*}
&d\tilde m_{\nu_1, \nu_2}(ga)=e^{\psi_1(\beta_{g^+}(e,ga))+\psi_2(\beta_{g^-}(e,ga))}d\nu_{1}(g^+)d\nu_{2}(g^-)d\ell(b+\beta_{e^-}(e,a))\\
&=e^{\psi_1 (\beta_{g^+}(e, g)+\beta_{e^+}(e,a))+ \psi_2  (\beta_{g^-} (e,g)+\beta_{e^-}(e, a)) } \;   d\nu_{1} (g^+) d\nu_{2} (g^-) d\ell(b)\\
&=e^{(\psi_1-\psi_2\circ \i)(\log a)} d\tilde m_{\nu_1, \nu_2}(g). 
\end{align*} This proves the claim.
\end{proof}

The measure $\tilde m_{\nu_1, \nu_2}$ 
gives rise to a left $\Ga$-invariant and right $M$-invariant measure on $G$, by integrating along the fibers of $G\to G/M$ with respect to the Haar measure on $M$.
By abuse of notation, we will also denote this measure by $\tilde m_{\nu_1, \nu_2}$.
We denote by $m_{\nu_1, \nu_2}$
 the measure on $\Gamma\ba G$ induced by $\tilde m_{\nu_1, \nu_2}$, and call it the generalized BMS-measure associated to the pair $(\nu_1, \nu_2)$.

 \medskip

\noindent
 \textbf{BMS-measure: $m_{\nu_\psi, \nu_{\psi\circ\i}}^{\BMS}$.} Let $\psi\in\fa^*$  and let $\nu_{\psi}$ and $\nu_{\psi\circ \i}$ be  respectively
  $(\Gamma, \psi)$ and $(\Gamma, \psi \circ \i)$-PS measures.
We set 
\begin{equation}\label{def.BMS}
m^{\BMS}_{\nu_\psi, \nu_{\psi\circ\i}}:=m_{\nu_\psi, \nu_{\psi\circ\i}}
\end{equation}
and call it the Bowen-Margulis-Sullivan measure associated to $(\nu_\psi, \nu_{\psi\circ\i})$.
By Lemma \ref{qa}, $m^{\BMS}_{\nu_\psi,\nu_{\psi\circ\i}}$ is an $A$-invariant measure, whose support is given by 
 $$\text{supp}(m^{\BMS}_{\nu_\psi, \nu_{\psi\circ\i}})=\{x \in\Ga\ba G : x^\pm\in\La_\Ga\};$$
 since $\Lambda_\G$ is $\Gamma$-invariant, the condition $x^{\pm}\in \Lambda_\G$ is a well-defined condition.
\medskip 

 \noindent
 \textbf{BR-measures: $m^{\BR}_{\nu_\psi}$.}   We set 
 \begin{equation}\label{def.BR}
m^{\BR}_{\nu_{\psi}}=m_{m_o, \nu_{\psi}};
\end{equation} 
 and call it the Burger-Roblin measure associated to $\nu_\psi$.
 Note that the support of $m^{\BR}_{\nu_\psi}$
 is given by 
$$ \text{supp}(m^{\BR}_{\nu_\psi})=\{x \in \Gamma\ba G: x^-\in \Lambda_\Gamma\}.$$
 \begin{lemma}\label{lem.BRinv}
The Burger-Roblin measure $ m_{\nu_\psi}^{\BR}$ is  right $N^+$-invariant.
\end{lemma}
\begin{proof}
Let $g\in G$ and $n\in N^+$.
By the identities \eqref{eq.basic0} and \eqref{eq.basic1}, we have $gn^-=g^-$, $\beta_{n^-}(e,n)=0$ and
\begin{equation*}
\beta_{gn^-}(e,gn)=\beta_{gn^-}(e,g)+\beta_{gn^-}(g,gn)=\beta_{gn^-}(e,g)+\beta_{n^-}(e,n)=\beta_{g^-}(e,g).
\end{equation*}
On the other hand, by the conformality \eqref{gc0},
\begin{align*}
dm_o(gn^+)&=e^{2\rho(\beta_{gn^+}(gng^{-1},e))}dm_o(g^+)\\
&=e^{2\rho(\beta_{gn^+}(gng^{-1},gn)+\beta_{gn^+}(gn,e))}dm_o(g^+)\\
&=e^{2\rho(\beta_{g^+}(e,g)-\beta_{gn^+}(e,gn))}dm_o(g^+).
\end{align*}
Combining these, we have
\begin{align*}
d\tilde m_{m_o, \nu_\psi}(gn)&=e^{2\rho (\beta_{gn^+}(e, gn))+ \psi (\beta_{g^-} (e, gn)) } \;   dm_o (gn^+) d\nu_{\psi} (g^-) d\ell(b)\\
&=e^{2\rho (\beta_{g^+}(e, g))+ \psi (\beta_{g^-} (e, g)) } \;   dm_o (g^+) d\nu_{\psi} (g^-) d\ell(b)\\
&=d\tilde m_{m_o,\nu_\psi}(g).
\end{align*} \end{proof}

 Similarly, 
but with a different parametrization $g=(g^+,g^-,b=\beta_{g^+}(e,g))$, we define the following $N^-$-invariant locally finite measure:
\be\label{dualb} d\tilde m ^{\BR_*}_{\nu_\psi}(g)=e^{\psi (\beta_{g^+}(e, g))+2\rho (\beta_{g^-} (e, g )) } \; d\nu_{\psi} (g^+) dm_o (g^-)  db. \ee
\noindent
\textbf{Haar measure $m^{\Haar}$.} Recall that the $K$-invariant probability measure $m_o$ is a conformal density
for  the linear form $2\rho$.  We denote by $dx=dm^{\Haar}(x)$ the generalized BMS measure associated to the pair
$(m_o, m_o)$:
\be\label{Haar5} dx=dm^{\Haar}:=dm_{m_o,m_o}.\ee
  Since $m_o$ is a $(G, 2\rho)$-conformal measure, $dm^{\Haar}$ is $G$-invariant;
 the proofs of Lemma \ref{qa} and Lemma \ref{lem.BRinv} show that $m_o$ is invariant under $AM$ and $ N^{\pm}$. As these subgroups generate $G$,
 the $G$-invariance follows). 
 
\section{Disintegration of the \texorpdfstring{$\BMS$}{} and \texorpdfstring{$\BR$}{}-measures along \texorpdfstring{$N$}{}-orbits}\label{sec.Norb}
In this section, we fix a linear form $\psi\in \dg$,  a $(\Gamma, \psi)$-PS measure $\nu_\psi$ and
a $(\Gamma, \psi\circ \i)$-PS measure $\nu_{\psi\circ \i}$ on $\F$.
To simplify the notations, we write 
 $$\nu:=\nu_{\psi},\;\nu_{\op{i}}:=\nu_{\psi\circ \op{i}}, \; \tilde m^{\BMS}:=\tilde m^{\BMS}_{\nu_\psi, \nu_{\psi\circ \i}}, \; \tilde m^{\BR}:=\tilde m^{\BR}_{\nu_{\psi\circ \i}}, \; \tilde m^{\BR_*}:=\tilde m^{\BR_*}_{\nu_\psi}.$$

\subsection{PS-measures on $gN^{\pm}$} \label{psm} We start by defining measures on $N^{\pm}$. Firstly, for $g\in G$, define $\mu_{gN^+}^{\mathrm{PS}}:=\mu_{gN^+, \nu}^{\mathrm{PS}} $ and
$ \mu_{gN^-}^{\mathrm{PS}}:=\mu_{gN^-, \nu_\i}^{\mathrm{PS}}$ on $N^{\pm}$ by the formulas:
for $n\in N^+$ and $h\in N^-$, 
\be\label{eq.PSN}   d\mu_{gN^+}^{\mathrm{PS}}(n):=e^{\psi(\beta_{(gn)^+} (e,gn) ) }d\nu ((gn)^+)\ee
and $$d\mu_{gN^-}^{\mathrm{PS}}(h):=e^{\psi \circ \i(\beta_{(gh)^-}(e, gh)) }d\nu_\i((gh)^-).$$

These are left $\Gamma$-invariant; for any  $\gamma\in \Gamma$ and $g\in G$, 
$$\mu_{\gamma gN^\pm }^{\mathrm{PS}}= \mu_{gN^{\pm}}^{\mathrm{PS}}.$$
When $xN^{\pm}$ is closed in $\G\ba G$ for $x=[g]\in \Gamma \ba G$,
$\mu_{gN^{\pm}}^{\mathrm{PS}}$ induces a locally finite Borel measure on 
$ \op{Stab}_{N^{\pm}}(x)\ba N^{\pm} \simeq xN^{\pm}$ which we will denote by
$d\mu_{xN^{\pm}}^{\PS}$.

Recalling that $A$ normalizes $N^{\pm}$, we will use the following lemma:

\begin{lem}\label{MANconj} For any $g\in G,\,a\in A,\,n_0,n\in N^+$, we have
\begin{equation*}
d\mu_{gN^+}^{\mathrm{PS}}\big( an_0n a^{-1}\big)=e^{-\psi(\log a)}d\mu_{gan_0N^+}^{\mathrm{PS}}(n). \end{equation*}
\end{lem}
\begin{proof}
By \eqref{eq.basic0} and \eqref{eq.basic1}, we have $(gan_0na^{-1})^+=gan_0n^+$, and
\begin{align*}
\beta_{gan_0n^+}(e,gan_0na^{-1})&=\beta_{gan_0n^+}(e,gan_0n)+\beta_{gan_0n^+}(gan_0n,gan_0na^{-1})\\
&=\beta_{gan_0n^+}(e,gan_0n)+\beta_{e^+}(e,a^{-1}).
\end{align*}
Also note that $\beta_{e^+}(e,a^{-1})=-\log a$.
Consequently,
\begin{align*}
d\mu_{gN^+}^{\mathrm{PS}}(an_0na^{-1})&=e^{\psi(\beta_{gan_0n^+}(e,gan_0na^{-1}))}d\nu (gan_0n^+)
\\&=e^{-\psi(\log a)}e^{\psi(\beta_{gan_0n^+}(e,gan_0n))}d\nu (gan_0n^+)
\\&=e^{-\psi(\log a)}d\mu_{gan_0N^+}^{\mathrm{PS}}(n).
\end{align*}
\end{proof}

The measures $\mu_{gN^{\pm}}^{\mathrm{PS}}$ allow us to decompose the BMS-measure as follows:
The product map $N^+\times P^-\to G$ is a diffeomorphism onto a Zariski open neighborhood of $e$.

\subsection{Product structure of BMS measures}\label{measdecompsec}
Given $g\in G$, the BMS measure
 $\tilde{m}^{\mathrm{BMS}}$ can be disintegrated in $g N^+P^-$ as follows. 

\begin{lem}\label{NPdecomp}
For $g\in G$, $f\in C_c(gN^+P^-)$, and $nham\in N^+N^-AM$,
\begin{equation*}
\tilde{m}^{\mathrm{BMS}}(f)=\int_{N^+}\left(\int_{N^-AM}  f(gnham)\,dm\, da \; d\mu_{gnN^-}^{\mathrm{PS}}(h)\right)d\mu_{gN^+}^{\mathrm{PS}}(n).
\end{equation*}
\end{lem}
\begin{proof} 
By the identities \eqref{eq.basic0} and \eqref{eq.basic1}, we have $gnha^-=gnh^-$, $gnha^+=gn^+$, and
\begin{align*}
 \beta_{gnh^\pm}(e,gnha)&= \beta_{gnh^\pm}(e,gnh)+\beta_{gnh^\pm}(gnh,gnha)\\
&=\beta_{gnh^\pm}(e,gnh)+\beta_{e^\pm}(e,a).
\end{align*}
Note that $\beta_{e^-}(e,a)=-\op{i}\log a$ and
$\beta_{e^+}(e,a)=\log a$.
Hence
\begin{align*}
&d\tilde{m}^{\mathrm{BMS}}(gnha)\\
&=e^{\psi( \op{i}(\beta_{gnh^-}(e,gnha)+\beta_{gnh^+}(e,gnha)))}d\nu_{\op{i}}(gnh^-)d\nu(gn^+)\,d\ell(\beta_{gnh^-}(e,gnha))\\
&\begin{multlined}[t]
=e^{\psi( \op{i}\beta_{gnh^-}(e,gnh)-\log a+\beta_{gnh^+}(e,gnh)+\log a)}
\\
\times d\nu_{\op{i}}(gnh^-)\, d\nu(gn^+)\,d\ell(\beta_{gnh^-}(e,gnh)+\log a)
\end{multlined}\\
&=da\,d\mu_{gnN^-}^{\mathrm{PS}}(h)\,d\mu_{gN^+}^{\mathrm{PS}}(n).
\end{align*}
Hence for $nham\in N^+N^-AM$,
\begin{align*}
d\tilde{m}^{\mathrm{BMS}}(gnham)=dm\,d\tilde{m}^{\mathrm{BMS}}(gnha)=dm\,da\,d\mu_{gnN^-}^{\mathrm{PS}}(h)d\mu_{gN^+}^{\mathrm{PS}}(n),
\end{align*}
proving the claim.
\end{proof}
In a similar manner, one can decompose the BMS measure according to $gP^-N^+$:
\begin{lem}\label{PNdecmop}
For $g\in G$, $f\in C_c(gP^-N^+)$, and $hamn\in N^-AMN^+$,
\begin{equation*}
\tilde{m}^{\mathrm{BMS}}(f)=\int_{N^-AM}\left(\int_{N^+}f(ghamn)\,d\mu_{ghamN^+}^{\mathrm{PS}}(n)\right)e^{-\psi(\log a)}\,dm\,da\,d\mu_{gN^-}^{\mathrm{PS}}(h).
\end{equation*}
\end{lem}
\begin{proof}
For each $m\in M$, consider the change of variable $n_0=mnm^{-1}$.
Then for $hamn\in N^-AMN^+$, we have
\begin{equation*}
d\tilde m^{\BMS}(ghamn)=d\tilde m^{\BMS}(ghan_0m)=dm\,d\tilde m^{\BMS}(ghan_0).
\end{equation*}
By the identity \eqref{eq.basic1}, we have $\beta_{gh^-}(e,ghan_0)=\beta_{gh^-}(e,gh)-\i\log a$ and
$$ \op{i}(\beta_{gh^-}(e,ghan_0))+\beta_{ghan_0^+}(e,ghan_0)= \op{i}(\beta_{gh^-}(e,gh)-\i\log a)+\beta_{ghan_0^+}(e,ghan_0).
$$
It follows that
\begin{align*}
&d\tilde{m}^{\mathrm{BMS}}(ghan_0)\\
&=e^{\psi(\beta_{ghan_0^+}(e,ghan_0))}d\nu(ghan_0^+)\,e^{-\psi(\log a)}\,da\,e^{\psi\circ \op{i}(\beta_{gh^-}(e,gh))}d\nu_{\op{i}}(gh^-)\\
&=d\mu_{ghamN^+}^{\PS}(n)\,e^{-\psi(\log a)}\,da\,d\mu_{gN^-}^{\PS}(h).
\end{align*}
Hence
\begin{equation*}
d\tilde{m}^{\mathrm{BMS}}(ghamn)=d\mu_{ghamN^+}^{\PS}(n)\,e^{-\psi(\log a)}\,dm\,da\,d\mu_{gN^-}^{\PS}(h),
\end{equation*}
finishing the proof.
\end{proof}
Define, for $ham\in N^-AM$,
$$d\mu_{gP^-}^{\mathrm{PS}}(ham)=e^{-\psi(\log a)}\,dm\,da\,d\mu_{gN^-}^{\mathrm{PS}}(h).$$ 
This also allows us to succinctly rewrite the decomposition in Lemma \ref{PNdecmop} as follows:
for any $ f\in C_c(gP^-N^+)$,
\begin{equation}\label{PNgood}
\tilde{m}^{\mathrm{BMS}}(f)=\int_{P^-}\int_{N^+} f(gp n)\,d\mu_{gpN^+}^{\mathrm{PS}}(n)\,d\mu_{gP^-}^{\mathrm{PS}}(p).
\end{equation}

\noindent\textbf{Lebesgue measures on $gN^{\pm}$.} For $g\in G$, 
we note that the Haar measure on $gN^{\pm}$ can be given as follows:
for $n\in N^{\pm}$, 
\begin{equation}\label{eq.Leb}
d\mu_{gN^-}^{\mathrm{Leb}}(n)=
e^{2\rho (\beta_{(gn)^-}(e, gn)) }dm_o((gn)^-),
\end{equation}
and
\be \label{eq.Leb2} d\mu_{gN^+}^{\mathrm{Leb}}(n)=e^{2\rho (\beta_{(gn)^+} (e,gn) ) }dm_o ((gn)^+) .\ee
Using \eqref{gc0}, it can be checked that these are $N^-$ and $N^+$ invariant measures respectively.
Moreover, $d\mu_{gN^{\pm}}^{\mathrm{Leb}}$ does not depend on $g\in G$, so we simply write $dn$.

\subsection{Decomposition of $m^{\BR}$}
Similarly to Lemma \ref{PNdecmop},
 the $\BR$ and $\BR_*$ measures can be decomposed in terms of the $gP^-N^+$ decomposition of $G$:

For all $f\in C_c(g P^-N^+)$,
\begin{equation}\label{BRPN}
\tilde{m}^{\mathrm{BR}}(f)=\int_{P^-}\int_{N^+} f(g pn)\, dn\, d\mu_{gP^-}^{\mathrm{PS}}(p);
\end{equation}
$$
\tilde{m}^{\mathrm{BR}_*}(f)=\int_{P^-}\int_{N^+}f(ghamn)\,d\mu_{ g hamN^+ }^{\mathrm{PS}}(n)e^{-2\rho(\log a)}\,dm \,da\, dh.
$$

We also have the following description of the BR-measures $\tilde m_{\nu_\psi }^{\BR}$ and $\tilde m_{\nu_\psi }^{\BR_*}$ on $G/M$: \begin{lem}\label{lem.BRD}
\begin{itemize}
    \item 
For $g=k\exp(b)n\in KAN^+$ and $[g]=gM$, we have 
\begin{equation*}
d\tilde m_{\nu_\psi }^{\BR}([g])=e^{-(\psi \circ\i)(b)}\,dn\,d\ell(b)\,d\tilde\nu_{\psi} (k)
\end{equation*}
where $\tilde\nu_\psi$ is given by
 \be\label{km} d\tilde\nu_{\psi}(k):=d\nu_{\psi}(k^-)\,dm.\ee
 \item 
 For $g=k\exp(b)n\in KAN^-$ and $[g]=gM$, we have 
\begin{equation*}
d\tilde m_{\nu_\psi }^{\BR_*}([g])=e^{\psi(b)}\,dn\,d\ell(b)\,d\tilde\nu_{\psi} (k)
\end{equation*}
where $\tilde\nu_\psi$ is given by
 \begin{equation*}
     d\tilde\nu_{\psi}(k):=d\nu_{\psi}(k^+)\,dm.
 \end{equation*} 
 \end{itemize}
\end{lem}
\begin{proof} 
For $g=k\exp(b)n\in KAN^+$, 
we have $\beta_{g^-}(e,g)=-\i (b)$. Since $m_o$ is a $(G, 2\rho)$-conformal measure, we have
\begin{align*}
d\tilde m^{\BR}_{\nu_\psi} ([g])& = e^{2\rho(\beta_{k\exp(b)n^+}(e,k \exp (b)n)) }e^{-(\psi \circ\i)(b)}  dm_o(k\exp(b)n^+) db\, d \nu_{\psi}(k^-)\,dm \\
&=e^{-(\psi \circ\i)(b)}\,dn\,db\, d \tilde \nu_{\psi}(k).
\end{align*}
The second part can be proved similarly.
\end{proof}

\section{\texorpdfstring{$\BMS$}{}-mixing and translates of \texorpdfstring{$\PS$}{}-measures}\label{step1}
 In this section, we fix 
\begin{enumerate}
\item an element $u\in  \L_\Gamma\cap \op{int}\fa^+$,
\item a linear form $\psi\in \dg$ tangent to $\psi_{\Gamma}$ at $u$, 
\item  a $(\Gamma, \psi)$-PS measure  $\nu=\nu_\psi$ on $ \F$, and
\item a $(\Gamma, \psi\circ \i)$-PS measure $\nu_{\i}=\nu_{\psi\circ \i}$  on $\F$.
\end{enumerate}
 As before, we set
$$ m^{\BMS}= m^{\BMS}_{\nu, \nu_{\i}}, \quad \mu^{\PS}_{gN^+}=\mu^{\PS}_{gN^+, \nu}.$$
For all $t \geq 0$ and $v\in \ker \psi$, define $$a(t,v):=\exp(t  u+\sqrt t v)\in A .$$

\begin{Def}\label{lmp} \rm We say that $m^{\BMS}$ satisfies the \emph{local mixing property}  if there exist functions $\Psi : \bb (0,\infty)\to\bb (0,\infty)$
and $J : \ker\psi\to\bb (0,\infty)$ such that
\begin{enumerate}
\item  For all $v\in\ker\psi$ and $f_1,\,f_2\in C_c(\GaG)^M$,
\begin{equation}\label{eq.hypo}
\lim_{t\rightarrow\infty} \Psi(t)\int_{\GaG} f_1\big(x\, a(t,v) \big) f_2(x)  \,dm^{\mathrm{BMS}}(x)=
J(v)\, m^{\mathrm{BMS}}(f_1)\,m^{\mathrm{BMS}}(f_2);
\end{equation}
\item There exists $C=C(f_1,f_2)>0$ such that for all $(t,v)\in(0,\infty)\times\ker\psi$ with $a(t, v)\in A^+$,
\begin{equation*}
\left|\Psi(t)\int_{\GaG} f_1\big(x\, a(t,v) \big) f_2(x)  \,dm^{\mathrm{BMS}}(x)\right|< C.
\end{equation*}
\end{enumerate}
\end{Def}
\medskip
The main goal of this section is to establish the following:
\begin{prop}\label{p1}
Suppose that $m^{\BMS}$ satisfies the local mixing property  for the pair $(\Psi, J)$. 
Then for any $x=[g] \in \GaG$, $v\in\ker\psi$, $f\in C_c(\GaG)^M$, and $\phi\in C_c(N^+)$,
\begin{equation} \label{PSequi} 
\lim_{t\rightarrow\infty} \Psi(t)\int_{N^+} f\big(x n\, a(t,v)\big)\phi(n)\,d\mu_{gN^+}^{\mathrm{PS}}(n)= J(v)\,m^{\mathrm{BMS}}(f)\,\mu_{gN^+}^{\mathrm{PS}}(\phi),
\end{equation}
and there exists $C'=C'(f,\phi)>0$ such that
\begin{equation*}
\left|\Psi(t)\int_{N^+} f\big(x n\, a(t,v) \big)\phi(n)\,d\mu_{gN^+}^{\mathrm{PS}}(n)\right|<C'
\end{equation*}
for all $(t,v)\in(0,\infty)\times\ker\psi$ with $a(t, v)\in A^+$.
\end{prop}

For $\epsilon>0$, let $G_{\epsilon}$ denote the open ball of radius $\epsilon$ around $e$ in $G$. For a subgroup $S<G$, we define $S_{\epsilon}:=S\cap G_{\epsilon}$. The choices $S=P^\pm,\,N^{\pm},\,A$ are the only subgroups we will require. We will carry out a thickening argument using PS measures as in e.g.\ \cite{OhWinter}; the following lemma is needed:

\begin{lem}\label{fullmeas}
For any $g\in G$,
\begin{equation*}
\nu (gN^+(e^+) )>0\quad\text{and} \quad \nu_{ \op{i}}(gN^-(e^-))>0. \end{equation*}
\end{lem} 
\begin{proof}  
The Zariski density of $\Lambda_\Gamma$ in $\F$  is proved in
\cite{Ben}. This also follows from Lemma \ref{Can}.
Since each $gN^\pm(e^{\pm}) $ is a Zariski open subset of $\cal F$ and the support of $\nu$ is equal to $\Lambda_\Ga$, the conclusion follows. \end{proof}

We will also need the following continuity property of the PS-measures \cite[Proposition 2.15]{OhShah}:
\begin{lem}\label{PScont}
 For any fixed $\rho\in C_c(N^{\pm})$ and $g\in G$,
the map $N^{\mp}\to \br$ given by $h \mapsto \mu_{gh N^{\pm}}^{\mathrm{PS}}(\rho)$ is continuous.
\end{lem} 
\begin{proof}  
We will only prove the case when $\rho\in C_c(N^+)$; the other case can be proved similarly.
Define a function $\tilde{\rho}_g : N^-\times G/P\to\bb R$ by
\begin{equation*}
\tilde{\rho}_g(h, \xi):=\begin{cases} \rho(n)e^{\psi (\beta_{ghn^+}(e,ghn))}\qquad&\mathrm{if\;}\xi=ghn^+\text{ for some }n \in N^+,\\0\qquad&\mathrm{otherwise.}
\end{cases}
\end{equation*}
 Since $N^+\cap P=\lbrace e\rbrace$, $\tilde{\rho}_g$ is well-defined. By continuity of the Busemann function, $\tilde \rho_g$ is continuous in $h\in N^-$. This gives
\begin{align*}
\mu_{ghN^+}^{\mathrm{PS}}(\rho)=\int_{N^+} \rho(n)e^{\psi(\beta_{ghn^+}(e,ghn))}\,d\nu(ghn^+)=\int_{G/P} \tilde{\rho}_g(h,\xi)\,d\nu(\xi),
\end{align*}
hence $|\mu_{gh_1N^+}^{\mathrm{PS}}(\rho)-\mu_{gh_2N^+}^{\mathrm{PS}}(\rho) |\leq \max_{\xi\in G/P} |\tilde{\rho}_g(h_1,\xi)-\tilde{\rho}_g(h_2,\xi)|$. The continuity of $\tilde{\rho}_g$ then implies the claimed statement.
\end{proof} 
A function on $N^{\pm}$ is said to be \emph{radial} if it is invariant under conjugation by elements of $M$ i.e. $f(mnm^{-1})=f(n)$ for all $m\in M$, and $n\in N^{\pm}$.
\begin{cor}\label{thickcor}
Given $\epsilon>0$ and $g\in G$, there exist $R>1$ and a non-negative radial function $\rho_{g,\epsilon}\in C_c(N^-_R)$ such that
for all $ n\in N_{\epsilon}^+$, 
\begin{equation*}
\mu_{gnN^-}^{\mathrm{PS}}(\rho_{g,\epsilon})>0.
\end{equation*} 
\end{cor} 
\begin{proof}
For each $j\in \NN$, let $\phi_j\in C_c(N_{j+1}^-)$ be a nonnegative radial function such that $\phi_j|_{N_j^-}=1$.
By Lemma \ref{fullmeas}, for each $n\in N^+$, there exists some $j_{n}\in \NN$ such that $\mu_{gnN^-}^{\mathrm{PS}}(N^-_{j_{n}})>0$.
By Lemma \ref{PScont}, for each $n\in N^+$, there exists $r_{n}>0$ such that
\begin{equation*}
\mu_{gn_0N^-}^{\mathrm{PS}}(N_{j_n}^-)>0 \quad  \text{ for all  } n_0\in \scrB(n):=\lbrace n_0\in N^+\,:\, \dist(n,n_0)<r_{n}\rbrace.
\end{equation*}
Using the relative compactness of $N_{\epsilon}^+$, we choose $n_1,\ldots,  n_k\in N^+$ such that $N_{\epsilon}^+\subset \bigcup_{i=1}^k \scrB(n_i)$.
Choosing $R:=\max(j_{n_1},\cdots,j_{n_k})+1$ and $\rho_{g,\epsilon}:=\phi_{R-1}$ completes the proof.
\end{proof}

\noindent{\bf Proof of Proposition \ref{p1}.}
Fixing $v\in \op{ker}(\psi)$, for simplicity, we denote $a_t=a(t, v)$.
Let $x=[g]$, and $\epsilon_0>0$ be such that $\phi\in C_c(N_{\epsilon_0}^+)$. 
By Corollary \ref{thickcor}, there exist $R>0$ and a nonnegative $\rho_{g,\epsilon_0}\in C_c(N^-_R)$ such that
\begin{equation*}
\mu_{gnN^-}^{\mathrm{PS}}(\rho_{g,\epsilon_0})>0\quad \text{ for all  }n\in N_{\epsilon_0}^+.
\end{equation*} 
Given arbitrary $\epsilon>0$, choose a non-negative function  $q_{\epsilon}\in C_c(A_{\epsilon})$ satisfying $\int_{A}q_{\epsilon}(a)\,da=1$.
Then 
\begin{align}\label{fthickening}
&\int_{N^+} f(x n a_t)\phi(n)\,d\mu_{gN^+}^{\mathrm{PS}}(n)
\\\notag&=\int_{N^+} f(xn a_t )\phi(n)\left( \tfrac{1}{\mu_{gnN^-}^{\mathrm{PS}}(\rho_{g,\epsilon_0})}\int_{N^-A} \rho_{g,\epsilon_0}(h)q_{\epsilon}(a)\,da\,d\mu_{gnN^-}^{\mathrm{PS}}(h)\!\right)d\mu_{gN^+}^{\mathrm{PS}}(n)
\\\notag&=\int_{N^+} \left(\int_{N^-A}f(xna_t ) \tfrac{\phi(n)\rho_{g,\epsilon_0}(h)q_{\epsilon}(a)}{\mu_{gnN^-}^{\mathrm{PS}}(\rho_{g,\epsilon_0})} \,da\,d\mu_{gnN^-}^{\mathrm{PS}}(h)\!\right)d\mu_{gN^+}^{\mathrm{PS}}(n).
\end{align}
We now define a right $M$-invariant function $\tilde \Phi_{\epsilon}\in C_c(gN_{\epsilon_0}^+N_R^-A_{\epsilon}M)\subset C_c(G)$ by
\begin{equation*}
\tilde \Phi_{\epsilon}(g_0):=\begin{cases}\frac{\phi(n)\rho_{g,\epsilon_0}(h)q_{\epsilon}(a)}{\mu_{gnN^-}^{\mathrm{PS}}(\rho_{g,\epsilon_0})}\qquad&\mathrm{if\;}g_0=gnham,\\0\qquad&\mathrm{otherwise.}
\end{cases}
\end{equation*}
Note that the continuity of $\tilde  \Phi_{\epsilon}$ is a consequence of Lemma \ref{PScont}.
Also observe that $\tilde  \Phi_{\epsilon}$ depends on our choice of representative for $x=[g]$.

We now assume without loss of generality $f\geq 0$ and define, for all $\epsilon>0$, functions $f_{\epsilon}^{\pm}$ as follows:  for all $ z\in \GaG$,
$$
f_{\epsilon}^+(z):= \sup_{b\in N_{\epsilon}^+N^-_{\epsilon}A_{\epsilon}} f(zb)\;\;\text{and} \;\;
f_{\epsilon}^-(z):= \inf_{b\in N_{\epsilon}^+N^-_{\epsilon}A_{\epsilon}} f(zb).$$

Since $u\in\inte\fa^+$, for every $\epsilon>0$, there exists $t_0(R,\epsilon)> 0$ such that
\begin{equation*}
{a_t^{-1}}  N_{R}^- a_t
 \subset N_{\epsilon}^-\qquad \text{ for all }t\geq 
 { t_0(R,\epsilon)}. 
\end{equation*}
Then, as $\mathrm{supp}(\tilde  \Phi_{\epsilon})\subset gN_{\epsilon_0}^+ N_{R}^- A_{\epsilon}M$, we have
\begin{align}\label{fplus}
f(xn{a_t} )\tilde  \Phi_\epsilon(gnha) \leq f_{2\epsilon}^+ (x nh a a_t )\tilde  \Phi_\epsilon(gnha)
\end{align}
for all $nh a\in N^+N^-A$ and $t\geq t_0(R,\epsilon)$.
Similarly,
\begin{equation*}
 f_{2\epsilon}^- (xnh a a_t )\tilde  \Phi_\epsilon(g nha)\leq f(x n a_t )\tilde  \Phi_\epsilon(g nha).
\end{equation*}
We now use $f_{2\epsilon}^+$ to give an upper bound on the limit we are interested in; $f_{2\epsilon}^-$ is used in an analogous way to provide a lower bound. Entering the definition of $\Phi_{\epsilon}$ and the above inequality \eqref{fplus} into \eqref{fthickening} gives
\begin{align*}
&\limsup_{t\rightarrow\infty}\, \Psi(t)\int_{N^+} f(xn a_t )\phi(n)\,d\mu_{gN^+}^{\mathrm{PS}}(n)
\\&\leq \limsup_{t\rightarrow\infty} \,\Psi(t)\int_{N^+} \left(\int_{N^-AM}f_{2\epsilon}^+(x nha a_t ) \tilde \Phi_{\epsilon}(g nha)dm\,da\,d\mu_{gnN^-}^{\mathrm{PS}}(h)\right)\,d\mu_{gN^+}^{\mathrm{PS}}(n)
\\&=\limsup_{t\rightarrow\infty}\, \Psi(t)\int_{G}f_{2\epsilon}^+([g_0] a_t ) \tilde  \Phi_{\epsilon}(g_0)\,d\tilde{m}^{\mathrm{BMS}}(g_0)
\\&=\limsup_{t\rightarrow\infty}\, \Psi(t)\int_{\GaG} f_{2\epsilon}^+([g_0] a_t ) {\Phi}_{\epsilon}([g_0])\,dm^{\mathrm{BMS}}([g_0]),
\end{align*}
where
\begin{equation*}
 {\Phi}_{\epsilon}([g_0]):=\sum_{\gamma\in \Gamma} \Phi_{\epsilon}(\gamma g_0),
\end{equation*}
and Lemma \ref{NPdecomp} was used in the second to last line of the above calculation. By the standing assumption \eqref{eq.hypo}, we then have
\begin{align*}
\limsup_{t\rightarrow\infty}\, \Psi(t)&\int_{N^-} f(x n a_t )\phi(n)\,d\mu_{gN^+}^{\mathrm{PS}}(n) 
\\&\leq  J(v)\, \,m^{\mathrm{BMS}}(f_{2\epsilon}^+)m^{\mathrm{BMS}}({\Phi}_{\epsilon})=J(v)\, \,m^{\mathrm{BMS}}(f_{2\epsilon}^+)\tilde{m}^{\mathrm{BMS}}( \tilde \Phi_{\epsilon}).
\end{align*}
Using Lemma \ref{NPdecomp} and the $M$-invariance of $ \tilde \Phi_\epsilon$, we have
\begin{align*}
\tilde{m}^{\mathrm{BMS}}( \tilde  \Phi_{\epsilon})&=\int_{N^+} \left(\int_{N^-A}\tilde  \Phi_{\epsilon}(g nha)\,da\,d\mu_{gnN^-}^{\mathrm{PS}}(h)\right)\,d\mu_{gN^+}^{\mathrm{PS}}(n)
\\&=\int_{N^+} \frac{\phi(n)}{\mu_{gnN^-}^{\mathrm{PS}}(\rho_{g,\epsilon_0})}\left(\int_{N^-A} \rho_{g,\epsilon_0}(h)q_{\epsilon}(a)\,da\,d\mu_{gnN^-}^{\mathrm{PS}}(h)\!\right)\,d\mu_{gN^+}^{\mathrm{PS}}(n)
\\&= \mu_{gN^+}^{\mathrm{PS}}(\phi).
\end{align*}
Since $\epsilon>0$ was arbitrary, taking $\epsilon\rightarrow0$ gives
\begin{equation*}
\limsup_{t\rightarrow\infty}\, \Psi(t)\int_{N^-} f(x n a_t )\phi(n)\,d\mu_{gN^+}^{\mathrm{PS}}(n) \leq 
J(v)\,
m^{\mathrm{BMS}}(f)\,\mu_{gN^+}^{\mathrm{PS}}(\phi).
\end{equation*}
The lower bound given by replacing $f^+_{2\epsilon}$ with $f^-_{2\epsilon}$ in the above calculations proves the 
first statement.

For the second claim of the proposition, observe that if $tu+\sqrt tv\in\mathfrak a^+$, then
$$f(x n a_t)\tilde  \Phi_\epsilon(gnha)\leq f_{R+\e}^+(xnha a_t)\tilde  \Phi_\e(gnha),$$
as in \eqref{fplus}.
Hence
$$\Psi(t)\int_{N^+} f(xn a_t)\phi(n)\,d\mu_{gN^+}^{\mathrm{PS}}(n)\leq C(f_{R+\e}^+,\Phi_\e).$$
Choosing $C'(f,\phi):=C(f_{R+\e}^+,\Phi_\e)$ finishes the proof.$\qed$

\section{Translates of Lebesgue measures and Haar mixing}\label{s2}
We continue with the setup of Section \ref{step1}: recall that we have fixed $u\in  \L_\Gamma\cap \op{int}\fa^+$, a linear form $\psi\in \dg$ such that $\psi(u)=\psi_\Gamma (u)$,  a $(\Gamma, \psi)$-PS measure $\nu=\nu_{\psi}$ and
a $(\Gamma, \psi\circ\i )$-PS measure $\nu_{\i}=\nu_{\psi\circ\i}$ on $\F$. We set
$$ m^{\BMS}= m^{\BMS}_{\nu,\nu_\i},m^{\BR}= m^{\BR}_{\nu_\i}, m^{\BR_*}= m^{\BR_*}_{\nu} ,\;\; \text{and} \; \; \mu^{\PS}_{gN^+}=\mu^{\PS}_{gN^+, \nu}.$$
The main goal in this section is to prove a local mixing statement for the Haar measure on $\GaG$. In order to do this, we first convert equidistribution of translates of $\mu_{gN^+}^{\PS}$ (Proposition \ref{p1}) into equidistribution of translates of the Lebesgue measure on $xN^+$:
\begin{prop}\label{p2} Suppose that $m^{\BMS}$ satisfies the local mixing property  for the pair $(\Psi, J)$. 
Then for any $x=[g] \in \GaG$, $v\in\ker\psi$, $f\in C_c(\GaG)^M$, and $\phi\in C_c(N^+)$,
\begin{equation*}
\lim_{t\rightarrow\infty} \Psi(t)
e^{(2\rho-\psi)(t  u+\sqrt t v)}
\int_{N^+}f\big(x n \,a(t,v)\big)\phi(n)\,dn=  J(v)\,
m^{\mathrm{BR}}(f)\, \mu_{ gN^+}^{\mathrm{PS}}(\phi),
\end{equation*}
and there exists $C''=C''(f,\phi)>0$ such that 
\begin{equation*}
\left| \Psi(t)e^{
(2\rho-\psi)(t  u+\sqrt t v)}\int_{N^+}f\big(x n \,a(t,v)\big)\phi(n)\,dn\right|<C''
\end{equation*}
for all $(t,v)\in(0,\infty)\times\ker\psi$ with $a(t, v)\in A^+$.
\end{prop}

\begin{proof} For $\epsilon_0>0$, set  $\scrB_{\epsilon_0}=N_{\epsilon_0}^-A_{\epsilon_0}MN_{\epsilon_0}^+$.
Note that $MN_{\epsilon_0}^+=N_{\epsilon_0}^+M$ by the choice of the invariant metric on $G$. Given $x_0\in \GaG$, let $\epsilon_0(x_0) $ denote the maximum number $r$ such that the map $G\rightarrow\GaG$ given by $h\mapsto x_0 h$ for $h\in G$ is injective on $\scrB_{r}$.
 Note that $\psi(u)=\psi_\Gamma (u)=\psi_\Gamma \big(\i(u)\big).$
Fixing $v\in\ker\psi$, we set ,for all $t\in \br$, $$a_t:=a(t,v).$$
By using a partition of unity if necessary, it suffices to prove that for any $x_0\in \GaG$ and $\e_0=\e_0(x_0)$, the claims of the proposition hold for any non-negative $f\in C(x_0 \scrB_{\epsilon_0})^M$, non-negative $\phi\in C(N_{\epsilon_0}^+)$, and $x=[g]\in x_0 \scrB_{\epsilon_0}$, that is,
\begin{equation}\label{Haarequi}
\lim_{t\rightarrow\infty} \Psi(t)
e^{(2\rho-\psi)(\log a_t)}\int_{N^+}f(x n a_t)\phi(n)\,dn= 
J(v)\,
m^{\mathrm{BR}}(f)\, \mu_{ g N^+}^{\mathrm{PS}}(\phi),
\end{equation}
and for $\log a_t\in\mathfrak a^+$, we have
\begin{equation*}
 \Psi(t)e^{(2\rho-\psi)(\log a_t)}\int_{N^+}f(x n a_t)\phi(n)\,dn<C'',
\quad\text{for some $C''=C''(f,\phi)$.}\end{equation*}
Moreover, we may assume that $f$ is given as
\begin{equation*}
f([g])=\sum_{\gamma \in \Gamma} \tilde{f} (\gamma g)\qquad \text{ for all }  g\in G,
\end{equation*}
for some non-negative $\tilde{f}\in C_c(g_0\scrB_{\epsilon_0})\subset C_c(G)$. Note that for $x=[g] \in  [g_0] \scrB_{\epsilon_0}$,
\begin{align}\label{eq.un1}
\int_{N^+} f([g]n a_t)\phi(n)\,dn=\sum_{\gamma \in \Gamma} \int_{N^+} \tilde{f}(\gamma g n a_t)\phi(n)\,dn.
\end{align}
Note that $\tilde{f}(\gamma g n a_t)=0$ unless $\gamma g n a_t
\in g_0 \scrB_{\epsilon_0}$.
Together with the fact that $\mathrm{supp}(\phi)\subset N_{\epsilon_0}^+$, it follows that the summands in \eqref{eq.un1} are non-zero for only finitely many elements $\gamma \in \Gamma \cap  g_0 \scrB_{\epsilon_0}
a_t^{-1} N_{\epsilon_0}^+ g^{-1}$.

Suppose $\gamma g N_{\epsilon_0}^+ a_t
\cap g_0 \scrB_{\epsilon_0}\neq\emptyset$. 
Then $\gamma g a_t\in g_0 N^-_{\epsilon_0}A_{\epsilon_0}MN^+$, and there are unique elements
$p_{t,\gamma}\in N^-_{\epsilon_0}A_{\epsilon_0}M$ and $n_{t,\gamma}\in N^+$ such that
\begin{equation*}
\gamma g a_t = g_0 p_{t,\gamma}n_{t,\gamma} \in g_0 P^{-}_{\epsilon_0} N^+.
\end{equation*}
Let ${\Gamma_{t,v}}$ denote the subset $\Gamma \cap g_0 (N^-_{\epsilon_0}A_{\epsilon_0}MN^+ )a_t^{-1}
 g^{-1}$.
Note that although $
\Gamma_{t,v}$ may possibly be infinite, only finitely many of the terms in the sums we consider will be non-zero.
This gives
\begin{align*}
&\int_{N^+} f([g]n a_t)\phi(n)\,dn=\sum_{\gamma \in \Gamma} \int_{N^+} \tilde{f}(\gamma g n a_t)\phi(n)\,dn\\
&=\sum_{\gamma \in\Gamma_{t,v} } \int_{N^+} \tilde{f}(\ga g a_t(a_t^{-1} n a_t)\big)\phi(n)\,dn\\
&=e^{-2\rho(\log a_t)} \sum_{\gamma \in 
\Gamma_{t,v}} \int_{N^+} \tilde{f}(\gamma g a_t n)\phi(a_tn a_t^{-1})\,dn \\ &= e^{-2\rho(\log a_t)}
\sum_{\gamma \in  \Gamma_{t,v}} \int_{N^+} \tilde{f}\big(g_0p_{t,\gamma}n_{t,\gamma} n\big)\phi(a_t n a_t^{-1})\,dn \\
&={e^{-2\rho(\log a_t)}} \sum_{\gamma \in  {\Gamma_{t,v} }
} \int_{N^+} \tilde{f}\big(g_0p_{t,\gamma} n\big)\phi\big( a_t \,n_{t,\gamma}^{-1}n\; a_t^{-1}\big)\,dn.\end{align*}
Since $\supp(\tilde{f})\subset g_0 \scrB_{\epsilon_0}$, we have 
\begin{align*}
 \sum_{\gamma \in {\Gamma_{t,v}}}& \int_{N^+} \tilde{f}\big(g_0p_{t,\gamma} n\big)\phi\big( a_t \,n_{t,\gamma}^{-1}n\; a_t^{-1}\big)\,dn
\\\leq&\sum_{\gamma \in  \Gamma_{t,v} } \left(\sup_{n\in N_{\epsilon_0}^+} \phi\big(
a_t \,n_{t,\gamma}^{-1}\;a_t^{-1}(a_t na_t^{-1}) \big) \right)\cdot \int_{N^+} \tilde{f}\big(g_0p_{t,\gamma} n\big)\,dn,\end{align*}
and 
\begin{align*}
\sum_{\gamma \in  \Gamma_{t,v} }& \int_{N^+} \tilde{f}\big(g_0p_{t,\gamma} n\big)\phi\big(a_t\,n_{t,\gamma}^{-1}n\; a_t^{-1} \big)\,dn
 \\ &\geq\sum_{\gamma \in \Gamma_{t,v}
} \left(\inf_{n\in N_{\epsilon_0}^+} \phi\big(
a_t\,n_{t,\gamma}^{-1}\;a_t^{-1} (a_t na_t^{-1})
\big) \right)
\cdot \int_{N^+} \tilde{f}\big(g_0p_{t,\gamma} n\big)\,dn.
\end{align*}
Since $u$ belongs to $\op{int}\scrL_{\Gamma}$, there exist $t_0(v)>0$ and $\alpha>0$
such that
\begin{equation*}
a_t N_r^+ a_t^{-1}\subset N_{r e^{-\alpha t}}^+\qquad \text{ for all }\, r>0\text{ and }t>t_0(v).
\end{equation*}
Therefore, for all $n\in N_{\epsilon_0}^+$
and $t>t_0(v)$,
we have
\begin{align}\label{eq.x1}
 \phi_{\epsilon_0 e^{-\alpha t}}^-\big( a_t\,n_{t,\gamma}^{-1}a_t^{-1}  \big)&\leq  \phi\big(a_t\,n_{t,\gamma}^{-1}a_t^{-1}(a_t na_t^{-1}) \big)\leq  \phi_{\epsilon_0 e^{-\alpha t}}^+\big(a_t\,n_{t,\gamma}^{-1}\;a_t^{-1} \big),
\end{align}
where
\begin{equation*}
\phi^+_{\epsilon}(n):=\sup_{b\in N_{\epsilon}^+} \phi(nb),\text{ and } \phi^-_{\epsilon}(n):=\inf_{b\in N_{\epsilon}^+} \phi(nb) \qquad \text{ for all } n\in N^+,\,\epsilon>0.
\end{equation*}
We now have the following chain of inequalities (for $t>t_0(v)$):
\begin{align}\label{intchain}
&\sum_{\gamma \in \Gamma_{t,v}} \phi_{\epsilon_0 e^{-\alpha t}}^-\big(a_t\,n_{t,\gamma}^{-1}\;a_t^{-1}
\big)\int_{N_{\epsilon_0}^+} \tilde{f}\big(g_0p_{t,\gamma} n\big)\,dn\\
\notag &\leq e^{2\rho(\log a_t)}
 \,\int_{N^+} f([g] n a_t)\phi(n)\,dn
 \\&\leq\sum_{\gamma \in \Gamma_{t,v}} \phi_{\epsilon_0 e^{-\alpha t}}^+\big(a_t\,n_{t,\gamma}^{-1}\;a_t^{-1}
\big)\int_{N_{\epsilon_0}^+} \tilde{f}\big(g_0p_{t,\gamma} n\big)\,dn.
\end{align}
By Lemmas \ref{fullmeas} and \ref{PScont}, there exist $R>0$ and a radial function $\rho\in C_c(N_R^+)$ such that $\rho(n)\geq 0$ for all $n\in N^+$, and $\mu_{ g_0pN^+}^{\mathrm{PS}}(\rho)>0$ for all $p\in N_{\epsilon_0}^-A_{\epsilon_0}M$.
Define $\tilde{F}\in C_c(g_0 N_{\epsilon_0}^-A_{\epsilon_0}MN_R^+)$ by
\begin{equation*}
\tilde{F}(g):=\begin{cases} \frac{\rho(n)}{\mu_{ g_0pN^+}^{\mathrm{PS}}(\rho)}\int_{N_{\epsilon_0}^+} \tilde{f}\big(g_0p v\big)\,dv&\qquad\mathrm{if\;}g=g_0pn\in g_0 N_{\epsilon_0}^-A_{\epsilon_0}MN_R^+,\\0&\qquad\mathrm{otherwise.}
\end{cases}
\end{equation*}
Since $\rho$ is radial, $\tilde F$ is right $M$-invariant.
The key property of $\tilde{F}$ we will use is the following: for all $p\in P_{\epsilon_0}^-$,
\begin{equation*}
\int_{N^+} \tilde{F} (g_0pn)\,d\mu_{g_0pN^+}^{\mathrm{PS}}(n)=\int_{N_R^+} \tilde{F} (g_0pn)\,d\mu_{g_0pN^+}^{\mathrm{PS}}(n)=\int_{N_{\epsilon_0}^+}\tilde{f}(g_0 pn)\,dn.
\end{equation*}
Returning to \eqref{intchain}, we now give an upper bound for $\int_{N} f([g] n 
a_t
)\phi(n)\,dn$; the lower bound can be dealt with in a similar fashion.
We observe: \begin{align*}
& e^{2\rho(\log a_t)} \int_{N^+} f([g] n a_t)\phi(n)\,dn\\
\leq&\,  
\sum_{\gamma \in \Gamma_{t,v}
} \phi_{\epsilon_0 e^{-\alpha t}}^+\big(a_t
\,n_{t,\gamma}^{-1}\;a_t^{-1}
\big)\int_{N_{\epsilon_0}^+} \tilde{f}\big(g_0p_{t,\gamma} n\big)\,dn\\
=&\,\sum_{\gamma \in \Gamma_{t,v}} \phi_{\epsilon_0 e^{-\alpha t}}^+\big(a_t\,n_{t,\gamma}^{-1}\;a_t^{-1}\big)
\int_{N_R^+} \tilde{F} (g_0p_{t,\gamma}  n)\,d\mu_{g_0p_{t,\gamma}N^+}^{\mathrm{PS}}(n)
\\=&\,
\sum_{\gamma \in \Gamma_{t,v}
} \int_{N_R^+} \tilde{F} (g_0p_{t,\gamma}  n)\phi_{\epsilon_0 e^{-\alpha t}}^+\big(a_t\,n_{t,\gamma}^{-1}\;a_t^{-1}
\big)\,d\mu_{g_0p_{t,\gamma}N^+}^{\mathrm{PS}}(n).
\end{align*}
Similarly as before, we have, for all $t>t_0(v)$ and $\,n\in N_R^+$,
\begin{align}\label{eq.x2}
\phi_{\epsilon_0 e^{-\alpha t}}^+\big(a_t
\,n_{t,\gamma}^{-1}\;a_t^{-1}
\big)&=\phi_{\epsilon_0 e^{-\alpha t}}^+\big(a_t\,n_{t,\gamma}^{-1}n(n)^{-1}\;a_t^{-1}\big)\notag\\
&\leq \phi_{(R+\epsilon_0) e^{-\alpha t}}^+\big(a_t
\,n_{t,\gamma}^{-1}n\;a_t^{-1}\big).
\end{align}
Hence \eqref{intchain} is bounded above by
\begin{align*}
&\leq 
\begin{multlined}[t]
\sum_{\gamma \in \Gamma_{t,v}
} \int_{N_R^+} \tilde{F} (g_0p_{t,\gamma}  n)\phi_{(R+\epsilon_0) e^{-\alpha t}}^+\big(a_t\,n_{t,\gamma}^{-1}n\;a_t^{-1}
\big)\,d\mu_{g_0p_{t,\gamma}N ^+}^{\mathrm{PS}}(n)
\end{multlined}
\\&=
\begin{multlined}[t]
\sum_{\gamma \in \Gamma_{t,v}
} \int_{N^+} \tilde{F} \big(g_0p_{t,\gamma} n_{t,\gamma}a_t^{-1}na_t\big)\phi_{(R+\epsilon_0) e^{-\alpha t}}^+(n) \,d\mu_{g_0p_{t,\gamma}N^+}^{\mathrm{PS}}(n_{t,\gamma}a_t^{-1}na_t).
\end{multlined}
\end{align*}
By Lemma \ref{MANconj},
\begin{equation*}
d\mu_{g_0p_{t,\gamma}N^+}^{\mathrm{PS}}(n_{t,\gamma}a_t^{-1}na_t)=e^{-\psi(\log  a_t^{-1})}d\mu_{g_0p_{t,\gamma} n_{t,\gamma}a_t^{-1}N^+}^{\mathrm{PS}}(n).
\end{equation*}
Since  $g_0p_{t,\gamma}n_{t,\gamma}a_t^{-1}=\gamma g$, it follows that for all $t>t_0(v)$,
\begin{align*}
 & e^{(2\rho-\psi)(\log a_t)} \int_{N^+} f([g] n a_t)\phi(n)\,dn\,  \\ \le 
  &\sum_{\gamma \in \Gamma_{t,v}} \!\int_{N^+} \tilde{F} (\gamma g na_t )\phi_{(R+\epsilon_0) e^{-\alpha t}}^+(n)\,d\mu_{\gamma gN^+}^{\mathrm{PS}}(n)\\
\leq&\, \int_{N^+}\left(\sum_{\gamma \in \Gamma} \tilde{F} (\gamma g na_t )\right)\phi_{(R+\epsilon_0) e^{-\alpha t}}^+(n)\,d\mu_{gN^+}^{\mathrm{PS}}(n).
\end{align*}
Define a function $F$ on $\GaG$ by
\begin{equation*}
F([g]):= \sum_{\gamma\in\Gamma} \tilde{F}(\gamma g).
\end{equation*}
Then for any $\epsilon>0$ and for all $t>t_0(v)$ such that $(R+\epsilon_0)e^{-\alpha t} \leq \epsilon$,
\begin{multline*} \Psi(t)\int_{N^+} F([g] n a_t)\phi^-_{\epsilon}(n)\,d\mu_{ gN^+}^{\mathrm{PS}}(n)
\leq\Psi(t) e^{(2\rho-\psi)(\log a_t)}\int_{N^+}f([g] n a_t)\phi(n)\,dn
\\ \leq \, \Psi(t)\int_{N^+} F([g]n  a_t )\phi^+_{\epsilon}(n)\,d\mu_{gN^+}^{\mathrm{PS}}(n).
\end{multline*}
Since $F$ is right $M$-invariant, by Proposition \ref{p1}, letting $\epsilon\rightarrow0$ gives
\begin{align*}
\lim_{t\rightarrow\infty} \Psi(t) e^{(2\rho-\psi)(\log a_t)}\int_{N^+}f([g]n  a_t )\phi(n)\,dn=J(v)\,
m^{\mathrm{BMS}}(F)\, \mu_{gN^+}^{\mathrm{PS}}(\phi).
\end{align*}
From the definition of $F$, together with Lemma \ref{PNdecmop} in the form \eqref{PNgood}, and \eqref{BRPN}, we have
\begin{align*}
&m^{\mathrm{BMS}}(F)=\tilde{m}^{\mathrm{BMS}}(\tilde{F})\\
&=\int_{ P^-} \left(\int_{N^+} \tilde{F}(g_0pn)\,d\mu_{g_0 pN^+}^{\mathrm{PS},}(n)\right)\,d\mu_{g_0P^-}^{\mathrm{PS}}(p)
\\&=\int_{ P^-} \left(\int_{N^+} \tilde{f}(g_0pn)\,dn\right)\,d\mu_{g_0P^-}^{\mathrm{PS}}(p)\\
&=\tilde{m}^{\mathrm{BR}}(\tilde{f})=m^{\mathrm{BR}}(f).
\end{align*}

This finishes the proof of the first statement.
For the second statement, note that the following inequalities corresponding to \eqref{eq.x1}, and \eqref{eq.x2} hold  with the weaker assumption $tu+\sqrt tv\in\mathfrak a^+$, rather than $t>t_0(v)$:
for all $n\in N_{\e_0}^+$,
$$ \phi\big(a_t
\,n_{t,\gamma}^{-1}\;a_t^{-1}
 (a_t na_t^{-1})
\big)\leq  \phi_{\epsilon_0}^+\big(a_t\,n_{t,\gamma}^{-1}\;a_t^{-1} \big),$$
and for all $n\in N_{R}^+$,
\begin{align*}
\phi_{\epsilon_0 }^+\big(a_t\,n_{t,\gamma}^{-1}\;a_t^{-1}\big)
=\phi_{\epsilon_0 }^+\big(a_t\,n_{t,\gamma}^{-1}n(n)^{-1}\;a_t^{-1}\big)\leq \phi_{R+\epsilon_0}^+\big(a_t\,n_{t,\gamma}^{-1}n\;a_t^{-1}\big).
\end{align*}
Now proceeding similarly as in the proof of the first statement, we have
\begin{align*}
&\int_{N^+} f([g] n a_t)\phi(n)\,dn\\
&\leq\,
 e^{(\psi-2\rho)(\log a_t)}
\int_{N^+}\left(\sum_{\gamma \in \Gamma} \tilde{F} (\gamma g na_t )\right)\phi_{R+\epsilon_0}^+(n)\,d\mu_{gN^+}^{\mathrm{PS}}(n),
\end{align*}
and hence
\begin{align*}
&\Psi(t) e^{(2\rho-\psi)(\log a_t)}\int_{N^+}f([g] n a_t)\phi(n)\,dn\\
& \leq \, \Psi(t)\int_{N^+} F([g]n a_t)\phi^+_{R+\epsilon_0}(n)\,d\mu_{gN^+}^{\mathrm{PS}}(n)\leq C'(F,\phi_{R+\e_0}^+),
\end{align*}
provided $\log a_t\in\mathfrak a^+$.
By setting $C''(f, \phi):=C'(F,\phi_{R+\e_0}^+)$,  this finishes the proof of the proposition.
\end{proof}

With the help of Proposition \ref{p1}, we are now ready to prove:
\begin{prop}\label{prop.mixH0}
 Suppose that $m^{\BMS}$ satisfies the local mixing property  for the pair $(\Psi, J)$. 
Then for any $f_1, f_2\in C_c(\Gamma\ba G)^M$ and $v\in\ker\psi$, we have
\begin{align*}
\lim_{t\rightarrow\infty} \Psi(t)e^{(2\rho-\psi)(t  u+\sqrt t v)} &\int_{\GaG} f_1 (x a(t,v))f_2(x)\,dx
=J(v)\,m^{\mathrm{BR}}(f_1)\, m^{\mathrm{BR}_*}(f_2),
\end{align*}
and there exists $C_0=C_0(f_1,f_2)>0$ such that if $a(t,v)\in\mathfrak a^+$,
$$
\left| \Psi(t)
e^{(2\rho-\psi)(t  u+\sqrt t v)}  \int_{\GaG} f_1 (x a(t,v))f_2(x)\,dx\right| <C_0.
$$
\end{prop}
\begin{proof}
Note that the hypotheses above coincide with those of Propositions \ref{p1} and \ref{p2}; this allows us to apply Proposition \ref{p2} in the following argument.

By compactness, we can find $\epsilon_0>0$ and $x_i\in\Ga\ba G$, $i=1,\cdots , \ell$ such that the map $G\to\Ga\ba G$ given by $g\to x_ig$ is injective on $R_{\epsilon_0}=N^-_{ \epsilon_0}A_{ \epsilon_0}N_{ \epsilon_0}^+M$, and $\bigcup_{i=1}^\ell x_i R_{\epsilon_0/2}$ contains both $\supp f_1$ and $\supp f_2$. As before, set $a_t=\exp (tu+\sqrt t v)$.
We use continuous partitions of unity to write $f_1$ and $f_2$ as finite sums $f_1=\sum_{i=1}^\ell f_{1,i}$ and $f_2=\sum_{j=1}^\ell f_{2,j}$ with $\supp f_{1,i}\subset x_iR_{\epsilon_0/2}$ and $\supp f_{2,j}\subset x_jR_{\epsilon_0/2}$.
Writing $p=ham\in N^-AM$ and using the following decomposition of the Haar measure on $G$: 
$$d(hamn)=e^{-2\rho(\log a)}\,dn\,dm\,da\,dh$$
(cf. \cite[Prop 8.45]{Knapp}),  we have
\begin{align}\label{eq.pu1}
&\int_{\GaG} f_1 (x
a_t
)f_2(x)\,d x\\
&=\sum_{i,j} \int_{R_{\e_0}} f_{1,i} (x_ipn
a_t
)f_{2,j} (x_ipn )e^{-2\rho(\log 
a
)}\,dn\,dm\,da\,dh\notag
\\&=\sum_{i,j} \int_{N_{\e_0}^-A_{\e_0}M}\left( \int_{N_{\e_0}^+} f_{1,i} (x_i pn
a_t
 )f_{2,j} (x_ipn )\,dn\right)e^{-2\rho(\log 
 a
)}\,dm\,da\,dh.\notag
\end{align}
Applying Proposition \ref{p2}, it follows (cf.\ also \eqref{BRPN})

\begin{align*}
&\lim_{t\rightarrow\infty} \Psi(t)
e^{(2\rho-\psi)(\log a_t)} \int_{\GaG} f_1 (x a_t)f_2(x)\,dx
\\&= J(v) \sum_{i} m ^{\mathrm{BR}}(f_{1,i})\sum_j\int_{N_{\e_0}^-A_{\e_0}M} \mu_{x_ipN^+}^{\mathrm{PS}} (f_{2,j}(x_i p\,\cdot\,) )
e^{-2\rho(\log 
a)}\,dm\,da\,dh
\\&=
J(v) \sum_{i} m ^{\mathrm{BR}}(f_{1,i})\sum_jm ^{\mathrm{BR}_*}(f_{2,j})\\
&=
J(v) m ^{\mathrm{BR}}(f_{1})\,m ^{\mathrm{BR}_*}(f_{2})
\end{align*}
where the second last equality is valid by \eqref{BRPN}. This justifies the first statement.
For the second statement, note that if $tu+\sqrt tv\in\mathfrak a^+$, \eqref{eq.pu1} together with Proposition \ref{p2} gives
$$\left| \Psi(t)e^{(2\rho-\psi)(\log a_t)}\int_{\GaG} f_1 (xa_t)f_2(x)\,d x\right| \leq C_0(f_1,f_2), $$
where 
$C_0(f_1,f_2):=\sum_{i,j} C''(f_{1,i},f_{2,j})\int_{N_{\e_0}^-A_{\e_0}M}e^{-2\rho(\log(a(t,v))}\,dm\,da\,dh.$
This completes the proof.
\end{proof}

We make the following observation, which will be used in the proof of Theorem \ref{uniq2}.
\begin{cor}\label{uniq}
Fix $u\in  \L_\Gamma\cap \op{int}\fa^+$, and let $\psi\in D_\Ga^\star$ be tangent to $\psi_\Gamma$ at $u$.
For $k=1,2$, let
 $\nu_k$ and $\bar \nu_{k}$
 be respectively $(\Gamma, \psi)$ and $(\Gamma, \psi\circ \i)$-PS measures on $\La_\G$.  
 Setting $m_k^{\BMS}:=m^{\BMS}_{\nu_k, \bar \nu_k}$,
suppose that there exist functions $\Psi_k : \bb R_{>0}\to\bb R_{>0}$ such that for all $f_1, f_2\in C_c(\Ga\ba G)^M$,
$$
\lim\limits_{t\to\infty}\Psi_k(t)\int_{\Ga\ba G}f_1(x\exp(tu))f_2(x)\,dm^{\BMS}_k(x)=m_k^{\BMS}(f_1)m_k^{\BMS}(f_2).
$$
Then $\nu_1=\nu_2$ and $\bar\nu_2=\bar\nu_2$.
\end{cor}
\begin{proof}
By an argument similar to the proof of Proposition \ref{prop.mixH0}, 
\begin{align*}
\lim_{t\rightarrow\infty} \Psi_k(t)e^{t(2\rho-\psi_\Gamma)(u)}&\int_{\GaG} f_1 (x \exp(tu))f_2(x)\,dx
=m_{\bar\nu_k}^{\mathrm{BR}}(f_1)\, m_{\nu_k}^{\mathrm{BR}_*}(f_2).
\end{align*}
Fix $f_1 \in C_c(\Ga\ba G)^M$ with $m_{\bar\nu_k}^{\mathrm{BR}}(f_1)>0$ for each $k=1,2$. By considering
 $f_2\in C_c(\Ga\ba G)^M$ with $m_{\nu_k}^{\mathrm{BR}_*}(f_2)>0$, it follows from the hypothesis
 that $c_0:=\lim_{t\to\infty}\frac{\Psi_1(t)}{\Psi_2(t)} >0$.
Set $c_1:=c_0\cdot\frac{m_{\bar\nu_2}^{\mathrm{BR}}(f_1)}{m_{\bar\nu_1}^{\mathrm{BR}}(f_1)}.$
Then for any $f_2\in C_c(\Ga\ba G)^M$, we have
\be\label{ftwo} m_{\nu_1}^{\BR_*}(f_2)=c_1\cdot m_{\nu_2}^{\BR_*}(f_2).\ee Recall from Lemma \ref{qa} that for all $a\in A$,
$a_*m_{\nu_k}^{\BR_*}=e^{(-\psi+2\rho)(\log a)}\,m_{\nu_k}^{\BR_*}.$
We claim that $\nu_1=\nu_2$.
Let $g_0\in G$ be arbitrary. 
Fix $\e>0$ and consider $F\in C(G/P)$ supported in $(g_0 N^+_\epsilon)^-$.
Choose $q_1\in C_c(A)$, and a radial function $q_2\in C_c(N)$ such that $\int_A q_1\,da=\int_{N^-} q_2\,dn=1$.
Define $\tilde f_2\in C_c(G)$ by
\begin{equation*}
\tilde f_2(g):=\begin{cases} F(g_0h^+)q_1(a)q_2(n)e^{-\psi(\log a)}&\quad\mathrm{if\;}g=g_0hman\in g_0 N^+P^-,\\0&\quad\mathrm{otherwise.}
\end{cases}
\end{equation*}
Note that $\tilde f_2$ is $M$-invariant, as $q_2$ is radial. 
Defining $f_2\in C_c(\Ga\ba G)$ by 
$f_2([g]):=\sum_{\ga\in\Ga}\tilde f_2(\ga g)$, a direct computation shows
\begin{align*}
m_{\nu_k}^{\BR_*}(f_2)&=\int_{N^+}\int_{P^-} \tilde f_2(g_0hman)\,e^{\psi(\log a)}\,dm\,da\,dn\,d\nu_k(g_0h^+)\\
&=\int_{N^+}F(g_0h^+)\left(\int_{P^-} q_1(a)q_2(n)\,dm\,da\,dn\right)\,d\nu_k(g_0h^+)\\
&=\nu_k(F).
\end{align*}
Hence for any $g_0^-\in G/P$,  $\nu_1(F)=c_1\cdot\nu_2(F)$ for all $F\in C(G/P)$ supported in
$(g_0\cal O_\e)^-$.  
By using a partition of unity, we get $\nu_1(F)=c_1\cdot\nu_2(F)$ for all $F\in C(G/P)$. 
Since $|\nu_1|=|\nu_2|=1$, we have $c_1=1$ and hence $\nu_1=\nu_2$.
Repeating the same argument for $\i (v)$,  we also get $\bar\nu_1=\bar\nu_2$ (this implies $c_0=1$).
\end{proof}

\section{Anosov groups}\label{pingpong}\label{sec.Anosov}
\subsection{Anosov subgroups}
Let $\Sigma$ be a finitely generated word hyperbolic group and let $\partial \Sigma $ denote the Gromov boundary of $\Sigma$.
We call a Zariski dense discrete subgroup $\Gamma<G$ {\it Anosov}  with respect to $P$ 
if it arises as the image of a $P$-Anosov representation of $\Sigma$. 
A representation $\Phi: \Sigma \to G$ is $P$-Anosov  if $\Phi$ induces a continuous equivariant map $\zeta:\partial \Sigma\to
\F$ such that $(\zeta(x), \zeta(y))\in \F^{(2)}$ for all $x\ne  y\in \partial \Sigma$ following Labourie \cite{La} and Guichard-Wienhard \cite{GW}.

Let $\tau_d:\PSL_2(\br)\to \PSL_d(\br)$ be the $d$-dimensional irreducible representation of $\PSL_2(\br)$. 
For any  torsion-free uniform lattice $\Sigma$ in $\PSL_2(\br)$,
 the connected component of $\tau_d|_{\Sigma}$ in the space $\op{Hom}(\Sigma, \PSL_d(\br))$ is called the Hitchin component. Representations $\Sigma \to \PSL_d(\br)$ in the Hitchin component  are known to be $P$-Anosov \cite{La}. 
In fact, Hitchin components are defined for representations of $\Sigma$ into any split real simple Lie group $G$, and all representations $\Sigma \to G$ in the Hitchin component are known to be $P$-Anosov (\cite{FG}, \cite{GW}). 

We  mention that if $\rho_i:\Sigma \to G_i$ are $P_i$-Anosov where $P_i$ is a minimal parabolic subgroup of $G_i$, then
$\rho_1\times \rho_2: \Sigma\to G_1\times G_2$ is $P_1\times P_2$-Anosov whenever its image is Zariski dense.
Indeed, if $\zeta_i:\partial\Sigma\to G_i/P_i$  denotes the limit map of $\rho_i$, the map $\zeta(x)=(\zeta_1(x),\zeta_2(x))$ provides the desired limit map for $\rho_1\times\rho_2$, and hence $\{(\rho_1(g), \rho_2(g)): g\in \Sigma\}$ is an Anosov subgroup of $G_1\times G_2$.

One subclass of Anosov groups consists of
 Schottky groups, which generalize the Schottky subgroups of rank one Lie groups.
 
For a loxodromic element $g\in G$, we denote by $y_g\in \cal F$ the unique attracting fixed point of $g$.
Let $\ga_1, \cdots,\ga_p$ be loxodromic elements of $G$ ($p\ge 2$).
For each $1\le i\le p$, set $\xi_i^{+1}=y_{\ga_i}$ and $\xi_i^{-1}=y_{\ga_i^{-1}}$.

\begin{Def} \label{def.sch} \rm The subgroup $\Gamma$ generated by $\{\ga_1, \cdots, \ga_p\}$ is called Schottky if
 there exist
 open subsets $b_i^{\pm}, B_i^{\pm}\subset \cal F$, ${1\le i\le p}$  and $0<\e<1$ such that

\begin{enumerate}
\item for all $i\ne j$ and $\omega, \varpi\in \{-1,1\}$, $b_i^\omega\subset B_j^{\varpi}$;
\item 
$\overline{b_i^\omega}\times \overline{b_j^\varpi}\subset\cal F^{(2)}$ whenever $i\ne j$ or $\omega\ne \varpi$;
\item for all $ i$ and $\omega\in \{-1, 1\}$, $\xi^\omega\in \op{int} b_i^{\omega}$,  $\gamma_i^{\omega}B_i^{\omega} \subset b_i^{\omega}\subset B_i^{\omega}$ and the restriction of $\gamma_i$ to
$B_i^{\omega}$ is $\e$-Lipschitz;

\item the intersection  $\bigcap_{1\le i\le p, \omega\in \{1,-1\}} B_i^{\omega}$ is non-empty.
\end{enumerate}
\end{Def}

This is the same definition  as given in \cite[Section 4.2]{Q5}, except for the extra  condition (2), which we added to ensure the following lemma:
\begin{lem}\label{Scc}
Any Zariski dense Schottky subgroup $\Ga<G$  is Anosov.
\end{lem}
\begin{proof} 
The Gromov boundary $\partial\Ga$ can be identified with the set of infinite words of the form $\underline a=(a_0a_1a_2\cdots)$ where $a_i\in\{\ga_1^{\pm 1},\cdots,\ga_p^{\pm 1} \}$ and $a_i\neq a_{i+1}^{-1}$. Fix an element
$\xi_0\in \bigcap_{1\le i\le p, \omega\in \{1,-1\}} B_i^{\omega}$, which exists by (3).
Under the above definition of a Schottky group,
the proof of \cite[Proposition 3.3]{Q4} gives that the map $\underline a\mapsto \lim_{n\to \infty} (a_0a_1\cdots a_n) \xi_0$ induces
 a $\Ga$-equivariant homeomorphism $\zeta : \partial\Ga\to\La_\Ga$ (see also \cite[Proposition 4.5]{Q5}).
Let $\xi\ne \eta\in \Lambda_\Ga$, and write them as
$\xi= \lim_{n\to \infty} (a_0a_1\cdots a_n) \xi_0$ and
$ \eta= \lim_{n\to \infty} (a_0'a_1'\cdots a_n') \xi_0$ for some infinite words $\underline a$
and $\underline a'$. Let $k$ be the smallest integer such that $a_k\ne a_k'$, and set $\gamma:=a_0 \cdots a_{k-1}\in \Gamma$.
It follows from (1) and (2) that $\gamma^{-1}\xi\in \overline{b_i^\omega}$ and $\gamma^{-1}\eta\in \overline{b_j^\varpi}$
where $i\ne j$ or $\omega\ne \varpi$. Hence
$(\gamma^{-1}\xi, \gamma^{-1}\eta)\in \cal F^{(2)}$ by (2); consequently, $(\xi,\eta)\in \cal F^{(2)}$.
This shows that $\Ga$ is Anosov. \end{proof}
Schottky groups are found everywhere, in the following sense:

\begin{lem}  Any Zariski dense discrete subgroup $\Ga$ contains a Zariski dense Schottky subgroup. \end{lem}
\begin{proof}  This follows from the proof of a more general theorem \cite[Proposition 4.3]{Ben}.
We give a sketch of the proof for the sake of completeness. Since the set of loxodromic elements of $\G$ is Zariski dense,
we may choose a loxodromic element $\ga_1\in \Ga$. There exists a proper Zariski closed subset $F_{\ga_1}\subset G$
which contains all Zariski connected and Zariski closed proper subgroups of $G$ containing $\ga_1$ \cite[Proposition 4.4]{T2}.
We may choose a second loxodromic element $\ga_2\in \Ga-F_{\ga_1}$ such that 
$\{(y_{\ga_2^{\pm 1}}, y_{\ga_1}), (y_{\ga_2^{\pm 1}}, y_{\ga_1^{- 1}})\}\subset \cal F^{(2)}$.
Moreover we can assume that  $\ga_2^k$ generates a Zariski connected subgroup, and hence $\ga_2^k\notin F_{\ga_1}$, for any $k\in \mathbb N$. Let $\varphi_i\in G$ be so that $\ga_i\in \varphi_i( \inte A^+ )\varphi_i^{-1}$.
Then $y_{\ga_i^{\pm 1}}=\varphi_i e^{\pm} $. For $\e>0$, let $b_i^{\pm 1} (\e)$ be the $\e$-neighborhood of $\varphi_i e^{\pm}$
and set $B_i^{\pm 1}(\e) =\varphi_i N_{\e^{-1}}^{\pm} e^{\pm}$ where $N^{\pm}_{\e^{-1}}$ denotes the $\e^{-1}$-neighborhood of $e$ in $N^{\pm}$.
For any $\e_1, \e_2>0$,
$\ga_i^{\pm k} B_i^{\pm 1}(\e_2)\subset b_i^{\pm 1}(\e_1)$ for all sufficiently large $k$.
It follows that we can find $k$, $\e_1>0$ and $\e_2>0$ so that
$\ga_1^k$ and $\ga_2^k$ satisfy  the conditions in \ref{def.sch} 
with $b_i^{\pm 1}:=b_i^{\pm 1} (\e_1)$ and $ B_i^{\pm 1}:=B_i^{\pm 1}(\e_2)$.
 \end{proof}

\subsection{ BMS-mixing}\label{bmsm} The following theorem was first proved by Sambarino \cite{Samb1}, when 
$\Sigma$ is the fundamental group of a compact negatively curved manifold, using the work of Thirion \cite{Thirion}. Using the reparametrization theorem of \cite{BCLS} for a general case, Chow and Sarkar \cite{CS} proved the following theorem; their proof uses a different symbolic coding than
\cite{Samb1} and works for functions which are not necessarily $M$-invariant.

\begin{thm} \label{mixing2}
Let $u\in \op{int}\L_\Gamma$. Let $\psi\in D_\G^{\star}$ be tangent to $\psi_\Gamma$ at $u$, and $\nu$ and $\nu_{\i}$
 be respectively $(\Gamma, \psi)$ and $(\Gamma, \psi\circ \i)$-$\op{PS}$ measures on $\La_\Ga$.  Set $m^{\BMS}:=m^{\BMS}_{\nu, \nu_{\i}}$.
Then there exists $\kappa_u>0$ such that
for any $v\in \ker \psi$ and any $f_1,\,f_2\in C_c(\GaG)^M$,
\begin{multline*}
\lim_{t\rightarrow\infty} t^{(r-1)/2}\int_{\GaG} f_1(x) f_2(x \exp(tu+\sqrt t v)) \,dm^{\mathrm{BMS}}(x)\\
=\kappa_{u} \,e^{-I(v)/2}\, m^{\mathrm{BMS}}(f_1)\,m^{\mathrm{BMS}}(f_2),
\end{multline*}
where $I : \op{ker}( \psi)\to \bb R$ is given by
\begin{equation}\label{ii}
I(v):= c \cdot
\frac{\|v\|_*^2 \|u\|_*^2 -\la v, u\ra_*^2}{\|u\|_*^2}
\end{equation} for some inner product $\la\cdot,\cdot \ra_*$ and some $c>0$. Moreover, the left-hand side is uniformly bounded 
over all $(t,v)\in(0,\infty)\times\ker\psi$ with $tu+\sqrt tv\in\mathfrak a^+$.
\end{thm}
Although the second statement of Theorem \ref{mixing2} is not stated in \cite[Theorem 3.8]{Samb1},  its proof uses the same technique as \cite[Theorem 1.1]{Thirion}, where the corresponding statement can be found.
\begin{remark}\label{general} \rm
Theorem \ref{mixing2} is the main reason for the assumption that $\Gamma$ is a Anosov subgroup. In fact, all our results stated in the introduction hold whenever $\Gamma$ satisfies Theorems \ref{mixing2} and \ref{int}.  
\end{remark}

 In the rest of the paper,
 let $\Gamma$ be an Anosov subgroup of $G$.
  The following theorem was proved by Sambarino \cite{Samb2} for a special case
  and by Potrie-Sambarino \cite[Propositions 4.6 and 4.11]{PS} in general: \begin{thm}\label{int}  
\begin{enumerate}
\item  $\L_\Gamma-\{0\}\subset \op{int}  \mathfrak a^+$. 
 \item $\psi_\Gamma$ is strictly concave and analytic on $\{v\in\op{int}\L_\Gamma: \|v\|=1\}$.
 \end{enumerate}
 \end{thm}
By Lemma \ref{di}, and Theorem \ref{int}, we get the following corollary:
\begin{cor} \label{uuu}For each unit vector $u\in \op{int}\L_\Gamma$, there exists a unique $\psi_u\in \dg$ tangent to $\psi_\Gamma$ at $u$, which is given by \begin{equation*}
\psi_u(\cdot)=\la \nabla\psi_\Ga (u), \cdot\ra.
\end{equation*}
\end{cor}

\begin{thm} \label{uniq2} For any $\psi\in \dg$,
there exists a unique $(\Gamma, \psi)$-$\PS$ measure on $\Lambda_\Gamma$. 
\end{thm}
 \begin{proof} When $\Gamma$ is the fundamental group of a closed negatively curved Riemmanian manifold, this  follows by combining  \cite[Theorem 3.1]{Samb3} and \cite[Proposition 7.8]{Samb3}
  The general case follows from Theorem \ref{mixing2} by Corollary \ref{uniq}.
 \end{proof}
 
Note that this theorem implies the $\G$-ergodicity of every $(\Ga,\psi)$-PS measure.

\begin{rmk} \rm The $\Gamma$-ergodicity of a $(\Ga, \psi)$-$\PS$ measure for $\psi\in \dg$ also follows from \cite[Thm A.2]{Ca2} together with the reparameterization theorem \cite[Thm. 4.10]{CS}. By
the well-known Sullivan's argument, it implies the uniqueness of
$(\Ga,\psi)$-$\PS$ measure.
\end{rmk}

For each $u\in \op{int}\L_\Gamma$,
we denote by $\nu_u$ the unique $(\Gamma, \psi_u)$-PS measure on $\La_\Ga$.
We now set
\be \label{def.BMS2}
 m^{\BMS}_{u}:=m^{\BMS}_{\nu_u, \nu_{\i(u)}}, \quad m^{\BR}_{\i(u)}:=m^{\BR}_{\nu_{\i(u)}},\quad 
 m^{\BR_*}_u:=m^{\BR_*}_{\nu_u} .\ee

We mention that
all three measures are infinite measures when $G$ has rank at least $2$ (cf. \cite[Corollary 4.9]{LO}). We also refer to \cite{LO} and \cite{LO2} where ergodic properties of these measures are discussed.

We deduce the following from Proposition \ref{p1}, Proposition \ref{prop.mixH0}, and Theorem \ref{mixing2}. 
\begin{thm} \label{m11}\label{prop.mixH} Let $u\in \op{int}\L_\Gamma$. 
\begin{enumerate}
\item For any $f_1, f_2\in C_c(\GaG)^M$ and   $v\in \ker \psi_u$,
 $$
\begin{multlined}[t]
\lim_{t\to +\infty} t^{(r-1)/2} e^{(2\rho -\psi_u) ( tu+\sqrt tv)}  \int_{\Gamma\ba G} f_1(x \exp(tu+\sqrt tv) ) f_2(x) dx 
\\
= \kappa_u\, e^{-I(v)/2} m^{\BR}_{{\i(u)}} (f_1) m^{\BR_*}_{u} (f_2).
\end{multlined}
$$
\item
 For any $f\in C_c(\GaG)^M$, $\phi\in C_c(N^+)$,  $x=[g]\in \Gamma\ba G$, and  $v\in \ker \psi_u$,
\begin{multline*}
\lim_{t\rightarrow\infty}  t^{(r-1)/2}
e^{(2\rho-\psi_u)(tu+\sqrt tv)}
\int_{N^+}f (x n \exp(tu+\sqrt tv))\phi(n)\,dn\\
=  \kappa_u\, e^{-I(v)/2} 
m_{{\i(u)}}^{\mathrm{BR}}(f)\, \mu_{ gN^+,\nu_u}^{\mathrm{PS}}(\phi).
\end{multline*}
\end{enumerate}
Moreover the left-hand sides of the above equalities are uniformly bounded 
 for all $(t,v)\in\bb (0,\infty)\times\ker\psi_u$ with $tu+\sqrt tv\in\mathfrak a^+$. 
\end{thm}

Recalling that $\psi_u(u)=\psi_\Ga(u)$, the special case of Theorem \ref{prop.mixH} when $v=0$ now implies Theorem \ref{m1} and Theorem \ref{m2}.

\section{Equidistribution of  translates of \texorpdfstring{$\Gamma\ba\Gamma H$}{}}\label{sec.aff}
\subsection*{Symmetric subgroups of $G$}\label{gc}
Let $H<G$ be a symmetric subgroup; that is to say, $H$ is the identity component of the set of fixed points of an involution $\sigma$ of $G$. We start by reviewing some general structure theory regarding symmetric subgroups; see Chapter 6 of \cite{Sch} for more details on this. The involution $\sigma$ induces a Lie algebra involution on $\mathfrak g$, which (using a slight abuse of notation) we also denote by $\sigma$. There exists a Cartan involution of $G$ that commutes with $\sigma$; without loss of generality, we may assume that $\theta$ from Section \ref{ps} commutes with $\sigma$. These involutions give rise to the decompositions $\mathfrak{g}=\mathfrak{k}\oplus\mathfrak{p}$ and $\mathfrak{g}=\mathfrak{h}\oplus\mathfrak{q}$ into the $+1$ and $-1$ eigenspace decompositions of $\theta$ and $\sigma$, respectively. Let $\mathfrak{a}$ be a maximal abelian subalgebra of $\mathfrak{p}$ such that $\fb:=\mathfrak{a}\cap\mathfrak{q}$ is a maximal abelian subalgebra of $\mathfrak{p}\cap\mathfrak{q}$. Denote the dimension of $\fa$ by $r$ and the dimension of $\fb$ by $r_0$.

Let $\Sigma_\sigma\subset \fb^*$ be the root system of $\fb$,  i.e.  $$\Sigma_\sigma=\lbrace \lambda\in\fb^*\setminus\lbrace 0\rbrace: \exists X\in\fg-\{0\}\,\text{ with }\, \ad_YX=\lambda(Y)X\text{ for all }Y\in\fb\rbrace.$$
 From now on, we fix a closed positive Weyl chamber $\fb^+\subset \fb$ for $\Sigma_{\sigma}$ which has been chosen compatibly with $\fa^+$ as follows \cite[Sec. 3]{GOS}: denoting the positive roots of $\fa$ by $\Sigma^+$, we assume that there exists a collection of positive roots $\Sigma_\sigma^+$ of $\fb$ such that the elements of $\Sigma_{\sigma}^+$ are all obtained by restricting elements of $\Sigma^+$ to $\fb$, hence $\fb^+\subset \fa^+$. We will denote $B=\exp(\fb)$ and $B^+=\exp(\fb^+)$.

Let $\cal W_\sigma:=\op{N}_K(\mathfrak b)/Z_K(\mathfrak b)$ and $\cal W_{\sigma,\theta}:=\op{N}_{K\cap H}(\mathfrak b)/Z_{K\cap H}(\mathfrak b)$. There then exists a finite set of representatives $\cal W\subset \op{N}_K(\mathfrak a)\cap\op{N}_K(\mathfrak b)$ for $\cal W_{\sigma,\theta}\ba\cal W_\sigma$, and we have the following generalized Cartan decomposition:
\begin{equation}\label{eq.HBK}
G=H\op{exp}(\fb)K=H\cal W\op{exp}(\fb^+)K,
\end{equation}
in the sense that for any  $g\in G$,
there exist unique elements $b\in B^+$ and $\omega\in \cal W$ such that
$$g\in H \omega b  K.$$

\subsection*{Directions in $\fb^+\cap\inte\scrL_\Ga$}
Let $\Gamma$ be an Anosov subgroup of $G$.  In the rest of this section,  we assume that
\begin{equation*}
 \fb^+\cap\mathrm{int}(\scrL_{\Gamma})\neq \emptyset.
\end{equation*}
Since $\scrL_{\Gamma}\subset\inte(\fa^+)$ by Theorem \ref{int}, it follows that
 $\fb^+\cap \inte(\fa^+)\neq \emptyset$.

We now fix a unit vector (with respect to the norm $\|\cdot\|$ on $\fa$) 
$$v\in\fb^+\cap\inte\scrL_{\Gamma}  $$
and set
\be\label{deee} \delta:=\psi_\Gamma (v)>0.\ee

 By Corollary \ref{uuu}, the linear form $\Theta\in \fa^*$ defined as 
 \be\label{DefT} \Theta(w)=\langle \nabla \psi_\Gamma (v), w\rangle\ee
gives the unique linear form in $D_\Gamma^\star$ such that 
$\Theta(v)=\psi_\Gamma(v)=\delta$. 

 \begin{lem} \label{large}We have  $\L_\Gamma\cap\ker \Theta =\{0\}$.
\end{lem}
\begin{proof} We use the fact that $\psi_\Ga$ is strictly concave (Theorem \ref{int}).
Since $\Theta\ge \psi_\Ga$ on $\fa$, and $\Theta(v)=\psi_\Ga(v)$, it follows that
$\Theta (w)> \psi_\Ga (w)$ for all  vectors $w\in \fa^+ - \br v$.
Since $\psi_\Ga\ge 0$ on $\L_\Ga$, we have $\Theta>0$ on $\L_\Gamma-\{0\}$, i.e.,
$\L_\Gamma\cap\ker \Theta=\{0\}$.\end{proof}

We use the following notation: for $t>0$ and $w\in \ker \Theta$,
\begin{equation*}
\begin{aligned}
a(t,w):=\exp ( tv+\sqrt{t}w), \quad\text{and}\quad
\underline a(t,w):=tv+\sqrt{t}w.
\end{aligned}
\end{equation*}

We set 
$$m^{\mathrm{BR}}=m^{\BR}_{\nu_{\i (v)}} \quad \text{ and} \quad m^{\mathrm{BR}_*}=m_{\nu_v}^{\BR_*}.$$
\subsection*{Patterson-Sullivan measures on ${H}$} \label{psmH}
Let $P=MAN$ be the minimal parabolic subgroup.
Since $\fb^+\cap\op{int}(\fa^+)\neq\emptyset$, it follows that $M=Z_K(\fb)$, and the unipotent subgroup whose Lie algebra is the sum of positive root spaces corresponding to $\Sigma_\sigma$ coincides with $N$. 
\begin{Lem} We have $H\cap N=\{e\}.$\end{Lem}
\begin{proof}  Fix $\alpha\in \Sigma_\sigma$.
Since $v\in \fb^+\cap \inte\fa^+$, we have $\alpha(v)>0$. Letting $X\in \mathfrak n$ be such that $ [v,X]=\alpha(v) X$, we have
$\sigma( X ) = \alpha(v)^{-1} \sigma([v,X])= \alpha(v)^{-1} [\sigma(v),\sigma(X)]= \alpha(v)^{-1} [-v,\sigma(X)],$
i.e. $[v,\sigma(X)]=-\alpha(v)\sigma(X)$.
Therefore $\sigma(X)\in \mathfrak n^+$.  Since $\sigma$ fixes $H$ point-wise and
 swaps $N$ and $N^+$, we get $H \cap N =\{e\} $.
\end{proof}

By \cite[Theorem 3-(iii)]{Mat}, 
\be\label{hcp} H\cap P=(H\cap M)(H\cap A)(H\cap N)=(H\cap M)(H\cap A).\ee
Together with the fact $H\cap B=H\cap N=\{e\}$, it then follows that $H\cap MBN=H\cap M$. 

\begin{Def}\label{Hskinning} \rm 
Define a measure $\mu_{ H}^{\mathrm{PS}}=\mu_{ H,v}^{\mathrm{PS}}$ on ${H}$ as follows: for $\phi\in C_c(H)$, let
$$
 \mu_{{H}}^{\mathrm{PS}}(\phi)=  \int_{ h_0\in H/(H\cap P)} 
  \int_{p\in H\cap P}    \phi( h_0p ) e^{\Theta(\beta_{ h_0^+}(e,  h_0p))}  \,dp
 \, d\nu ( h_0^+),
$$ where $dp$ is a right-Haar measure on $H\cap P$; for $h_0\in H/(H\cap P)$, $h_0^+$ is well-defined independent of the choice of a representative.
The measure defined above is $\Gamma\cap H$-invariant: for any  $\gamma\in \Gamma\cap H$, $\gamma_*\mu_{ {H}}^{\mathrm{PS}}=
 \mu_{{H}}^{\mathrm{PS}}$.  Therefore, if $\Gamma\ba \Gamma H$ is closed in $\Gamma\ba G$,
$d\mu_{H,v}^{\PS}$ induces a locally finite Borel measure on $\Gamma\ba \Gamma H\simeq (\Gamma\cap H)\ba H$, which we denote by $\mu^{\PS}_{[e]H}=\mu^{\PS}_{[e]H,v}$.
\end{Def}

Note that $\op{supp}\mu_{[e]H}^{\PS}=\{[h]\in \Ga\ba\Ga H : h^+\in\La_\Ga\}$,
and hence 
\be\label{hpp} |\mu_{[e]H}^{\PS}|=0\;\; \text{ if and only if } \;\; \La_\Ga\cap HP/P=\emptyset.\ee

For a subset $S\subset G$ and $\epsilon>0$, set $S_\epsilon:=\{s\in S:d(e,s)\leq\e\}$.
Let $M'\subset M$ be a Borel section for the map $m\mapsto (H\cap M)m$, and $P'=M'BN$ be the subset of the minimal parabolic subgroup $P=MAN$.
Note that the map $H\times P'\to G$ given by $(h,p')\mapsto hp'$ is injective, which is an open map in a neighborhood of $e$  (\cite[Prop. 7.1.8(ii)]{Sch}, \cite[Prop. 4.3]{EM}).
For $\e>0$, let $\rho_\e\in C((NB)_\e)$ be a non-negative function such that
$$
\int_{NB}\rho_\e(nb)\,dn\,db=1,
$$
and $ \rho_\e(mnbm^{-1})=\rho_\e(nb)$ for all $m\in M$ and $nb\in NB$.
Fixing $\phi\in C_c(H)^{H\cap M}$ and $\e>0$ smaller than the injectivity radius of $\text{supp}(\phi)$, define $\tilde \Phi_\e\in C_c(G)$ by
\begin{equation}\label{po}
\tilde \Phi_\e(g):=\begin{cases} \phi( h)\rho_\e(nb)\qquad&\mathrm{if\;}g= hm'nb \in {H}P',\\0\qquad&\mathrm{otherwise.}
\end{cases}
\end{equation}
Observe that $\tilde \Phi_\e$ is right $M$-invariant. Define now $\Phi_\e\in C_c(\Ga\ba G)$ by
$\Phi_\e([g])=\sum_{\ga\in\Ga}\tilde\Phi_\e(\ga g)$.
Since $\tilde \Phi_\e$ is right $M$-invariant, so is $\Phi_\e$.
\begin{lemma}\label{lem.BRPS} 
For $\phi\in C_c(H)^{H\cap M}$ and $\Phi_\e$ as above, we have
\begin{equation*}
m^{\mathrm{BR}_*}(\Phi_\e)=\mu_{H}^{\PS}(\phi)(1+O(\e)).
\end{equation*}

\end{lemma}
\begin{proof}
Note that $\supp(\tilde \Phi_\e)\subset HM (NB)_\e$. 
The proof of the lemma now relies on the following three observations.

Firstly, for each $hm\in  HM'$ and $n\in N$,
$$d(hm n) =e^{2\rho(\beta_{(hmn)^-}(e,hmn))} dm_o(hmn^-)=  d n$$
is a Lebesgue measure on $hmN$.

Secondly, for $g=hm'nb\in H P'$,
the decomposition $$\beta_{g^+}(e, g)= \beta_{h^+}(e, h)+\beta_{e^+}(e, nb)=
 \beta_{h^+}(e, h)+  \i(\log b)$$
induces an isomorphism $A\cong (A\cap H) \times B$.
This implies that
$$d(\beta_{g^+}(e, g))=d( \beta_{h^+}(e, h))\,d(\beta_{e^+}(e, p))=d( \beta_{h^+}(e, h))\,d(b),$$
and for all $hm'\in HM'$ and $nb\in (NB)_\e$,
$$e^{\Theta(\beta_{h^+}(e,hm'nb))}= e^{\Theta(\beta_{h^+}(e,h))}(1 +O(\e))$$
by continuity of the Busemann function.
 
Finally, we also have:
\begin{align*}
\beta_{{}(hm'nb)^-}(e,{}hm'nb)&=\beta_{{}(hm'n)^-}(e,{}hm'n)+\beta_{{}(hm'n)^-}({}hm'n,{}hm'nb)\\
&=\beta_{{}(hm'n)^-}(e,{}hm'n)+\beta_{e^-}(e,b)\\
&=\beta_{{}(hm'n)^-}(e,{}hm'n)-\i(\log b).
\end{align*}
Consequently, $e^{2\rho(\beta_{(hm'nb)^-}(e,hm'nb))}=e^{2\rho(\beta_{(hm'n)^-}(e,hm'n))}(1 +O(\e))$. 
Using the definition of $m^{\mathrm{BR}_*}(\Phi_\e) $ and the second and third observations above give
\begin{align*}
&\begin{multlined}[t]m^{\mathrm{BR}_*}(\Phi_\e)= 
\int_{H  \times M'\times (NB)_\e}  \phi(h)  \rho_\e(nb)e^{\Theta(\beta_{h^+}(e,hm'nb))+2\rho(\beta_{(hm'n)^-}(e,hm'nb))}
\\ \times dm_o\big((hm'n)^-\big)\,d(\beta_{h^+}(e,hm'nb)) \,d\nu(h^+)
\end{multlined}\\
&
=\begin{multlined}[t]
(1+O(\e)) \int_{H  \times M'\times (NB)_\e}\phi(h)  e^{\Theta(\beta_{h^+}(e,h))}  \rho_\e(nb)e^{2\rho(\beta_{(hm'n)^-}(e,hm'n))}\\
\times dm_o\big((hm'n)^-\big)\,db\, d( \beta_{h^+}(e, h)) \, d\nu (h^+).
\end{multlined}
\end{align*}
We now choose a section $H_0\subset H$  for the map $h\mapsto h(H\cap P)$, and write  $h=h_0m  a_h  \in H_0 (M\cap H)(A\cap H)=H_0(H\cap P)$. Using the first observation above, we then have
\begin{align*}
&=(1+O(\e)) \int_{H_0(H\cap P)} \phi(h_0m a_h) e^{\Theta(\beta_{h_0^+}(e,h))} \left(\int_{(NB)_\e} \rho_\e(nb)dn\, db \right)\, d( a_h) \,
 d\nu (h_0^+)\\
&=(1+O(\e))  \int_{ h_0\in H/(H\cap P)}  \int_{p\in H\cap P}    \phi( h_0p )e^{\Theta(\beta_{ h_0^+}(e,  h_0p))}   \,dp
 \, d\nu ( h_0^+)
\\ &= (1+O(\e)) \mu^{\PS}_{H}(\phi).
\end{align*}
\end{proof}

\subsection*{Equidisitribution of translates of $\Gamma\ba \Gamma H$}
\begin{prop}\label{prop.mixH2}\label{eq} For any $f\in C_c(\Gamma\ba G)^{M}$, $\phi\in C_c(H)^{H\cap M}$, and $w\in \fb\cap\ker\Theta$,
$$
\lim_{t\rightarrow\infty} t^{(r-1)/2}e^{(2\rho-\Theta)(\underline a(t,w))} \int_{H} f ([h]a(t,w)) \phi (h)\,dh
 = \kappa_{v}\, e^{-I(w)/2}  \,m^{\mathrm{BR}}(f)\, \mu_{H}^{\PS}(\phi),$$
and there exists $C'=C'(f,\phi)>0$ such that for all $(t,w)$ with $\underline a(t,w)\in\fb^+$, 
\be\label{qqq}
\left| t^{(r-1)/2} e^{(2\rho-\Theta)(\underline a(t,w))} \int_{H} f ([h]a(t,w)) \phi (h)\,dh\right| <C'.
\ee
\end{prop}
\begin{proof}
For $\epsilon>0$, let $R_\epsilon:=N_{\e}A_{\e}N_{\e}^+M$ and define $f_\e^\pm\in C_c(\Ga\ba G)$ by
\begin{equation}\label{ff}
f_\e^+(y)=\sup_{g\in R_\epsilon}f(yg),\text{ and }f_\e^-(y)=\inf_{g\in R_\epsilon}f(yg).
\end{equation}
Since $R_\e$ is right $M$-invariant, it follows that $f_{\e}^\pm\in C_c(\Ga\ba G)^M$.
Let $C_0\subset H$ denote  the support of $\phi$; we may assume that $C_0$ injects to its image under the map $G\mapsto \Gamma\ba  G$. Choosing $\rho_\e \in C_c((NB)_\e)$ and defining $\Phi_\e$ as above,  we let $d\lambda (m')$ denote the density on $M'$ of total mass one such that
\be\label{ll} d(hm'nb)= dh\, d\lambda(m')\,dn\,db\ee
(where $h\in  H$, $m'\in M'$, $n\in N$, and $b\in B$) is a Haar measure on $G$.
 
We then obtain 
\begin{align*}
\int_{H}f([h]a(t,w))\phi(h)\,dh & = \int_{C_0}f([h]a(t,w)) \phi(h)\left(\int_{(NB)_{\epsilon}}\rho_\e(nb) dn\,db \right)dh \\
&=\int_{C_0(NB)_{\epsilon}}f([h]a(t,w))\Phi_\e([h]nb)\, dh\,dn\,db.
\end{align*}
Using the $M$-invariance of $f$ and the definitions of $\Phi_\e$ and $d\lambda$ gives
\begin{equation*}
f([h]a(t,w))\Phi_\e([h]nb)=\int_{M'}f([h]m'a(t,w))\Phi_\e([h]m'nb)\,d\lambda(m'),
\end{equation*}
and so
\begin{equation*}
\int_{ H}f([h]a(t,w))\phi(h)\,dh=\int_{C_0M'(NB)_{\epsilon}}f([h]m'a(t,w))\Phi_\e([h]m'nb)\, dh\,d\lambda(m')\,dn\,db.
\end{equation*}

Since $v\in \op{int}\fb^+$, 
for all $(t,w)$ such that $\underline a(t,w)\in\fb^+$, and for all $nb\in (NB)_\e$, we have
\begin{align}\label{eq.ee}
f(xhm'a(t,w))&=f(xhm'nba(t,w)\cdot(a(t,w)^{-1}(nb)^{-1} a(t,w)))\notag\\
&\leq f_\e^+(xhm'nba(t,w)).
\end{align}
Consequently,
\begin{align*}
&\int_{C_0}f([h]a(t,w))\phi(h)\,dh \\ & \leq\int_{C_0M'(NB)_{\epsilon}}f_\e^+([h]m'nba(t,w))\Phi_\e([h]m'nb)\, dh\,d\lambda(m')\,dn\,db\\
&=\int_{\Ga\ba G} f_\e^+(ya(t,w))\Phi_\e (y)\,dy.
\end{align*}
A similar computation shows that
\begin{equation*}
\int_{C_0}f([h]a(t,w))\phi(h)\,dh\geq \int_{\Ga\ba G} f_\e^-(ya(t,w))\Phi_\e(y)\,dy.
\end{equation*}
On the other hand, we have
\begin{align*}
&\lim_{t\rightarrow\infty} t^{(r-1)/2} e^{(2\rho-\Theta)(\underline a(t,w))} \int_{\GaG} f_{\e}^\pm (y a(t,w))\Phi_\e(y)\,dy\\
&=\kappa_{v}\, e^{-I(w)/2}  \,m_{ }^{\mathrm{BR}}(f_{\e}^\pm)\, m_{ }^{\mathrm{BR}_*}(\Phi_\e)\\
&=\kappa_{v}\, e^{-I(w)/2}  \,m_{ }^{\mathrm{BR}}(f_{\e}^\pm)\,\mu_{{H}}^{\PS}(\phi)(1+O(\e)).
\end{align*}
Taking $\e\to0$ in the last equality proves the first statement.
The second statement is clear with the choice of $C'=C(f_\e^+,\Phi_\e)$, finishing the proof.
\end{proof}

\section{Counting in affine symmetric spaces}\label{gamma}
Let $\Gamma$ be an Anosov subgroup of $G$, and let $H$ be a symmetric subgroup of $G$.
 We continue to use the notation for $v$, $\nu$, $\delta$, $r$, $r_0$, $\Theta$, and
  $G=H\cal W\op{exp}(\mathfrak  b^+)K$, etc. from Section \ref{sec.aff};
hence
$v\in\fb^+\cap\inte\scrL_{\Gamma}  $ is a unit vector (with respect to the norm $\|\cdot\|$ on $\fa$).
We denote by $| \cdot |$ the norm on $\fa$ induced by an inner product
$(\cdot,\cdot)$ with respect to which $v$ and $\ker \Theta$ are orthogonal to each other, and such that $|v|=1$.

In the following we fix a convex cone
$\mathcal C\subset \fb^+ \cap (\mathrm{int}( \fa^+)\cup\lbrace 0\rbrace)$  such that  \be\label{fcc} v \in \inte_{\fb} \mathcal C \quad \text{ and }\quad 
\overline{\mathcal C} \cap \ker \Theta =\{0\}\ee
where $\inte_{\fb} \mathcal C $ means the interior of $\cal C$ in the relative topology of $\fb$.
Note that there are convex cones which contain $ \L_\Gamma \cap \fb^+$ and
satisfy \eqref{fcc}  by Theorem \ref{int} (1) and Lemma \ref{large}. 
\begin{Rmk} \rm Note that if $v=u_\Ga\in\fb^+$, then by Lemma \ref{di}, $\Theta(w)=\delta_{\Gamma}\langle u_\Ga,w\rangle$, hence the cone $\fb^+\cap \left(\inte(\fa^+)\cup\{0\}\right)$ satisfies the conditions placed on $\scrC$ above.
\end{Rmk}

By the condition \eqref{fcc}, we have
\be\label{fc} \fc\subset \{tv+\sqrt{t}w: t\ge 0, w\in \ker \Theta\}.\ee

Denote $$\fc_T:=\{w\in \fc: |w|<T\}\quad\text{ for $T>1$. }$$ 
For  $w\in \ker \Theta$, set
\begin{align*} 
&R_T(w):=\{t\in\bb R : tv+\sqrt t w \in  \fc_T \}.
\end{align*}
for all sufficiently large $T$, $R_T(w)$ is an interval of the form
$$R_T(w)=[t_w, \sfrac{1}{2} \left( -|w|^2 +\sqrt{ |w|^4+4T^2}\right)].$$

\begin{lem}\label{dom}  \begin{enumerate}
 \item There exists $c>0$ such that for all $w\in\fb\cap \ker \Theta$ and $T>1$,
$$e^{-\delta T}T^{(r-r_0)/2} \int_{R_T(w)}t^{\frac{r_0-r}{2}}e^{\delta t} dt \le \delta^{-1} e^{-c \delta |w|^2}.$$

\item We have \begin{equation}\label{eq.WR2} 
\lim_{T\to \infty} e^{-\delta T}T^{(r-r_0)/2} \int_{R_T(w) }t^{\frac{r_0-r}{2}}e^{\delta t} dt =\delta^{-1} e^{- \delta |w|^2/2}.\end{equation}
\end{enumerate}
\end{lem}

\begin{proof} Note that for any non-zero vector $x\in\cal C$, 
$(v,x)> 0$. Since $\cal C$ is a  convex cone with $\overline{\cal C}\cap \ker \Theta=\{0\}$,
it follows that there exists $0<\theta_0<\pi/2$ such that the angle between any vector in $\cal C$ and $v$ is at most $\theta_0$.
Now, as $v$ is perpendicular to $\ker\Theta$ with respect to $(\cdot,\cdot)$, we have that for any $t$ such that
 $\underline a(t,w)\in\fc$,
$$\frac{\sqrt{t}|{w|}}{t|v|}\leq\tan\theta_0,\text{ or, equivalently, }{|w|}^2\leq\tan^2\theta_0\cdot t .$$

In particular, for $t\in R_T(w)$, we have
$$
T^2\geq t^2+t {|w|}^2\geq \left(\frac{1}{\tan^2\theta_0}+\frac{1}{\tan^4\theta_0}\right) {|w|}^4. $$
This gives the upper bound
$$
\tfrac{-{|w|}^2+\sqrt{{|w|}^4+4T^2}}{2}- T
=-\tfrac{{|w|}^2}{2}+\tfrac{{|w|}^4}{2\left(\sqrt{{|w|}^4+4 T^2}+2 T\right)} \leq -c{|w|}^2,
$$
with
\be\label{c1}
c:=\frac{1}{2}\left( 1- \left({1 +4\left(\tfrac{1}{\tan^2\theta_0}+\tfrac{1}{\tan^4\theta_0}\right) }\right)^{-1/2}\right) >0.
\ee
Now, by changing variables,
\begin{align*}
&\int_{t_w}^{\frac{-{|w|}^2+\sqrt{{|w|}^4+4T^2}}{2}}t^{\frac{r_0-r}{2}}e^{\delta t}\,dt=\int_{t_w-T}^{\frac{-{|w|}^2+\sqrt{{|w|}^4+4T^2}}{2}- T}(t+T)^{\frac{r_0-r}{2}}e^{{\delta}(t+T)}\,dt,
\end{align*}
and hence
\begin{align}\label{eq.DC11}
&e^{-{\delta} T} T^{\frac{r-r_0}{2}}\int_{t_w}^{\frac{-{|w|}^2+\sqrt{{|w|}^4+4 T^2}}{2}}t^{\frac{r_0-r}{2}}e^{\delta t}\,dt\notag\\
&=\int_{t_w-T}^{\frac{-{|w|}^2+\sqrt{{|w|}^4+4 T^2}}{2}- T}\left(\frac{t}{T}+1\right)^{\frac{r_0-r}{2}}e^{\delta t}\,dt \\
&\leq\int_{-\infty}^{-c{|w|}^2} e^{\delta t}\,dt=\tfrac{1}{\delta} e^{-c{\delta}{|w|}^2},\notag
\end{align}
which proves (1).

The second claim  (2) follows as well, because by the dominated convergence theorem, \eqref{eq.DC11} converges to
$$
\int_{-\infty}^{-\frac{{|w|}^2}{2}}e^{\delta t}\,dt=\tfrac{1}{\delta} e^{-\frac{{\delta}{|w|}^2}{2}}.
$$

\end{proof}

We fix  a left $(H\cap M)$-invariant function  $\tau_H\in C_c(H)$  with its support injecting to $\Gamma\ba G$, and 
and a right $M$-invariant function $\tau_K\in C(K)$.
Define a function $Z_T: G\rightarrow \RR$ as follows:  
\begin{equation*} Z_T(g):=\begin{cases} \tau_H(h)  \mathbf{1}_{\cal C_T} (\log b)\tau_K( k) &\text{$g=hb k\in H\exp(\mathcal C) K$,}\\
0 & \text{$g\not\in H\exp(\cal C ) K$.}\end{cases} \end{equation*}
Since $\scrC \subset\mathrm{int}(\fa^+)\cup\{0\}$,  $hM$ and $Mk$ are uniquely determined and hence $Z_T$ is well-defined.

\begin{Def}[Bi-sector counting function]\rm For $T>0$, define $F_T= F_{T,\tau_H,\tau_K}: \GaG\to\bb R$ by
\begin{equation}\label{FTdef}
F_{T}([g])=\sum_{\ga\in \Ga}Z_T(\ga g).
\end{equation}
\end{Def}

For $\Phi\in C_c(\Gamma\ba G)$ and a left $M$-invariant Borel function $f$ on $K$,
we define the following $M$-invariant function on $\Gamma\ba G$: for $x\in\GaG$,
$$\Phi* f (x):=\int_{k\in K}\Phi(xk)f(k)\,dk.$$

In the definition \eqref{cv} below,  the integral over the trivial subspace $\{0\}$ should be interpreted as $\int_0 f(w) dw=f(0)$.
In particular, if $\fb\cap \ker \Theta=\{0\}$, then $c_v=\frac{s_v \kappa_{v}}{\delta}$.
\begin{prop}\label{lem.w1}
Let $\Phi\in C_c(\Ga\ba G)$.
As $T\to\infty$, we have
$$
\la F_{T},\Phi\ra\sim
  c_v \;  e^{\delta T} T^{(r_0-r)/2}     \;
  \mu_{H}^{\PS}(\tau_H)\, m^{\BR}(\Phi *\tau_K )
$$
where $c_v$ is given by
 \be \label{cv}
 c_v:=\frac{s_v \kappa_{v}}{\delta} \int_{\fb\cap\op{ker} \Theta}e^{-(I(w)+\delta |w|^2) /2} dw ;\ee
here $\kappa_v$ and $I(w)$ are as in Theorem \ref{mixing2} and in \eqref{ii} respectively, and $ s_v=\frac{1}{|\op{det} S_v|}$, where $S_v:\fa\to \fa$ is any linear map such that $S_v|_{\ker\Theta}=\op{Id}$ and $S_vv$ is a unit vector orthogonal to $\ker\Theta$
with respect to the inner product on $\fa$ induced by the Killing form.

\end{prop}
\begin{proof}
In view of the decomposition \eqref{eq.HBK}, we will need the following formula for the Haar measure $dg$ on $ G$; for all $\phi\in C_c(G)$,
\begin{equation}\label{eq.HIF}
\int_{G}\phi(g)\,dg=\sum_{w\in \cal W} \int_H \int_K \left(\int_{\fb^+}\phi(h w(\exp b)k)\xi(b)\,db\right)dk\, dh ,
\end{equation}
where $\xi : \fb\to\bb R$ is given by 
\begin{equation}\label{sinh}
\xi(b)=\prod_{\alpha\in\Sigma_\sigma^+}(\op{sinh}\alpha(b))^{\ell_\alpha^+}(\op{cosh}\alpha(b))^{\ell_\alpha^-};
\end{equation}
here $\ell_{\alpha}^{\pm}:=\dim(\fg_{\alpha}^{\pm})$, where each $\fg_{\alpha}^{\pm}$ is the $\pm 1$ eigenspace of the root space $\fg_{\alpha}$ with respect to the involution $\theta\sigma$ (cf. \cite{Sch}, \cite[p.18]{GOS}).

Substituting $b=\underline a(t,w)$ for $t\geq 0$ and $w\in \fb\cap\op{ker}\Theta$
gives $db=s_vt^{\frac{r_0-1}{2}}\,dt\,dw$. Now, 
\be\label{ssame} \langle F_{T},\Phi\rangle 
=\int_{K}  \tau_K(k) \int_{b\in \cal C_T} \left(\int_{[e] H}\Phi([h] (\exp b) k)\tau_H(h) \,d[h]\right)\xi(b)\,db\,dk.\ee
Hence \begin{align*}
&\langle F_{T},\Phi\rangle 
\notag \\ &=
\int_{\fb\cap\op{ker}\Theta}  \int_{t\in R_T(w)} t^{\frac{r_0-1}{2}}\xi(\underline a(t,w)) \left(\int_{[e] H}(\Phi*\tau_K) ([h]a(t,w))\tau_H(h) \,d[h]\right)s_v\,dt\,dw
\notag \\
&= s_ve^{\delta T} T^{(r_0-r)/2} \begin{multlined}[t]\int_{\fb \cap\op{ker}\Theta}   p_T(w)\,dw,\notag
\end{multlined}
\end{align*}
where we define $p_T(w)$ to be
\begin{multline}\label{defpt}
e^{-\delta T} T^{(r-r_0)/2}  \int_{R_T(w)} t^{\frac{r_0-1}{2}}\xi(\underline a(t,w))
\left(\int_{[e] H}(\Phi*\tau_K) ([h]a(t,w))\tau_H(h) \,d[h]\right)\,dt.
\end{multline}
We next look for an integrable function on $\fb\cap \op{ker}\Theta$ that bounds the family of functions $p_{T}(w)$ from above, in order to apply the dominated convergence theorem. By \eqref{sinh},
there exists a constant $c_1>0$ such that for all $(t,w)$ with $\underline a(t,w)\in\fc$,
\be\label{eee}
e^{-2\rho(\underline a(t,w))}\xi(\underline a(t,w))\le c_1;
\ee

By Proposition \ref{prop.mixH2}, we may assume
$$
\left|t^{\frac{r-1}{2}}e^{2\rho(\underline a(t,w))-{\delta t}}\int_{[e] H}(\Phi *\tau_K)([h]a(t,w)) \tau_H(h) \,d[h]\right| \le c_1
$$
as well.
Hence
$$
p_T(w)\leq c_1^2\psi(w),
$$
where $\psi(w)=\delta^{-1} e^{-c \delta |w|^2}$ is as given in Lemma \ref{dom}.
Since $\psi$ is integrable over $ \fb \cap\ker\Theta$,
we may apply Proposition \ref{eq} and the dominated convergence theorem to deduce that
\begin{align*}
&\int_K \tau_K(k)\int_{b\in \cal C_T} \left(\int_{[e] H}\Phi ([h](\exp b) k)\tau_H(h) \,d[h]\right)\xi(b)\,db\,dk\\
&\sim s_v\kappa_{v}  e^{\delta T} T^{(r_0-r)/2}  \cdot \mu_H^{\PS}(\tau_H)   \int_{w\in \fb\cap\op{ker}\Theta} {\delta}^{-1} e^{-{{\delta}{|w|}^2}/{2}}e^{-I(w)/2}m_{ }^{\BR}(\Phi*\tau_K)\,dw\\
&=c_v   e^{\delta T} T^{(r_0-r)/2}   \mu_H^{\PS}(\tau_H)\, m_{ }^{\BR}(\Phi*\tau_K),
\end{align*}
with $c_v$ as given in the statement of the proposition.
\end{proof}

We now fix a left $(H\cap M)$-invariant compact subset $\Omega_H\subset H$ and a right $M$-invariant
compact subset $\Omega_K\subset K$.
Let $\e>0$ be a number smaller than the injectivity radius at $[e]$ in $\Gamma\ba G$.
There exists a symmetric neighborhood $\Cal O_\e$ of $e$ in $G$ such that for all $T\ge 1$,
\begin{align}\label{inq}
\Om_{H, \e}^{-} \exp( \fc_{T,\e}^-)\Omega_{K, \e}^{-} &\subset \bigcap_{g\in\cal O_\e}        \Om_H \exp(\fc_T) \Om_K g ,\text{ and}\notag\\
  \Om_H \exp(\fc_T) \Om_K \Cal O_\e &\subset \Om_{H, \e}^{ +} \exp(\fc_{T,\e}^+) \Omega_{K,\e}^{+}, \end{align}
where $\cal C_{T,\e}^+=\cal C_T+\mathfrak b_\e$,  $\cal C_{T,\e}^-=\cap_{b\in\mathfrak b_\e}(\cal C_T+b)$, $\Omega_{K, \e}^+=\Omega_K K_\e$, $\Omega_{K, \e}^-=\bigcap_{k\in K_\e}\Omega_K k$, and $\Omega_{H, \e}^{ \pm}$
are defined similarly (see \cite{GOS}, \cite{GOS2}).
We will additionally fix  convex cones $\cal C^\flat,\cal C^\sharp\subset \fb^+ \cap (\mathrm{int}( \fa^+)\cup\lbrace 0\rbrace)$ such that $\cal C^\flat\subset\op{int}_\fb\cal C$,
 $\cal C\subset\op{int}_\fb\cal C^\sharp$, and which satisfy the condition \eqref{fcc}.
Note that $\cal C_{T,\e}^+$ is no longer contained in $\cal C$, but there exists $T_0>0$ (independent of $0<\e<1$) such that
\begin{equation}\label{eq.DC}
\cal C_{T,\e}^+-\cal C_{T_0,\e}^+\subset\cal C_{T+\e}^\sharp \quad \text{ and }\quad \cal C_{T-\e}^\flat-\cal C_{T_0}^\flat\subset  \cal C_{T,\e}^-.
\end{equation}

Choose a nonnegative function $\phi_\epsilon\in C_c(G)$ such that $\int_G\phi_\epsilon(g)\,dg=1$ and $\op{supp}(\phi_\epsilon)\subset \Cal O_\e$.
Define $\Phi_\epsilon : \Gamma\ba G\to\bb R$ by 
\begin{equation}\label{ppp}
\Phi_\epsilon([g]) =\sum_{\ga\in\Ga}\phi_\epsilon(\ga g).
\end{equation}

\begin{lem}\label{mz}\label{km2}
Let $\tau\in C(K)$ be left $M$-invariant. Then 
$$\lim_{\e\to 0} m^{\BR} (\Phi_\e* \tau ) =\int_K \tau(k^{-1})\, \mu^{\PS,*}_{K, \i(v)}(k).$$
where   $\mu^{\PS,*}_{K, \i(v)}:= \tilde \nu_{\Theta\circ \i}$ is given in \eqref{km}.
 \end{lem}

\begin{proof} Set $\tilde \nu:=\tilde \nu_{\Theta\circ \i}$.
We use Lemma \ref{lem.BRD} and write
\begin{align*}
&m_{ }^{\BR}(\Phi_\e* \tau)=\int_{K}\int_{G}\phi_\epsilon(gk)\tau(k)\,d\tilde m_{ }^{\BR}(g)\,dk\\
&=\int_{K}\int_{G} \phi_\epsilon(k'\exp(q)nk)\tau(k)\,e^{-\Theta (q)}\,dn\,dq\,d\tilde\nu(k')\,dk.
\end{align*}
Substituting $g=\exp(q)nk\in AN^+K$, the density of the Haar measure is given by
\begin{equation*}
dg=e^{-2\rho(q)}\,dn\,dq\,dk.
\end{equation*}
For $g\in G$, let $\kappa(g)$ denote the $K$-component of $g$, and $a_g$ denote the logarithm of $A$-component of $g$, in the decomposition $G=AN^+ K$.
Then
\begin{align*}
&m_{ }^{\BR}(\Phi_\e* \tau) =\int_{K}\int_{G} \phi_\epsilon(k'g) \tau(\kappa(g)) \,e^{(2\rho-\Theta)(a_g)}\,dg\,d\tilde\nu(k')\\
&=\int_{K}\int_{G}\phi_\epsilon(g) \tau(\kappa (k^{-1} g) ) \,e^{(2\rho-\Theta)(a_{k^{-1}g})}\,dg\,d\tilde \nu(k).
\end{align*}
By shrinking $\cal O_\e$ if necessary, we can assume that for all $k\in K$,
$$
\kappa(k^{-1}\mathcal O_\e)\subset k^{-1}K_{\e}.
$$
By the uniform continuity of $\tau$, there exist positive $\eta=\eta_{\epsilon}\rightarrow 0$ as $\epsilon\rightarrow 0$ such that for all $g\in\scrO_{\e}$ and $k\in K$,
$$\tau(k^{-1})-\eta \le\tau(\kappa(k^{-1}g))\le \tau(k^{-1})+\eta.$$
It follows from the fact that the multiplication map $A\times N^+\times K\to G$ is a diffeomorphism
that for some $C>1$, we have that for all $g\in\scrO_{\e}$ and $k\in K$,
$$1- C\e  \le e^{(2\rho-\Theta)(a_{k^{-1}g})}\le 1+ C \e .$$
Since $\int_G \phi_\e \,dg=1$,
we get   $$(1-C \e) \int_K (\tau(k^{-1})-\eta)\,d\tilde\nu(k)\le  
  m_{ }^{\BR}(\Phi_\e * \tau)\le  (1+ C \e ) \int_K ( \tau(k^{-1})+\eta)\,d\tilde\nu(k).$$
The claim now follows from letting $\epsilon\rightarrow 0$. 
\end{proof}

\begin{cor} Let $\mathcal C\subset \fb^+\cap (\inte \fa^+\cup\{0\})$ be a convex cone satisfying \eqref{fcc}.
\label{mmm} If $\mu^{\PS,*}_{K, \i(v)}(\partial \Omega_K^{-1})=\mu_H^{\PS}(\partial\Om_H)=0$, then
$$\lim_{T\to \infty}\frac{\# ( \Gamma\cap \Om_H  \exp(\scrC_T) \Om_K)}{ e^{\delta T} T^{(r_0-r)/2}   }= c_v \;   \mu_H^{\PS}(\Om_H) \mu^{\PS,*}_{K, \i(v)}(\Om_K^{-1}). $$
If $v=u_\Ga\in\fb^+$, then we may take $\scrC$ to be $\fb^+$.
\end{cor}\begin{proof} Write $\tilde\nu:= \mu^{\PS,*}_{K, \i(v)}$ for simplicity.
For $g\in H (B^+\cap \inte A^+) K$,
let $g=h_g(\exp b_g) k_g$ denote the $HB^+K$ decomposition of $g\in G$; note that $h_g(H\cap M)$, $b_g$ and
$M k_g$ are uniquely defined.

Since $\Om_H (H\cap M)=\Om_H$ and $M\Om_K=\Om_K$, we 
may define $X_T : G\to\bb R$ by
\begin{align*}
X_T(g)&=\mathbf{1}_{\Om_H}(h_g)\mathbf{1}_{\cal C_T} ( b_g)  \mathbf{1}_{\Om_K}( k_g)
\end{align*}
 for $g\in H \exp \cal C^\sharp K$, and $X_T(g)=0$ otherwise.
Let $\phi_{K,\e}\in C(K_\e)^M$ and $\phi_{H,\e}\in C(H_\e)^{H\cap M}$ be non-negative functions with integral one.
Set
\begin{align*}
\tau_{K,\e}^{\pm} &:=\mathbf{1}_{\Om_{K,2\e}^\pm}*\phi_{K,\e},\text{ }\tau_{H,\e}^{\pm} :=  \mathbf{1}_{\Om_{H,2\e}^\pm}*\phi_{H,\e} ,\text{ and }\\
F_{T+\e}^+&:=F_{\cal C_{T+\e}^\sharp,\tau_{H,3\e }^\pm,\tau_{K,3\e }^\pm},\text{ }F_{T-\e}^-:=F_{\cal C_{T-\e}^\flat,\tau_{H,3\e }^-,\tau_{K,3\e }^-}.
\end{align*}
Note that by definition, 
\begin{equation}\label{eq.b0}
\mathbf{1}_{\Om_{K,\e}^+}\leq \tau_{K,\e}^+ \leq\mathbf{1}_{\Om_{K,3\e}^+},\text{ and }\mathbf{1}_{\Om_{K,3\e}^-}\leq \tau_{K,\e}^- \leq\mathbf{1}_{\Om_{K,\e}^-}.
\end{equation}

By \eqref{inq}, \eqref{eq.DC} and \eqref{eq.b0}
there exists a uniform constant $C>0$ such that for all $g\in\cal O_\e$,
$$
F_{T-\e}^-([ g])-C \leq\sum_{\ga\in\Ga} X_T(\ga)\leq F_{T+\e}^+([ g]) +C.
$$

Integrating this against $\Phi_\e$ given in \eqref{ppp},
we get
$$
\langle F_{T-\e}^-, \Phi_\e\rangle -C \leq\sum_{\ga\in\Ga} X_T(\ga)\leq \langle F_{T+\e}^+, \Phi_\e\rangle +C.
$$
For simplicity, we set  $x_T:= \sum_{\ga\in\Ga} X_T(\ga)=\# \Gamma\cap \Om_H  \exp(\scrC_T) \Om_K$.
Hence by Proposition \ref{lem.w1}, fixing $\e_0>0$, for all $0<\e<\e_0$, we have
\begin{align}\label{eq.ub}
 \limsup\limits_{T\to\infty}\frac{x_T}{e^{\delta T} T^{(r_0-r)/2} }&\leq c_v \mu_H^{\PS}(\tau_{H,3\e}^+)m^{\BR}(\Phi_\e*\tau_{K,3\e}^+)\notag\\
&\leq c_v \mu_H^{\PS}(\tau_{H,3\e}^+)m^{\BR}(\Phi_\e*\tau_{K,9\e_0}^+).
\end{align}

On the other hand, since $\tilde\nu(\partial\Om_K^{-1})=\mu_H^{\PS}(\partial\Om_H)=0$, we get from \eqref{eq.b0} that
\begin{equation}\label{eq.b1}
\lim_{\e\rightarrow 0}\mu_{H}^{\PS}(\tau_{H,3\e }^\pm)=\mu_{H}^{\PS}(\Om_H)\quad\text{and}\quad \lim_{\e \rightarrow 0}\int_K \tau_{K,\e }^\pm(k^{-1})\,d\tilde\nu(k)=
\tilde\nu(\Om_K^{-1}).
\end{equation}
Now letting $\e\to0$ first using Lemma \ref{mz} and \eqref{eq.b1}, and then letting $\e_0\to 0$, we get
 $$
 \limsup_{T\to\infty}\frac{x_T}{e^{\delta T} T^{(r_0-r)/2} }\leq c_v \mu_{H}^{\PS}(\Om_H)\tilde\nu(\Om_K^{-1}).
 $$
 Similarly, we can show that
 $$
 \liminf_{T\to \infty} \frac{x_T}{e^{\delta T} T^{(r_0-r)/2} }\geq c_v \mu_{H}^{\PS}(\Om_H)\tilde\nu(\Om_K^{-1}),
 $$
 which proves the corollary.
\end{proof}

 \begin{Rmk}\rm
Note that this corollary implies that the asymptotic of $\# ( \Gamma\cap \Om_H  \exp(\scrC_T) \Om_K)$ is independent of $\mathcal C$. 
\end{Rmk}

\noindent{\bf Proof of Theorem \ref{m3}.}
Theorem \ref{m3} now follows directly from applying Corollary \ref{mmm} and the following observation to a convex cone $\scrC$ such that $\fb\cap \scrL_\Ga\subset \inte_\fb\scrC$; by Lemma \ref{large}, such a cone always exists.

\begin{lem}
Suppose that $\fb^+\cap\scrL_\Ga\subset \inte_\fb \scrC$.
Then $$\#( \G\cap  \Om_H\exp( \mathfrak b^+ -\fc) \Om_K)<\infty . $$ 
\end{lem}
\begin{proof}
We can find a smaller closed convex cone $\cal C'\subset\op{int}_{\mathfrak b}\cal C$ such that $\mathfrak b^+\cap\cal L_\Ga\subset\op{int}_{\mathfrak b}\cal C'$.
Set $\mathcal Q:=\fb^+-\cal C$, $\mathcal Q':=\fb^+-\cal C'$, and $\Om_H':=\Om_H H_1$ where $H_1$ means the unit neighborhood of $e$ in $H$.
We  can find a bi-$K$-invariant neighborhood $\scrO\subset G$ of $e$ such that for all $g\in \cal O$,
$$
\Om_H\exp(\cal Q_T) K\subset \Om_H'\exp(\cal Q_{T+1}') Kg^{-1}.
$$

 Define $G_T([g])=\sum_{\gamma\in  \Gamma}\mathbf{1}_{ \Om_H'\exp(\mathcal Q_T') K} ( \gamma g)$.
 Now,
$$\#( \G\cap  \Om_H\exp( \mathfrak b_T^+ -\fc) \Om_K)= \sum_{\gamma\in  \Gamma}\mathbf{1}_{ \Om_H \exp(\mathcal Q_T ) K} ( \gamma ) \le \la G_{T+1}, \Phi\ra,$$
where $\Phi$ is a non-negative $K$-invariant
continuous function supported in
$ [e]\cal O$ and $\int_{\Gamma\ba G} \Phi\, dg=1$.
Now note that
$$ \langle G_T,\Phi\rangle 
= \int_{b\in  \mathcal Q_T'} \left(\int_{[e] \Om_H'}\Phi([h] (\exp b) )\,d[h]\right)\xi(b)\,db. $$
Recalling $\Om_H'$ is $H\cap M$-invariant,
 let $0\le \psi\le 1$ be an element of $C_c(H)^{H\cap M}$ which is one on $\Om_H'$.
For $\e_0>0$ smaller than the injectivity radius of $\text{supp}(\psi)$, let $\Psi=\psi\otimes \rho_{\e_0}$ be given as \eqref{po}.
Let $\Phi_{\e_0}^+(y)=\sup_{g\in N_{\e_0} A_{\e_0} N_{\e_0}^+M} \Phi(yg)$ be as in \eqref{ff}. 
Since the closure of $\cal Q'$ is disjoint from
$\fb^+\cap \L_\Gamma$,
by Proposition \ref{vanish}, there exists $T_0>0$ such that
 $\la (\exp b) \Phi_{\e_0}^+, \Psi\ra=0$ for all $b\in \cal Q'-\cal Q_{T_0}'$.
 Then by the same argument as in the proof of Proposition \ref{prop.mixH2}, we get that for all $\tilde b\in \cal Q'-\cal Q_{T_0}'$,
\begin{align*}  \int_{[e] H}\Phi([h] (\exp \tilde  b) ) \psi [h] dh
&= \int_{[e] H(NB)_{\e_0}} \Phi([h] (\exp \tilde b) ) \Psi ([h]n\tilde b) dn\,d\tilde b\, dh
\\ & \le  \la (\exp \tilde b) \Phi_{\e_0}^+, \Psi\ra=0. \end{align*}
Hence
\begin{align}\label{min} &\int_{\cal Q'}  \int_{[e] H}\Phi([h] (\exp b) ) \psi [h] dh  \xi(b)\,db
 =\int_{\cal Q_{T_0}'}  \la (\exp b) \Phi_{\e_0}^+, \Psi\ra  \xi(b)\,db \\ &
 \le \|\Phi_{\e_0}^+\|_2 \|\Psi\|_2 \text{Vol} ({\cal Q_{T_0}'} ).\notag \end{align}
 
 Hence  for all $T>T_0$, 
 $$\#( \G\cap  \Om_H\exp( \mathfrak b_T^+ -\fc) \Om_K) \le \la G_{T+1}, \Phi\ra \le  \|\Phi_{\e_0}^+\|_2 \|\Psi\|_2 \text{Vol} ({\cal Q_{T_0}'} ).$$
 This implies the claim in view of Corollary \ref{mmm}.
\end{proof}

For $\omega\in\cal W$, set $\Gamma^\omega=\omega^{-1}\Gamma \omega$,  $H^\omega=\omega^{-1} H \omega$, and $\Om_{H^\omega}= \omega^{-1}\Omega_H\omega \subset H^{\omega}$.
Then $$\# ( \Gamma \cap \Omega_H \omega\exp(\cal C_T)\Om_K) =\# ( \Gamma^\omega \cap \Omega_{H^{\omega}} \exp(\cal C_T) \Om_K \omega). $$
Since $\omega\in K$, it follows that $\psi_\Gamma =\psi_{\Gamma^\omega}$, and that the involution which stabilizes $H^\omega$ commutes with $\theta$.
 Hence
 $$\delta=\psi_{\Gamma^\omega} (v)=\max_{b\in \fb, \|b\|=1} \psi_{\Gamma^\omega} (b).$$
 
By applying Corollary \ref{mmm} to $\Gamma^\omega $ and $H^\omega$ for each $\omega\in \cal W$, we can also deduce the asymptotic of
$\# \big( \Gamma\cap \Om_H \cal W\exp (\cal C_T) \Om_K\big)$.

\medskip

Theorem \ref{thm3} follows from the following: we set $\op{sk}_{\Gamma,v} (H) = |\mu_{[e]H,v}^{\PS}|$.
\begin{thm}\label{thm33} Suppose that $\mathsf v_0\Gamma\subset H\ba G$ is discrete for $\vo=[H]$ and that $[e]H$ is uniformly proper. Then 
 $\op{sk}_{\Gamma,v} (H)<\infty$ and there exists $c_v>0$  such that
\be \label{skin}\lim_{T\to \infty} \frac{\# (\vo \Gamma\cap \vo  \exp (\fb^+\cap \L_\Ga)_T  K )}{    e^{\psi_\Ga(v) T} T^{(r_0-r)/2}}  =c_v\;
\op{sk}_{\Gamma,v} (H) 
   \ee
 Moreover,  for  $v=u_\Ga $, we have
\begin{equation*}
    \lim_{T\to \infty}\frac{ \# (\vo \Gamma\cap \vo  (\exp \fb^+_T)  K ) }{ e^{\psi_\Ga(u_\Ga) T} T^{(r_0-r)/2}   }=c_{u_\Ga}\; \op{sk}_{\Gamma,u_\Ga} (H) 
\end{equation*}
where $\fb^+_T= \{ w \in \fb^+ :\|w\| \le T\} $.
\end{thm}
\begin{proof}
Set $\cal C:=\fb^+\cap \L_\Ga$.
By Theorem \ref{int} (1) and Lemma \ref{large}, $\cal C$ satisfies the conditions in \eqref{fcc}.

By hypothesis, there exists a $K$-invariant open neighborhood $\cal O\subset \G\ba G$ of $[e]$ 
such that
$Y_0:=\{[e]h\in \Gamma\ba \Gamma H: [e]h \exp \Cal C \cap \cal O \ne \emptyset\}$
is bounded.
Set $$ F_T([g]):=\sum_{\gamma\in (\Ga\cap H)\ba\Ga}\mathbf{1}_{\vo \exp\cal C_T K}(\vo \gamma g).$$

Let $\cal O_\e\subset G$ and $\Phi_\e$ be as in \eqref{ppp}. We may assume $\Phi_\e$ is $K$-invariant as our functions $F_T$ are $K$-invariant in deducing the following:
  \be\label{ooo} \langle F_{T-\e},\Phi_\e\rangle \le F_T([e])\le \langle F_{T+\e},\Phi_\e\rangle  .\ee
Observe that
$$\langle F_{T\pm \e},\Phi_\e\rangle 
=\int_{K}   \int_{b\in \cal C_{T\pm \e}} \left(\int_{[e]H}\Phi_\e([h] (\exp b k) ) \,d[h]\right) \, \xi(b)\,db\,dk.$$

 For $S>0$, let $\tau_S\in C_c([e]H)$ be a function satisfying
 $0\le \tau \le 1$ and $\tau_S=1$ on the $S$-neighborhood of $Y_0$. 
By the definition of $Y_0$, it follows that
$$\langle F_{T\pm \e},\Phi_\e\rangle =\int_K
\int_{b\in \cal C_{T\pm \e}}\left(\int_{[e]H}\Phi_\e([h] (\exp b k) )\tau_S([h]) \,d[h]\right)\, \xi(b)\,db dk.$$

Note that this integral is same as the one in \eqref{ssame}, as $\tau_S$ is compactly supported. Hence
we get
$$\lim_{T\to \infty}\frac{\langle F_{T\pm \e},\Phi_\e\rangle}{ e^{\delta (T\pm \e)} T^{(r_0-r)/2}}=   c_v  \mu_{[e]H}^{\PS}(\tau_S) m^{\BR}(\Phi_\e).$$
By sending $\e\to 0$ and applying Lemma \ref{mz}, we get from \eqref{ooo}
$$ \lim_{T\to \infty} \frac{  F_{T}([e]) }{  e^{\delta T}T^{(r_0-r)/2}}=c_v  \mu_{[e]H}^{\PS} (\tau_S).$$

It follows that $\mu_{[e]H}^{\PS} (\tau_S)$ is a constant function of $S>0$, and hence
$$|\mu_{[e]H}^{\PS}|=\mu_{[e]H}^{\PS} (\tau_S)<\infty.$$
This proves the first claim.
The second claim follows from the first one in view of the following lemma by taking $\cal C=\fb^+\cap \L_\Ga$. \end{proof}
\begin{lem}
Let $\cal C\subset \fb^+$ be a convex cone with $v=u_\Ga \in \inte_\fb \scrC$. Set $Q:=\fb^+ -\cal C$.
Then there exist $0<\delta' <\delta_\Ga=\psi_\Gamma (u_\Gamma)$ and $C>0$ such that for all $T\ge 1$,
$$  {\#( \G\cap  \Om_H\exp(\cal Q_T) \Om_K)}\le C e^{\delta'  T} .$$  
\end{lem}
\begin{proof}
Choose a closed convex cone $\cal C'\subset\op{int}_{\mathfrak b}\cal C$ such that $u_\Ga\in \op{int}_{\mathfrak b}\cal C'$ and
set $\mathcal Q':=\fb^+-\cal C'$, and $\Om_H':=\Om_H H_1$ where $H_1$ means the unit neighborhood of $e$ in $H$.
We  can find a bi-$K$-invariant neighborhood $\scrO\subset G$ of $e$ such that for all $g\in \cal O$,
$$
\Om_H\exp(\cal Q_T) K\subset \Om_H'\exp(\cal Q_{T+1}') Kg^{-1}.
$$

 Define $G_T([g])=\sum_{\gamma\in  \Gamma}\mathbf{1}_{ \Om_H'\exp(\mathcal Q_T') K} ( \gamma g)$.
 Now,
\be\label{now}
\#( \G\cap  \Om_H\exp( \cal Q_T)\Om_K)= \sum_{\gamma\in  \Gamma}\mathbf{1}_{ \Om_H \exp(\mathcal Q_T ) K} ( \gamma ) \le \la G_{T+1}, \Phi\ra,\ee
where $\Phi$ is a non-negative $K$-invariant
continuous function supported in
$ [e]\cal O$ and $\int_{\Gamma\ba G} \Phi\, dg=1$.
Now note that
$$ \langle G_T,\Phi\rangle 
= \int_{b\in  \mathcal Q_T'} \left(\int_{[e] \Om_H'}\Phi([h] (\exp b) )\,d[h]\right)\xi(b)\,db. $$



If we write $b=tu_\Gamma+\sqrt t w\in \cal Q_T'$ with $w\in \ker \Theta$, then $t^2+t\|w\|^2\le T^2$ and
$\|w\|^2\ge \tan^2\theta_0 \cdot t$ for some $0<\theta_0<\pi/2$ depending on the distance between
$u_\Gamma$ and $\cal Q'$ (cf. proof of Lemma \ref{dom}). Hence  if $b=tu_\Gamma+\sqrt t w\in \cal Q_T'$, 
then \be \label{the} 0\le t\le T \cdot \cos \theta_0 .\ee
Fix a non-negative function $\phi\in C_c(H)$ which is $1$ on $\Omega'_H$. Since \eqref{qqq} gives that $$\left| t^{(r-1)/2} e^{(2\rho-\Theta)(\underline a(t,w))} \int_{H} \Phi ([h]a(t,w)) \phi (h)\,dh\right| <C', $$
it follows that for all $t\geq 1$,
\begin{equation}\label{eq.p1}
\left|\int_{H} \Phi ([h]a(t,w)) \phi (h)\,dh\right|<C'e^{(\Theta-2\rho)(\underline a(t,w))}.
\end{equation}
Using \eqref{eee}, \eqref{the}, and $\Theta(u_\Gamma)=\delta_\Ga$, we deduce that
\begin{align*}
&\int_{b\in  \mathcal Q_T'} \left(\int_{H} \Phi ([h]a(t,w)) \phi (h)\,dh \right)  \xi(b)\,db \\ &\ll \int_{b\in  \mathcal Q_T'} e^{\delta_\Ga t}\,db\leq e^{(\delta_\Ga \cos\theta_0) T}\op{Vol}(\mathfrak b_T).
\end{align*}
As $\op{Vol}(\mathfrak b_T)=O(T^{r_0})$, for any $\delta'$ satisfying $\delta_\Ga \cos \theta_0 <\delta ' <\delta_\Ga$, we get
$$  \int_{b\in  \mathcal Q_T'} \left(\int_{H} \Phi ([h]a(t,w)) \phi (h)\,dh \right)  \xi(b)\,db \ll  e^{\delta'  T}  .$$
This proves the lemma by \eqref{now}.
\end{proof}

\medskip

Finally we give examples satisfying the hypothesis of Theorem \ref{thm33}:
\begin{lem} \label{bd} Suppose that $\fa=\fb$ and $\Lambda_\Gamma\subset HP/P$.
\begin{enumerate}
\item The orbit $[e]H$ is uniformly proper. 
\item The support of $\mu_{[e]H}^{\PS}$ is a compact subset of $(H\cap \Gamma)\ba H$ and $|\mu_{[e]H}^{\PS}|>0 $.
\end{enumerate}

\end{lem}

\begin{proof}
The condition $\fa=\fb$ means $H\cap A=\{e\}$, and hence  $H\cap P$ is compact by \eqref{hcp}. Since $\Lambda_\G$ is a compact subset of $HP/P\simeq H/(H\cap P)$,
it follows that $\Lambda_\G\subset H_0P/P$ for some compact subset $H_0\subset H$.
 
 To show (1), let $\cal C$ be a closed cone contained in $ \inte \fa^+\cup\{0\}$.
It suffices to show that for any given compact subset $Z\subset \Gamma\ba G$,
$$\{[e]h\in \Gamma\ba \Gamma H: [e]h \exp \Cal C \cap Z\ne \emptyset\}$$ is bounded.
Suppose not; then  there exist $h_i\in H$ with $[e]h_i\to \infty$ in $\Gamma\ba \Gamma H$, $\gamma_i\in \Gamma$ and $c_i\in \exp \cal C$
 such that $\gamma_i h_i c_i$ converges to some $g\in G$. Since $[e]H\subset \Gamma\ba G$ is a closed subset, it follows that $c_i\to \infty$.
 We may write $h_ic_i = \gamma_i^{-1} gg_i $ where $g_i\to e$. Set $o=[K]\in G/K$.
 By passing to a subsequence, we may assume that $\gamma_i^{-1} gg_i o $ converges to some $\xi\in \F$ in the sense of \cite[Def.2.7]{LO}; note here that
 $\mu(\gamma_i^{-1} g g_i) \to\infty$ regularly in $\fa^+$ as $\L_\Gamma\subset \inte \fa^+$.
 Since $gg_i$ is bounded,  it follows that $\xi\in \La_\G$ by \cite[Lemma 2.12]{LO}. Since $h_ic_i \eta =\gamma_i^{-1}g g_i \eta \to \xi$ for all $\eta\in \F$ except for points on a proper submanifold
 of $\F$, we may choose $\eta=ne^+ \in N^+e^+$ so that $\lim_{i\to \infty} h_ic_i n e^+=\xi$.
 Writing $\xi=h_0e^+$ for some $h_0\in H_0$, we have  $\lim_{i\to \infty} h_ic_in c_i^{-1} P = h_0P$. Since $c_i\to \infty$ in $\exp \cal C$ and $\cal C\subset \inte \fa^+\cup \{0\}$, we have $\lim_{i\to \infty} c_inc_i^{-1}=e$. As $HP$ is open, we may write $c_inc_i^{-1}= h_i' p_i\in HP$ with $h_i'\in H, p_i\in P$ both tending to $e$.
 It follows that $h_i h_i' P\to h_0 P$ as $i\to \infty$. Since $H\cap P$ is compact, and  the sequence $h_i'$ is bounded, it follows that 
 the sequence $h_i$ is bounded, yielding a contradiction. This proves Claim (1).

Since $\text{supp} (\mu_{[e]H}^{\PS})=\{[e]h\in [e]H: he^+ \in  \La_\G\}\subset \{[e]h_0: h_0\in H_0(H\cap P)\}$, the first part of Claim (2) follows, and the second part follows from from \eqref{hpp}.
 \end{proof}

We remark that this lemma holds when $H$ is replaced by $N^+$ by the same proof.
Moreover, by replacing Proposition \ref{eq} by Proposition \ref{m11} and considering the Iwasawa decomposition $G=N(\exp \fa ) K$, 
the proofs of Theorems \ref{m3} and \ref{thm33} apply for $H=N^+$ and $\fb^+=\fa^+$.

\begin{Ex}\label{man}
\rm
Let $G=\op{PSL}_2(\bb C)\times \op{PSL}_2(\bb C)$ and $H=\op{PSL}_2(\bb R)\times \op{PSL}_2(\bb R)$.
Then $\mathfrak a=\mathfrak b$ and $G/P$ can be identified with $\bb S^2\times\bb S^2$ and $HP/P$ can be identified with $D\times D$ where $D$ is the open northern hemisphere of $\bb S^2$.
Let $\Sigma$ be a non-elementary finitely generated discrete subgroup of $\op{PSL}_2(\bb R)$.
For $i=1,2$, consider a convex cocompact representation $\pi_i : \Sigma\to \op{\PSL}_2(\bb C)$ such that $\La_{\pi_i(\Sigma)}\subset D$ and that
$\{(\pi_1(s),\pi_2(s)):s\in\Sigma\}$ is Zariski dense in $G$. Then it is easy to see that $\Ga:=\{(\pi_1(s),\pi_2(s)):s\in\Sigma\}$ is an Anosov subgroup with $\La_\Ga\subset D\times D$.
Hence 
$\Gamma$ satisfies the hypothesis of Lemma \ref{bd}.
\end{Ex}

\end{document}